\documentclass[letterpaper, 11pt,  reqno]{amsart}

\usepackage{amsmath,amssymb,amscd,amsthm,amsxtra, esint}
\usepackage{tikz} 
\usepackage{color}
\usepackage[implicit=true]{hyperref}

\usepackage{cases}%%%

\usepackage[left=32mm, right=32mm, 
bottom=27mm]{geometry}

\allowdisplaybreaks[2]

\usepackage{color}

\definecolor{gr}{rgb}   {0.,   0.69,   0.23 }
\definecolor{bl}{rgb}   {0.,   0.5,   1. }
\definecolor{mg}{rgb}   {0.85,  0.,    0.85}
\definecolor{yl}{rgb}   {0.8,  0.7,   0.}
\definecolor{or}{rgb}  {0.7,0.2,0.2}

\usetikzlibrary{shapes.misc}
\usetikzlibrary{shapes.symbols}
\usetikzlibrary{shapes.geometric}

\tikzset{
	ddot/.style={circle,fill=white,draw=black,inner sep=0pt,minimum size=0.8mm},
	>=stealth,
	}

\tikzset{
	ddot2/.style={circle,fill=black,draw=black,inner sep=0pt,minimum size=0.8mm},
	>=stealth,
	}

\newtheorem{theorem}{Theorem} [section]

\newtheorem{lemma}[theorem]{Lemma}
\newtheorem{proposition}[theorem]{Proposition}
\newtheorem{remark}[theorem]{Remark}

\newtheorem{definition}[theorem]{Definition}
\newtheorem{corollary}[theorem]{Corollary}

\DeclareMathOperator*{\supp}{supp}

\newcommand{\1}{\hspace{0.5mm}\text{I}\hspace{0.5mm}}
\newcommand{\II}{\text{I \hspace{-2.8mm} I} }
\newcommand{\III}{\text{I \hspace{-2.9mm} I \hspace{-2.9mm} I}}

\newcommand{\IV}{\text{I \hspace{-2.9mm} V}}

\newcommand{\noi}{\noindent}
\newcommand{\Z}{\mathbb{Z}}
\newcommand{\R}{\mathbb{R}}

\newcommand{\T}{\mathbb{T}}

\newcommand{\deff}{\stackrel{\textup{def}}{=}}
\newcommand{\U}{\Theta}
\newcommand{\B}{\mathbf{B}}
\newcommand{\cc}{\mathbf{q}}

\let\Re=\undefined\DeclareMathOperator*{\Re}{Re}
\let\Im=\undefined\DeclareMathOperator*{\Im}{Im}

\let\P= \undefined
\newcommand{\P}{\mathbf{P}}
\newcommand{\PP}{\mathbb{P}}
\newcommand{\Pc}{\mathcal{P}}

\DeclareMathOperator*{\Law}{Law}
\DeclareMathOperator*{\Var}{Var}

\newcommand{\Q}{\mathbf{Q}}

\newcommand{\E}{\mathbb{E}}
\newcommand{\En}{\mathcal{E}}

\renewcommand{\L}{\mathcal{L}}

\newcommand{\F}{\mathcal{F}}

\newcommand{\al}{\alpha}
\newcommand{\be}{\beta}
\newcommand{\dl}{\delta}

\newcommand{\nb}{\nabla}

\newcommand{\too}{\longrightarrow}

\newcommand{\Dl}{\Delta}
\newcommand{\eps}{\varepsilon}

\newcommand{\g}{\gamma}
\newcommand{\G}{\Gamma}
\newcommand{\ld}{\lambda}
\newcommand{\Ld}{\Lambda}
\newcommand{\s}{\sigma}

\newcommand{\ft}{\widehat}

\newcommand{\wt}{\widetilde}
\newcommand{\cj}{\overline}

\newcommand{\dt}{\partial_t}
\newcommand{\dd}{\partial}

\newcommand{\ta}{\theta}

\renewcommand{\l}{\ell}
\newcommand{\om}{\omega}
\renewcommand{\O}{\Omega}

\newcommand{\les}{\lesssim}
\newcommand{\ges}{\gtrsim}

\newcommand{\jb}[1]
{\langle #1 \rangle}

\newcommand{\ind}{\mathbf 1}

\newcommand{\M}{\mathcal{M}}

\newcommand{\N}{\mathbb{N}}

\newcommand{\X}{\mathcal{X}}
\newcommand{\Y}{\mathcal{Y}}
\newcommand{\NN}{\mathcal{N}}
\newcommand{\ZZ}{\mathcal{Z}}

\renewcommand{\H}{\mathcal{H}}
\newcommand{\D}{\mathcal{D}}

\newcommand{\Id}{\textup{Id}}

\newcommand{\Ph}{\Psi^\textup{heat}_N}
\newcommand{\Pw}{\Psi^\textup{wave}_N}
\newcommand{\sh}{\s^\textup{heat}_N}
\newcommand{\sw}{\s^\textup{wave}_N}

\newcommand{\rh}{\rho_{\textup{heat}}}
\newcommand{\rhN}{\rho_{\textup{heat}, N}}
\newcommand{\rw}{\rho_{\textup{wave}}}
\newcommand{\rwN}{\rho_{\textup{wave}, N}}

\newcommand{\Eh}{E_{\textup{heat}}}
\newcommand{\Ew}{E_{\textup{wave}}}

\newcommand{\bw}{\beta_{\textup{wave}}}
\newcommand{\bh}{\beta_{\textup{heat}}}

\newcommand{\BB}{\mathcal{B}}

\newtheorem*{ackno}{Acknowledgements}

\numberwithin{equation}{section}
\numberwithin{theorem}{section}

\begin{document}
\baselineskip = 15pt

\title[Parabolic and hyperbolic Liouville equations]
{On the parabolic and
hyperbolic Liouville equations}

\author[T.~Oh, T.~Robert, and Y.~Wang]
{Tadahiro Oh, Tristan Robert, and Yuzhao Wang}

\address{
Tadahiro Oh, School of Mathematics\\
The University of Edinburgh\\
and The Maxwell Institute for the Mathematical Sciences\\
James Clerk Maxwell Building\\
The King's Buildings\\
Peter Guthrie Tait Road\\
Edinburgh\\ 
EH9 3FD\\
 United Kingdom}

\email{hiro.oh@ed.ac.uk}

\address{
Tristan Robert, University of Rennes\\
 CNRS, IRMAR - UMR 6625\\ 
 F-35000, Rennes\\
France}

\email{tristan.robert@ens-rennes.fr}

\address{
Yuzhao Wang\\
School of Mathematics, 
University of Birmingham, 
Watson Building, 
Edgbaston, 
Birmingham\\
B15 2TT, 
United Kingdom}

\email{y.wang.14@bham.ac.uk}



\begin{abstract}
We study the two-dimensional
stochastic nonlinear heat equation (SNLH) 
and stochastic damped nonlinear  wave equation (SdNLW)
with an exponential nonlinearity $\ld\be  e^{\be u }$, forced by an additive space-time white noise.
(i)~We first study 
SNLH for general $\ld \in \R$.
By establishing 
higher moment bounds of 
the relevant Gaussian multiplicative chaos
%via
%the Brascamp-Lieb inequality
and exploiting the positivity of the Gaussian multiplicative chaos,  
we  prove local well-posedness of SNLH
for the range 
$0 < \be^2 < \frac{8 \pi}{3 + 2 \sqrt 2} \simeq 1.37 \pi$.
Our argument yields stability under the noise
perturbation, 
thus improving 
 Garban's local well-posedness result (2020).
(ii)~In the defocusing case $\ld >0$, 
we exploit a certain sign-definite  structure in the equation
and the positivity of the Gaussian multiplicative chaos.
This allows us to 
 prove global well-posedness of SNLH
for the range: $0 < \be^2 < 4\pi$.
(iii) 
As for SdNLW in the defocusing case $\ld > 0$, 
we go beyond the Da Prato-Debussche argument and introduce a decomposition of the nonlinear component, 
  allowing us to recover
 a sign-definite structure
for a rough part of the unknown,
while the other part enjoys a stronger smoothing property.
As a result, we reduce SdNLW into a system
of equations (as in the paracontrolled approach  for the dynamical $\Phi^4_3$-model)
and prove local well-posedness of SdNLW 
for the  range: $0 < \be^2 < \frac{32 - 16\sqrt3}{5}\pi \simeq 0.86\pi$.
This result 
(translated to the context of random data well-posedness
for the deterministic nonlinear wave equation with an exponential nonlinearity) 
solves  an open question posed by Sun and Tzvetkov (2020).
(iv)~When $\ld > 0$, 
these models formally preserve the associated Gibbs measures
with the exponential nonlinearity.
Under the same assumption on  $\be$ as in (ii) and (iii) above, we prove almost sure global well-posedness
(in particular for SdNLW)
and invariance of the Gibbs measures
in both the parabolic and hyperbolic settings.
(v) In Appendix, 
we present an argument
for proving local well-posedness of SNLH
for general $\ld \in \R$ {\it without} using
the positivity of the Gaussian multiplicative chaos.
This proves local well-posedness of SNLH
for the  range 
$0 < \be^2 < \frac 43 \pi \simeq 1.33 \pi$, 
slightly smaller
than that in (i),
but provides Lipschitz continuity 
of the solution map 
in initial data as well as the noise.

\end{abstract}

%\date{\today}

%\vspace*{-1cm}

\maketitle
%

%\vspace*{-1.7cm}

\tableofcontents

\baselineskip = 14pt

\section{Introduction}\label{SEC:1}

\subsection{Parabolic and hyperbolic Liouville equations}

We study the two-dimensional stochastic heat 
and wave equations with exponential nonlinearities, 
driven by an additive space-time white noise forcing.
More precisely, we consider the following stochastic nonlinear heat equations
(SNLH) on the two-dimensional torus
 $\T^2 = (\R/2\pi\Z)^2$:
\begin{align}
\begin{cases}
\dt u + \frac 12 (1- \Dl)  u   +  \frac 12 \ld \be e^{\be u} = \xi\\
u |_{t = 0} = u_0, 
\end{cases}
\quad ( t, x) \in \R_+ \times \T^2
\label{NH1}
\end{align}

\noi
and stochastic damped nonlinear wave equations (SdNLW) on $\T^2$:
\begin{align}
\begin{cases}
\dt^2 u + \dt u +  (1- \Dl)  u   +  \ld \be e^{\be u} = \sqrt{2} \xi\\
(u, \dt u) |_{t = 0} = (u_0, u_1) , 
\end{cases}
\quad ( t, x) \in \R_+ \times \T^2, 
\label{NW1}
\end{align}

\noi
where $\be, \ld \in \R \setminus \{0\}$
and  $\xi $ denotes a space-time white noise on $\R_+ \times \T^2$.
Our main goal is to establish local and global well-posedness
of these equations for certain ranges of the parameter $\be^2 >0$
and also prove invariance of the associated Gibbs measures in the defocusing case $\ld >0 $.
As we see below, 
due to the exponential  nonlinearity, the difficulty of these equations
depends
 sensitively on the value of $\be^2 > 0$
 as well as the sign of $\ld$.

Our study is  motivated by a number of perspectives. 
From the viewpoint of analysis on singular stochastic PDEs, 
the equations  \eqref{NH1} and \eqref{NW1} on $\T^2$
are very interesting models.
The main sources of the difficulty of these equations
come from the roughness of the space-time white noise forcing
and the non-polynomial nature of the nonlinearity.
The first difficulty can  already be seen at the 
level of the associated linear equations 
whose solutions (namely, stochastic convolutions) are known to be merely  distributions
for the spatial dimension $d \geq 2$.
This requires us to introduce
a proper  {\it renormalization}, adapted to the exponential nonlinearity, 
to give a precise meaning to the equations.
In recent years, we have seen a tremendous development
in the study of singular stochastic PDEs, in particular in the parabolic setting
\cite{DPD, Hairer0, Hairer, GIP, CC, Kupi, MW1, 
HS, CHS, Garban}.
Over the last few years, we have also witnessed a rapid progress
in the theoretical understanding of
 nonlinear wave equations with singular stochastic forcing
 and/or rough random initial data \cite{OTh2, GKO, GKO2, GKOT,  
 Deya1, Deya2, OPTz,ORTz, OOR, OOTz, STz, ORSW, ORSW2, OOTol, Bring2, OOTol2}.
On the two-dimensional torus $\T^2$, 
the stochastic heat and wave equations 
with a monomial  nonlinearity $u^k$ (see~\eqref{NH2} and~\eqref{NW2} below)
have been studied in \cite{DPD, GKO, GKOT}.
In particular, in the seminal work \cite{DPD}, 
Da Prato and Debussche introduced
the so-called Da Prato-Debussche trick\footnote{See
also the work by McKean \cite{McKean} and Bourgain  \cite{BO96}.} 
(see Subsection~\ref{SUBSEC:1.3})
which set a new standard in 
the study of  singular stochastic PDEs.
We point out that 
many of the known results focus on polynomial nonlinearities
and thus it is of great interest to extend
the existing solution theory to the case of non-polynomial nonlinearities.
We will come back and elaborate further
this viewpoint later. 
Furthermore, 
in this paper, we study 
both SNLH \eqref{NH1} and SdNLW \eqref{NW1}, 
which allows us to point out
similarity and difference
 between
the analysis of the stochastic heat and wave equations.
See also \cite{OO} for a comparison
of the stochastic heat and wave equations
on $\T^2$ with a quadratic nonlinearity 
driven by fractional derivatives
of a space-time white noise.

Another important point of view comes
from mathematical physics.
It is well known that many of singular stochastic PDEs
studied in the references mentioned above
correspond to parabolic and hyperbolic\footnote{This is the so-called
``canonical'' stochastic quantization equation.  See \cite{RSS}.} 
stochastic quantization equations
for various models arising in 
Euclidean quantum field theory;
namely, the resulting dynamics preserves
a certain Gibbs measure on an infinite-dimensional state space
of distributions.
See \cite{PW, RSS}.
For example, 
the well-posedness results in \cite{DPD, GKO, GKOT}
show that, for an odd integer $k \geq 3$,  
the $\Phi^{k+1}_2$-measure\footnote{In the hyperbolic case, 
it is coupled with the white noise measure $\mu_0$ on the $\dt u$-component.
See \eqref{Gibbs2}.}
is invariant under the dynamics of 
the parabolic $\Phi^{k+1}_2$-model on $\T^2$:
\begin{align}
\dt u + \tfrac 12 (1- \Dl)  u   +  u^{k} = \xi
\label{NH2}
\end{align}

\noi
and the hyperbolic $\Phi^{k+1}_2$-model on $\T^2$:
\begin{align}
\dt^2 u + \dt u +  (1- \Dl)  u   +  u^k = \sqrt{2} \xi, 
\label{NW2}
\end{align}

\noi
respectively.
From this point of view, 
when $\ld >0$, 
the equations \eqref{NH1} and \eqref{NW1}
correspond to 
the parabolic and hyperbolic stochastic quantization
equations
for the $\exp(\Phi)_2$-measure constructed in \cite{AH74}
(see \eqref{Gibbs1} and \eqref{Gibbs2} below);
namely, they formally preserve the associated
Gibbs measures with the exponential nonlinear potential.
This provides another motivation to study well-posedness
of the equations \eqref{NH1} and \eqref{NW1}.
We also point out that 
 the  $\exp(\Phi)_2$-measure and the resulting Gaussian multiplicative chaos
play an important role
in Liouville quantum gravity 
\cite{Ka85, DS11, DKRV, DRV, DS19,ORTW};
see also a recent paper \cite{Garban}
for a nice exposition and further references therein.
We also mention the works \cite{AY02, ADG}
on the elliptic $\exp(\Phi)_2$-model.

\medskip

Let us now come back to the viewpoint of analysis on singular stochastic PDEs
and discuss the known results for the stochastic heat and wave equations
with non-polynomial nonlinearities.
In the one-dimensional case,
the stochastic convolution (for the heat or wave equation)
has positive regularity
and thus there is no need for renormalization.
In this case, 
the well-posedness theory for \eqref{NH1} and~\eqref{NW1} on
the one-dimensional torus~$\T$ and invariance of the associated Gibbs measures
(when $\ld >0$) follow in  a straightforward manner
\cite{AKR, STz}.
In the two-dimensional case, 
the stochastic convolution is only a distribution, 
making the problem much more delicate.
To illustrate this, 
we first discuss the case of the sine-Gordon models on $\T^2$
studied in \cite{HS, CHS, ORSW, ORSW2}.
In the parabolic setting,
Hairer-Shen \cite{HS}
and Chandra-Hairer-Shen \cite{CHS}
studied the following 
parabolic sine-Gordon model on $\T^2$:
 \begin{align}
\dt u + \tfrac 12 (1- \Dl)  u   +   \sin(\be u) = \xi.
\label{SG2}
\end{align}

\noi
In this series of work, 
they observed that 
the difficulty of the problem depends sensitively on the value of $\be^2 > 0$.
By comparing the regularities of the relevant singular
stochastic terms,\footnote
{Namely, compare the regularities
of the imaginary Gaussian multiplicative chaos
with the stochastic convolution  
for the  $\Phi^3_d$-model
and with the renormalized square power of the stochastic convolution
for  the $\Phi^4_d$-model.}
we can compare this sine-Gordon model \eqref{SG2}
with  the $\Phi^3_d$- and  $\Phi^4_d$-models, 
at least at a heuristic level;
for example, the $\Phi^3_d$-model 
(and  the $\Phi^4_d$-model, respectively) 
formally corresponds 
to \eqref{SG2} with $d = 2+ \frac{\be^2}{2\pi}$
(and  $d = 2+ \frac{\be^2}{4\pi}$, respectively).
In terms of the actual well-posedness theory,  
the Da Prato-Debussche trick \cite{DPD}
along with a standard Wick renormalization yields
local well-posedness of \eqref{SG2}
for  $0 < \be^2 <4\pi$.
For the  sine-Gordon model \eqref{SG2} on $\T^2$, 
 there is an infinite number of thresholds:
$\be^2 = \frac{j}{j+1}8\pi$, $j\in\N$, 
where one encounters new divergent stochastic objects, requiring further renormalizations.
By using the theory of regularity structures~\cite{Hairer}, 
Chandra, 
Hairer,  and Shen proved local well-posedness of \eqref{SG2}
for  the entire subcritical regime $0 < \be^2<8\pi$. 
More recently, the authors with P.\,Sosoe
studied the hyperbolic counterpart
of the sine-Gordon problem
\cite{ORSW, ORSW2}.
Due to a weaker smoothing property
of the wave propagator, however, 
the resulting solution theory
is much less satisfactory than that in the parabolic case;
 in the damped wave case, local well-posedness
 was established
only for  $0 < \be^2 < 2\pi$.
See also Remark \ref{REM:SW}\,(ii) below.
It is this lack of strong smoothing in the wave case
which makes the problems in the hyperbolic setting
much more analytically challenging than those in the parabolic setting,\footnote{We
mention the recent works
\cite{GKO2, OOTol, Bring2, OOTol2}
on the paracontrolled approach 
to study the stochastic wave equations  on 
the three-dimensional torus $\T^3$, 
which are substantially more involved than 
the paracontrolled approach in the parabolic setting \cite{CC, MW1}.
Note that a standard application 
of the Da Prato-Debussche trick suffices 
to handle the quadratic nonlinearity on $\T^3$ 
in the parabolic setting~\cite{EJS}, while it is not the case in the hyperbolic setting
considered in \cite{GKO2}.
}
and one of our main goals in this paper
is to make a progress in the solution
theory of 
the more challenging
SdNLW \eqref{NW1} with the exponential nonlinearity.
See also 
Remark \ref{REM:STz}.

In terms of regularity analysis, 
SNLH \eqref{NH1} and SdNLW \eqref{NW1}
with the exponential nonlinearity
can also be formally compared to the
$\Phi^3_d$- and  $\Phi^4_d$-models
by 
the  heuristic argument mentioned above,
which yields the same correspondence as in the sine-Gordon case.
While 
 the sine-Gordon model enjoys a certain charge cancellation property
\cite{HS, ORSW}, 
there is no such cancellation property in the exponential model
under consideration, 
which provides an additional difficulty in 
studying the regularity property of the relevant stochastic term
(see Proposition \ref{PROP:Ups} below).
See also \cite{Garban} for  a discussion on intermittency of the problem
with an exponential nonlinearity.

In a recent paper \cite{Garban},
motivated from the viewpoint of Liouville quantum gravity,  
Garban studied the stochastic nonlinear heat equation \eqref{NH1} on $\T^2$
with an exponential nonlinearity~$e^{\be u}$:
\begin{align}
\dt u - \tfrac 12 \Dl  u   +    \tfrac{1}{(2\pi)^\frac{3}{2}} e^{\be u} = \xi.
\label{NH2a}
\end{align}

\noi
See also \eqref{NHG} below.
By studying the regularity property of the Gaussian multiplicative
chaos
(see \eqref{re-non} below)
and applying Picard's iteration argument, 
he proved local well-posedness  of \eqref{NH2a}
for $0 < \be^2 < \frac{8 \pi}{7 + 4\sqrt{3}} \simeq 0.57 \pi$.\footnote{Here, 
the numerology is converted to our scaling convention.
See Remark \ref{REM:Garban} below.}
Furthermore, by exploiting the {\it positivity} of the Gaussian multiplicative chaos, 
he also proved local well-posedness
for the range: $\frac{8 \pi}{7 + 4\sqrt{3}}\leq  \be^2 < \frac{8 \pi}{(1+\sqrt 2)^2} \simeq 1.37 \pi$.
This latter result is without stability under the perturbation of the noise
and,  
in particular, 
the solution $u$ was not shown to be a limit of the solutions
with regularized noises.

\medskip

Before we state  our first main  result
on SNLH \eqref{NH1}, 
let us introduce some notations.
Given $N \in \N$, 
we denote by 
$\P_N$  a smooth frequency projector
onto the (spatial) frequencies  $\{n\in\Z^2:|n|\leq N\}$,
associated with  a Fourier multiplier
\begin{align}
\chi_N(n) = \chi\big(N^{-1}n\big)
\label{chi}
\end{align}

\noi
for some fixed non-negative even function $\chi \in C^\infty_c(\R^2)$
with $\supp \chi \subset \{\xi\in\R^2:|\xi|\leq 1\}$ and $\chi\equiv 1$ 
on $\{\xi\in\R^2:|\xi|\leq \tfrac12\}$. 
Let 
$\{ g_n \}_{n\in\Z^2}$ and $\{ h_n \}_{n\in\Z^2}$ be sequences of mutually independent standard complex-valued\footnote
{This means that $g_0,h_0\sim\NN_\R(0,1)$
and  
$\Re g_n, \Im g_n, \Re h_n, \Im h_n \sim \NN_\R(0,\tfrac12)$
for $n \ne 0$.}
 Gaussian random variables 
 on 
a probability space $(\O_0,P)$ 
conditioned so that $g_{-n} = \cj{g_n}$ and $h_{-n} = \cj{h_n}$, $n \in \Z^2$.
Moreover, we assume that 
$\{ g_n \}_{n\in\Z^2}$ and $\{ h_n \}_{n\in\Z^2}$ are independent from the space-time white noise
$\xi$ in the equations~\eqref{NH1} and \eqref{NW1}.
Then, we define random functions $w_0$ and $w_1$ by setting
\begin{equation} 
w_{0}^{\om} = \sum_{n \in \Z^2 } \frac{ g_n(\om)}{\jb{n}} e_n
 \qquad
\text{and}
\qquad
w_{1}^{\om} = \sum_{n\in \Z^2}  h_n(\om) e_n,
\label{IV2}
\end{equation}

\noi
where  $\jb{n} = \sqrt{1 + |n|^2}$ and 
 $e_n(x) = \frac 1{2\pi}e^{in \cdot x}$ as in \eqref{exp}.
Lastly, given $s \in \R$, 
let $\mu_s$ denote the Gaussian measure on $\D'(\T^2)$
with the density:
\begin{align}
d \mu_s = Z_s^{-1} e^{ - \frac {1}{2} \| u \|_{H^s}^2} du. 
\label{Gauss1}
\end{align}

\noi
On $\T^2$, it is well known that $\mu_s$ is a Gaussian probability measure supported
on $W^{s - 1 - \eps, p}(\T^2)$ for any $\eps > 0$ and $1 \leq p \leq \infty$.
Note that 
the laws of $w_0$ and $w_1$ in \eqref{IV2} 
are given by the massive Gaussian free field $\mu_1$ and the white noise measure $\mu_0$, respectively.

We study
the following truncated SNLH:
\begin{align}
\begin{cases}
\dt u_N + \frac 12 (1- \Dl)  u_N   +  \frac 12 \ld \be C_N  e^{\be u_N} = \P_N \xi\\
u_N |_{t = 0} = u_{0, N}
\end{cases}
\label{NH3}
\end{align}

\noi
for a suitable  renormalization constant  $C_N > 0$, 
with  initial data $u_{0, N}$ of the form:
\begin{align}
u_{0, N} = v_0 + \P_N w_0, 
\label{IV1}
\end{align}

\noi
where $v_0$ is a given deterministic function 
and $w_0$ is as in \eqref{IV2}.
We now state our first local well-posedness result
%Then, we have the following well-posedness
%results 
for SNLH \eqref{NH1}.

\begin{theorem}[local well-posedness in the general case]\label{THM:1}
Let $\ld \ne 0$ and 
$0 < \be^2 < \bh^2: = \frac{8 \pi}{3 + 2 \sqrt 2}\simeq 1.37 \pi$.
%
%Let $\be \ne 0$.
Then, there exists a sequence of  positive constants $\{C_N \}_{N\in \N}$, tending to $0$, 
 \textup{(}see~\eqref{C1} below\textup{)} such that
% the following holds true. 
% Then, 
 the stochastic nonlinear  heat equation \eqref{NH1}
is locally well-posed
in the following sense; 
given  $v_0 \in L^\infty(\T^2)$, 
there exist
 an almost surely positive  stopping time $\tau =  \tau\big(\|v_0\|_{L^\infty}, \be, \ld \big)$ 
 and a non-trivial\footnote{Here, non-triviality means that the limiting process $u$ is not zero 
 or a linear solution. 
 As we see below, the limiting process $u$ admits a decomposition $u = v + z + \Psi$, 
 where $z = P(t)v_0$ denotes the (deterministic) linear solution defined in \eqref{linear}, 
 $\Psi$ denotes the stochastic convolution defined in \eqref{Ph1}, 
 and the residual term $v$ satisfies the nonlinear equation~\eqref{NH5a}.
See, for example,  \cite{HRW,OOR,ORSW},  where in contrast some triviality phenomena appear.
A similar comment applies
in the following statements.} stochastic process $u\in  C([0, \tau]; H^{-\eps}(\T^2))$
for any $\eps > 0$
such that, given any small $T >0 $,  
on the event $\{ \tau \geq T\}$, 
the solution $u_N$ 
to  the truncated SNLH~\eqref{NH3}
with initial data $u_{0, N}$ of the form \eqref{IV1}
converges in probability to $u$ in 
$C([0, T]; H^{-\eps}(\T^2))$.
\end{theorem}

Formally speaking, the limit $u$ in Theorem \ref{THM:1} is a solution to the following ``equation'':
\begin{align}
\begin{cases}
\dt u + \frac 12 (1- \Dl)  u   +  \frac 12  \infty^{-1}\cdot \ld \be  e^{\be u} =  \xi\\
u |_{t = 0} = v_0 + w_0.
\end{cases}
\label{NH3a}
\end{align}

\noi
We will describe a precise meaning of this limiting equation in Subsection \ref{SUBSEC:1.3}.

Note that the model \eqref{NH2a} studied in  \cite{Garban} corresponds to (a massless version of)
our model~\eqref{NH1}
with  $\ld = \frac{2 \be^{-1}}{(2\pi)^\frac{3}{2}}$.
In view of 
the symmetry (in law) for \eqref{NH1}: $(u, \xi, \be, \ld) \mapsto (-u, -\xi, -\be, \ld)$, 
Garban's result covers both\footnote{What is important is
the sign of $\ld$, not its magnitude.
Furthermore, as for the local well-posedness theory, 
there is no essential difference between the massive and massless case.} $\ld > 0$ and $\ld < 0$
as in  Theorem~\ref{THM:1}.
After rescaling,  the upper bound
$\frac{8 \pi}{3 + 2 \sqrt 2}\simeq 1.37 \pi$ on $\be^2$
in Theorem~\ref{THM:1}
agrees with the ``critical'' value $\g_\text{pos} = 2\sqrt 2 - 2$ in~\cite{Garban}.
See Remark \ref{REM:Garban} below.
Namely, the ranges of the parameter $\be^2$
in Theorem~\ref{THM:1} and \cite[Theorems 1.7 and  1.11]{Garban}
agree.
The difference between the result in \cite{Garban}
and Theorem~\ref{THM:1}
 for the range $\frac{8 \pi}{7 + 4\sqrt{3}} \simeq 0.57 \pi
\leq \be^2 < 
\frac{8 \pi}{3 + 2 \sqrt 2}\simeq 1.37 \pi$
appears in the approximation property of the solution.
In \cite{Garban}, 
Garban proved local well-posedness 
of the limiting equation \eqref{NH3a}
in the Da Prato-Debussche formulation
but without continuity in the noise.
In Section~\ref{SEC:fpa}, we will prove 
convergence of the solution $u_N$ of the truncated SNLH~\eqref{NH3}
to the limit $u$, thus establishing continuity in the noise.

In proving Theorem~\ref{THM:1}, 
we apply the Da Prato-Debussche trick as in \cite{Garban}.
By exploiting the positivity of the   Gaussian multiplicative chaos, 
we construct a solution 
by  a standard Picard's iteration argument.
For this purpose, 
we study  
higher moment bounds
of  the  Gaussian multiplicative chaos. This is done with two different approaches: the first one using
 the  Brascamp-Lieb inequality \cite{BL}\footnote
{This
is not to be confused with the  
Brascamp-Lieb concentration inequality \cite[Theorem 5.1]{BL2}
in probability theory, 
which was used in the study of the Gibbs measure
for the defocusing nonlinear Schr\"odinger equations on the real line~\cite{BO00}.}
(see Lemma~\ref{LEM:BL} below), and the other one relying on Kahane's classical approach.

This local well-posedness result by a contraction argument
does not directly provide  continuity in the noise
since in studying
the difference of Gaussian multiplicative chaoses, 
we can no longer exploit any positivity.
In order to prove convergence of the solutions $u_N$ to
the truncated SNLH~\eqref{NH3}, 
we employ a more robust 
 energy method (namely, an a priori bound
and a compactness argument) 
and combine it with the uniqueness of a solution to the limiting 
equation~\eqref{NH3a} in the Da Prato-Debussche formulation. This is turn yields the continuity in the noise property.
See also Remark \ref{REM:intro}\,(ii) below.

In the defocusing case $\ld >0$, we can improve the local well-posedness result of Theorem~\ref{THM:1} 
on two aspects. The first one is that the defocusing nonlinearity
allows us to prove  a global-in-time result in place of a local one. The second and less obvious one is that we can improve on the range of $\be^2>0$. Namely, the defocusing nature of the nonlinearity also improves the local Cauchy theory.

\begin{theorem}[global well-posedness in the defocusing case]\label{THM:1.2} 
Let  $\ld > 0$ and  $0<\be^2 < (\bh^*)^2: = 4\pi $.
Let $\{C_N\}_{N\in\N}$ be as in Theorem \ref{THM:1}.
Then, the stochastic nonlinear heat equation \eqref{NH1}
is globally well-posed in the following sense;
 given 
 $v_0 \in L^\infty(\T^2)$, 
there exists a 
non-trivial stochastic process $u\in C(\R_+; H^{-\eps}(\T^2))$
for any $\eps >0$
such that, given any $T>0$,  
the solution $u_N$ to the truncated SNLH~\eqref{NH3}
with initial data $u_{0, N}$ of the form~\eqref{IV1}
converges in probability to $u$ 
in $C([0, T]; H^{-\eps}(\T^2))$.
\end{theorem}

When $\ld > 0$, the equation \eqref{NH3}
indeed has a {\it sign-definite structure};
see \eqref{posX} for example.
We exploit such a sign-definite structure
at the level of the Da Prato-Debussche
formulation 
to prove Theorem \ref{THM:1.2}.
For $\be^2 \geq \frac{8 \pi}{3 + 2 \sqrt 2}$, 
we need to employ 
an energy method
even to prove existence of  solutions.
Both the sign-definite structure
and the positivity of the   Gaussian multiplicative chaos
play an important role.
We then prove uniqueness by establishing 
an energy estimate for the difference of two solutions.
Continuity in the noise
is shown by an analogous argument
to that in the proof of Theorem \ref{THM:1}.
 Theorem~\ref{THM:1.2} thus
 shows that there is 
 a significant improvement 
from~\cite{Garban}
on the range of $\be^2$
from 
$0 < \be^2 < \frac{8 \pi}{7 + 4\sqrt{3}} \simeq 0.57 \pi$ in \cite{Garban}
to $0 < \be^2 < 4\pi$
when $\ld > 0$.
This answers  Question 7.1 in~\cite{Garban}, 
showing that the value $\g_\text{pos}$ in \cite{Garban}
does {\it not} correspond to a critical threshold, 
at least in the $\ld > 0$ case. 
In view of the heuristic comparison to 
the $\Phi^4_d$-model mentioned above, 
 the range: 
$0 < \be^2 < 4\pi$
in Theorem~\ref{THM:1.2} 
corresponds to 
 the sub-$\Phi^4_3$ case.
 Note that in this range,  
 the Da Prato-Debussche trick and a contraction argument suffice
 for the parabolic sine-Gordon model~\cite{HS, ORSW2}.

\begin{remark}\label{REM:intro}\rm 
(i) For the sake of the argument, 
Theorems~\ref{THM:1} and~\ref{THM:1.2} are stated for  the initial data $u_{0, N} $ of the form \eqref{IV1}.
By a slight modification of the argument, 
 however, we can also treat general deterministic initial data $u_{0, N} = v_0 \in L^\infty(\T^2)$.
 See Remark \ref{REM:SW} below.
 A similar comment applies to Theorem \ref{THM:3}
 for SdNLW \eqref{NW1}.

\smallskip

\noi
(ii) 
In Appendix \ref{SEC:A}, 
we present a local well-posedness argument in the sense of Theorem~\ref{THM:1}, 
in particular for any $\ld \in \R\setminus \{0\}$, 
for the slightly smaller range 
$0 < \be^2 < \frac 43 \pi \simeq 1.33 \pi$
than that in Theorem \ref{THM:1}, 
but {\it without} using the positivity of the 
  Gaussian multiplicative chaos
  or any sign-definite structure of the equation.
This argument also provides stronger
Lipschitz dependence on initial data and noise.
See also Remark \ref{REM:uniq2}.

\smallskip

\noi
(iii) 
The well-posedness results in Theorem \ref{THM:1} and Theorem \ref{THM:A}
for general $\ld \ne 0$
are directly  applicable to the following 
parabolic sinh-Gordon equation on $\T^2$:
\begin{align}
\dt u + \tfrac 12 (1- \Dl)  u   +  \tfrac{1}{2} \be \sinh (\be u)  = \xi, 
\label{NH1x}
\end{align}

\noi
providing local well-posedness of \eqref{NH1x} for the same range of $\be^2$, 
in particular, 
with continuity in the noise.
The model \eqref{NH1x}
corresponds to the so-called $\cosh$-interaction in quantum field theory.
See Remark \ref{REM:cosh} below.
\end{remark}

We now investigate an issue of invariant measures
for \eqref{NH1}  when $\ld > 0$.
Define the energy $\Eh$ by 
\begin{align}
\Eh (u) = \frac 12 \int_{\T^2} |\jb{\nb} u|^2 dx + \ld \int_{\T^2} e^{\be u} dx,
\label{energy1}
\end{align}

\noi
where $\jb{\nb} =  \sqrt{1-\Dl}$.
The condition $\ld > 0$
guarantees that 
the problem is defocusing.
Note that the equation~\eqref{NH1}
formally preserves
the Gibbs measure $\rh$
associated with the energy $\Eh$,
whose density is  formally given by 
\begin{align}
\text{``}d \rh = 
Z^{-1} e^{- \Eh(u)} du 
= 
Z^{-1} \exp\bigg(- \ld \int_{\T^2}e^{\be u} dx \bigg) d\mu_1\text{''}, 
\label{Gibbs1}
\end{align}

\noi
where $\mu_1$ is the massive Gaussian free field defined  in \eqref{Gauss1}.
In view of the low regularity of the support of $\mu_1$, 
we need to apply a renormalization to the density in \eqref{Gibbs1}
so that $\rh$ can be realized as a weighted Gaussian measure
on $\D'(\T^2)$.

In order to preserve the sign-definite structure of the equation for $\ld > 0$, 
we can not use an arbitrary approximation to the identity
for regularization
but we need to use those with non-negative convolution kernels.
Let $\rho$ be a smooth, non-negative, even function compactly supported in $\T^2\simeq [-\pi,\pi)^2$ and such that $\int_{\R^2}\rho(x)dx = 1 $.
Then, given $N \in \N$, we define a smoothing operator $\Q_N$ by setting
\begin{equation}\label{Q}
\Q_Nf  =  \rho_N* f =  \sum_{n\in\Z^2}\big( 2\pi \ft\rho_N(n)\big) \ft f(n)e_n, 
\end{equation}

\noi
where  the mollifier $\rho_N$ is defined by 
\begin{align}
\rho_N(x) = N^2\rho(Nx).
\label{Q2} 
\end{align}
We then
 define the truncated Gibbs measure $\rhN$
by 
\begin{align}
d \rhN = Z_N^{-1} \exp\bigg(- \ld C_N \int_{\T^2}e^{\be \, \Q_N u} dx \bigg) d\mu_1, 
\label{Gibbs1a}
\end{align}

\noi
where $C_N$ is the renormalization constant from  Theorem \ref{THM:1}
but with $\Q_N$ instead of $\P_N$.
As a corollary to  the analysis on the Gaussian multiplicative chaos
(see Proposition \ref{PROP:Ups} below), 
we have the following convergence result.

\begin{proposition}\label{PROP:Gibbs}
Let  $\ld > 0$ and $0 < \be^2 <  (\bh^*)^2 = 4  \pi$.
The sequence $\{\rhN\}_{N \in \N}$
of the renormalized truncated Gibbs measures
converges in total variation to 
some limiting probability measure.
With a slight abuse of notation, 
 we denote the limit by $\rh$.
Then, the limiting renormalized Gibbs measure $\rh$
and the massive Gaussian free field $\mu_1$
are 
mutually  absolutely continuous.

\end{proposition}

\begin{remark}\label{REM:focusing}\rm
We only discuss the construction and invariance of the Gibbs measure in the defocusing case $\ld>0$. Indeed, in the focusing case $\ld<0$, the Gibbs measure \eqref{Gibbs1a} is not normalizable. More precisely, in \cite[Appendix A]{ORTW}, 
 N.~Tzvetkov and  the authors showed that the partition function satisfies 
\begin{align*}
Z_N = \int\exp\bigg(- \ld C_N \int_{\T^2}e^{\be  \Q_N u} dx \bigg) d\mu_1 \too \infty, 
\end{align*}
 as $N\to\infty$ in the case $\ld<0$. See Proposition A.1 in \cite{ORTW}.
 See also 
\cite{LRS, BS,  BB14, OST, OOTol, OST2, OOTol2, Robert}
on  non-normalizability results
for focusing Gibbs measures.

\end{remark}

The truncated Gibbs measure $\rhN$ is invariant under 
the following truncated SNLH:
\begin{align}
\begin{cases}
\dt u_N + \frac 12 (1- \Dl)  u_N   +  \frac 12 \ld \be C_N \Q_N  e^{\be \, \Q_N u_N} = \xi\\
u_N |_{t = 0} = u^\textup{Gibbs}_{0, N}\sim\rhN.
\end{cases}
\label{NH4}
\end{align}

\noi
See Lemma~\ref{LEM:inv} %in Subsection \ref{SUBSEC:Gibbs1} 
below.
By taking $N \to \infty$, we 
then have the following almost sure global well-posedness
and invariance of the renormalized Gibbs measure $\rh$ for SNLH \eqref{NH1}.

\begin{theorem}\label{THM:2}
Let $\ld >0$ and $0 < \be^2  < (\bh^*)^2= 4 \pi$.
Then, 
the stochastic nonlinear heat equation~\eqref{NH1}
is almost surely globally well-posed 
with respect to the random initial data distributed
by the renormalized Gibbs measure $\rh$.
Furthermore, 
 the renormalized Gibbs measure $\rh$ is invariant under the resulting dynamics.

More precisely, 
there exists a 
non-trivial stochastic process $u\in C(\R_+; H^{-\eps}(\T^2))$
for any $\eps >0$
such that, given any $T>0$,  
the solution $u_N$ to the truncated SNLH~\eqref{NH4}
with the random initial data $u_{0, N}^\textup{Gibbs}$ distributed by the truncated Gibbs measure $\rhN$ in \eqref{Gibbs1a}
converges in probability to $u$ 
in $C([0, T]; H^{-\eps}(\T^2))$.
Furthermore, the law of $u(t)$ for any $t \in \R_+$
is given by the renormalized Gibbs measure $\rh$.
\end{theorem}

A variant of  Theorem \ref{THM:1.2}
implies global well-posedness of \eqref{NH4}.
Then, 
in view of the mutual absolute continuity of the renormalized Gibbs measure $\rh$
and the massive Gaussian free field $\mu_1$
and the  convergence in total variation of the truncated Gibbs measure $\rhN$ in \eqref{Gibbs1a}
to the limiting renormalized Gibbs measure $\rh$
(Proposition \ref{PROP:Gibbs}), 
the proof of Theorem \ref{THM:2}
follows 
from a standard argument.
See Subsection \ref{SUBSEC:Gibbs1}.

\begin{remark}\rm
Note that the positivity of the operator $\Q_N$
is needed only for proving local well-posedness of the truncated SNLH \eqref{NH4}
and that Proposition \ref{PROP:Gibbs}
holds with $\P_N$ (or any approximation to the identity) in place of $\Q_N$.
Then, noting that the proof
of  Theorem~\ref{THM:1} does not exploit any sign-definite structure
of the equation, 
we conclude that 
even if we replace $\Q_N$
with $\P_N$ in \eqref{NH4}, 
the conclusion of 
Theorem \ref{THM:2} holds 
true for the range 
$0 < \be^2 < \frac{8 \pi}{3 + 2 \sqrt 2}\simeq 1.37 \pi$.
Since Theorem~\ref{THM:1} 
only provides local well-posedness, 
we need to use Bourgain's invariant measure
argument \cite{BO94, BO96}
to construct almost sure global-in-time dynamics. We refer to \cite{HM,ORTz,GKOT, OOTol,  OOTol2} for the implementation of Bourgain's invariant measure argument in the context of singular SPDEs.

\end{remark}

\medskip

Next, we turn our attention to the stochastic  damped nonlinear wave equation \eqref{NW1}.
Due to a weaker smoothing property of the associated linear operator, 
the problem in this hyperbolic setting is harder than that in the parabolic setting discussed above.
In the following, we restrict our attention
to the defocusing case ($\ld > 0$), where we can hope to exploit a (hidden) sign-definite  structure of the equation.
Given $N \in \N$, we study the following truncated SdNLW:
\begin{align}
\begin{cases}
\dt^2 u_N + \dt u_N +  (1- \Dl)  u_N   +  \ld \be C_N  e^{\be u_N} = \sqrt{2} \P_N \xi\\
(u_N, \dt u_N) |_{t = 0} = (u_{0, N}, u_{1, N}) 
\end{cases}
\label{NW3}
\end{align}

\noi
with  the renormalization constant  $C_N$  from  Theorem \ref{THM:1}
and  initial data $(u_{0, N}, u_{1, N})$  of the form:
\begin{align}
(u_{0, N}, u_{1, N})  = (v_0, v_1)  + (\P_N w_0, \P_N w_1), 
\label{IV4}
\end{align}

\noi
where $(v_0, v_1)$ is a pair of given deterministic functions 
and $(w_0, w_1)$ is as in \eqref{IV2}.

\begin{theorem}\label{THM:3}
Let $\ld > 0$,   $0<\be^2 <
\bw^2: = \frac{32-16\sqrt{3}}{5}\pi  \simeq 0.86\pi$, 
and $s > 1$.
Then, 
 the stochastic  damped nonlinear  wave  equation~\eqref{NW1}
is locally well-posed
in the following sense; 
given $(v_0, v_1)  \in \H^s(\T^2) =  H^s(\T^2)\times H^{s-1}(\T^2)$, 
there exist
 an almost surely positive  stopping time $\tau =  \tau\big(\|(v_0,v_1)\|_{\H^s}, \be, \ld \big)$ 
 and a non-trivial stochastic process $(u, \dt u) \in  C([0, \tau]; \H^{-\eps}(\T^2))$
for any $\eps > 0$
such that, given any small $T >0 $,  
on the event $\{ \tau \geq T\}$, 
the solution $(u_N, \dt u_N) $ 
to  the truncated SdNLW~\eqref{NW3}
with initial data $(u_{0, N}, u_{1, N})$ of the form \eqref{IV4}
converges in probability to $(u, \dt u)$ in 
$C([0, T]; \H^{-\eps}(\T^2))$.

\end{theorem}

Due to a weaker smoothing property of the linear  wave operator, 
 the range of $\be^2$ in Theorem~\ref{THM:3}
is much smaller than that in Theorem~\ref{THM:1.2}
and we can only prove local well-posedness for SdNLW \eqref{NW1}.
Furthermore,   we do not know how to prove
local well-posedness
of SdNLW \eqref{NW1}  in 
the focusing case ($\ld < 0$). 
Namely, there is no analogue of Theorem~\ref{THM:1} 
in this hyperbolic setting at this point.

As in the proof of Theorem \ref{THM:1.2}, 
we proceed with the Da Prato-Debussche trick
but the  proof of  Theorem \ref{THM:3} in the hyperbolic setting is 
 more involved than 
that of Theorem~\ref{THM:1.2} in the parabolic setting.
Due to the oscillatory nature of the Duhamel integral operator
(see \eqref{D} below)
associated with the damped Klein-Gordon
operator
$\dt^2   + \dt  +  (1- \Dl)    $, 
we can not exploit any sign-definite structure as it is.
We point out, however,  that near the singularity, 
the kernel for the Duhamel integral operator
is essentially non-negative.
This observation motivates us
to decompose
 the residual term $v$
in the Da Prato-Debussche argument 
as $v = X+ Y$, 
where the low regularity part $X$ enjoys a sign-definite structure
and the other part $Y$ enjoys 
a stronger smoothing property.
As a result, we reduce the equation~\eqref{NW3}
to a system of equations; see \eqref{XY2} below.
This decomposition of the unknown 
into a less regular but structured part $X$
and a smoother part $Y$ is reminiscent 
of the paracontrolled approach to the dynamical $\Phi^4_3$-model
in \cite{CC, MW1}.  See also \cite{GKO2}.
We will describe an outline of the proof
of Theorem \ref{THM:3}
 in Subsection~\ref{SUBSEC:1.3}.

\medskip

Lastly, we study the Gibbs measure $\rw$
for SdNLW \eqref{NW1} associated with the energy:
\begin{align*}
\Ew (u, \dt u) = \Eh(u) + \frac 12 \int_{\T^2} (\dt u)^2 dx, 
%\label{energy2}
\end{align*}

\noi
where $\Eh$ is as in \eqref{energy1}.
As in the parabolic case, we need to introduce a renormalization.
Define the truncated Gibbs measure $\rwN$
 by 
\begin{align}
d \rwN (u, \dt u) = Z_N^{-1} d(\rhN \otimes \mu_0) (u, \dt u), 
\label{Gibbs2a}
\end{align}

\noi
where $\mu_0$ is the white noise measure defined in \eqref{Gauss1}.
Then, it follows from Proposition~\ref{PROP:Gibbs} that 
 when $0 < \be^2 < 4\pi $, 
the truncated Gibbs measure $\rwN$
converges in total variation
to the 
 renormalized Gibbs measure $\rw$ given by 
 \begin{align}
d \rw (u, \dt u) = Z^{-1} d(\rh \otimes \mu_0) (u, \dt u).
\label{Gibbs2}
\end{align}

Now,  consider  the following truncated SdNLW:
\begin{align}
\begin{cases}
\dt^2 u_N + \dt u_N +  (1- \Dl)  u_N   +  \ld \be C_N  \Q_N e^{\be\,  \Q_N u_N} = \sqrt{2} \xi\\
(u_N, \dt u_N) |_{t = 0} = (u_{0, N}^\textup{Gibbs}, u_{1, N}^\textup{Gibbs}) , 
\end{cases}
\label{NW3a}
\end{align}

\noi
where $\Q_N$ is the mollifier with a non-negative kernel defined in \eqref{Q}
and  $C_N$ is the renormalization constant from  Theorem \ref{THM:1}
but with $\Q_N$ instead of $\P_N$.
Decomposing the truncated SdNLW \eqref{NW3a}
into the deterministic nonlinear wave dynamics:
\begin{align*}
\dt^2 u_N  +  (1- \Dl)  u_N   +  \ld \be C_N  \Q_N e^{\be\,  \Q_N u_N} = 0
\end{align*}

\noi
and the Ornstein-Uhlenbeck process (for $\dt u_N$):
\begin{align*}
\dt^2 u_N +  \dt u_N +  (1- \Dl)  u_N      = \sqrt{2} \xi, 
 \end{align*}

\noi
we see that the truncated Gibbs measure $\rwN$ is invariant under the truncated SdNLW~\eqref{NW3a}.
See Section 4 in \cite{GKOT}.
As a result, we obtain the following almost sure global well-posedness
of \eqref{NW1}
and invariance of the renormalized Gibbs measure $\rw$.

\begin{theorem}\label{THM:4}
Let $\ld > 0$ and  $0<\be^2 <
\bw^2 = \frac{32-16\sqrt{3}}{5}\pi 
\simeq 0.86\pi$.
Then, 
the stochastic damped nonlinear wave equation \eqref{NW1}
is almost surely globally well-posed 
with respect to the renormalized Gibbs measure $\rw$.
Furthermore, 
 the renormalized Gibbs measure $\rw$ is invariant under the resulting dynamics.

More precisely, 
there exists a 
non-trivial stochastic process $(u, \dt u) \in C(\R_+; \H^{-\eps}(\T^2))$
for any $\eps >0$
such that, given any $T>0$,  
the solution $(u_N, \dt u_N)$ to the truncated SdNLW~\eqref{NW3a}
with the random initial data $(u_{0, N}^\textup{Gibbs}, u_{1, N}^\textup{Gibbs})$ distributed by the truncated Gibbs measure $\rwN$ 
in~\eqref{Gibbs2a}
converges in probability to $(u, \dt u)$ 
in $C([0, T]; \H^{-\eps}(\T^2))$.
Furthermore, the law of $(u(t), \dt u(t))$ for any $t \in \R_+$
is given by the renormalized Gibbs measure $\rw$.

\end{theorem}

Unlike Theorem \ref{THM:1.2} in the parabolic setting, 
Theorem \ref{THM:3}
does not yield global well-posedness of SdNLW \eqref{NW1}.
Therefore, in order to prove Theorem \ref{THM:4}, 
we need to employ Bourgain's invariant measure argument \cite{BO94, BO96}
to first prove almost sure global well-posedness
by exploiting invariance of the truncated Gibbs measure $\rhN$
for the truncated dynamics~\eqref{NW3a}.
Since such an argument is by now standard, 
 we omit details.
See, for example,~\cite{ORTz, STz, OOTol, Bring2, OOTol2}
in the context of the (stochastic) nonlinear wave equations.

\begin{remark} \label{REM:STz} \rm

In \cite{STz}, Sun and Tzvetkov studied the following
 (deterministic) dispersion-generalized 
nonlinear wave equation (NLW) on $\T^d$ with the exponential nonlinearity:
\begin{align}
\dt^2 u  +  (1- \Dl)^\al  u   +    e^{ u} =0
\label{NW4}
\end{align}

\noi
and the associated Gibbs measure $\rho_\al$.
When $\al > \frac d2$, 
they proved almost sure global well-posedness
of \eqref{NW4} with respect to the Gibbs measure $\rho_\al$
and invariance of $\rho_\al$. 
We point out that, 
 when $\al > \frac d2$, 
a solution $u$ is a function and no normalization is required.
As such, 
 the analysis in \cite{STz} also applies to\footnote{In the massless case:
$\dt^2 u +( - \Dl)^\al  u   +    e^{ \be u} =0$, 
by scaling analysis, we can reduce the problem to the $\be = 1$ case 
(on a dilated torus, where the analysis in \cite{STz} still applies).}  
\begin{align}
\dt^2 u  +  (1- \Dl)^\al  u   +    e^{ \be u} =0
\label{NW4a}
\end{align}

\noi
for {\it any} $\be \in \R\setminus\{0\}$
and a precise value of $\be$ is irrelevant in this non-singular setting.

%due to the fact that 
%the base Gaussian measure~$\mu_\al$ for the Gibbs measure $\rho_\al$ has a continuous covariance function 
%when $\al > \frac d2$
%(compare this with the analysis in Section~\ref{SEC:Gauss},
%in particular Lemma~\ref{LEM:covar});
%namely, when $\al > \frac d2$, 
%a solution $u$ is a function and no normalization is required.

When $d = 2$, their result barely misses the $\al = 1$
case, 
corresponding to the wave equation, 
and
the authors in \cite{STz} posed the $\al = 1$ case on $\T^2$ as an interesting and challenging open problem.
By adapting the proofs of Theorems~\ref{THM:3} and~\ref{THM:4} to the deterministic NLW setting, 
our argument yields
 almost sure global well-posedness
of \eqref{NW4a} 
for $\al= 1$ and $0<\be^2<\bw^2$ with respect to the (renormalized) Gibbs measure $\rho_1$ (\,$= \rw$ in \eqref{Gibbs2})
and invariance of $\rw$, thus answering the open question in an affirmative manner in this regime of $\be^2$.

\end{remark}

\subsection{On the Gaussian multiplicative chaos}
\label{SUBSEC:1.2}

In this subsection, we go over a renormalization procedure for our problems.
In the following, we present a discussion in terms of the frequency truncation operator
$\P_N$ but exactly the same results hold
for the smoothing operator $\Q_N$ defined in \eqref{Q}.
We begin by studying the following 
linear stochastic heat equation with a regularized noise:
\[
\begin{cases}
\partial_t \Ph +\frac12(1-\Delta) \Ph = \P_N \xi\\
\Ph |_{t= 0} = \P_N w_0, 
\end{cases}
\]

\noi
where $w_0$ is the random distribution defined in \eqref{IV2}, 
distributed according to  the massive Gaussian free field $\mu_1$.
Then, the truncated stochastic convolution $\Ph$
is given by 
\begin{align}
\label{conv}
\Ph (t) = P(t) \P_N w_0  + \int_0^t P(t-t') \P_N  dW(t') ,
\end{align}

\noi
where  
 $P(t) = e^{\frac{t}2(\Dl-1)}$ denotes the linear heat operator
defined by 
\begin{align}
\label{heat1}
P(t) f  = e^{\frac{t}2(\Dl-1)} f = \sum_{n\in \Z^2} e^{-\frac{t}2 (1+|n|^2)} \ft f (n)e_n
\end{align}

\noi
and $W$ denotes 
the cylindrical Wiener process on $L^2(\T^2)$   defined by
\begin{align}
W(t) = \sum_{n\in \Z^2}  B_n (t)  e_n.
\label{Wiener}
\end{align}

\noi
Here, 
$\{B_n\}_{n\in \Z^2}$ is a family of mutually independent
complex-valued Brownian motions conditioned so that $B_{-n} = \cj{B_n}$, 
$n\in \Z^2$.
By convention, we normalize $B_n$ such that $\text{Var}(B_n(t)) = t$
and assume that $\{B_n\}_{n\in \Z^2}$
is independent from $w_0$ and $w_1$ in \eqref{IV2}.

Given $N \in \N$, 
we have $\Ph \in C(\R_+ \times \T^2)$.
For each fixed $t\ge 0$ and $x\in \T^2$,
it is easy to  see that $\Ph (t,x)$ is a mean-zero real-valued Gaussian random variable with variance
(independent of $(t, x) \in \R_+ \times \T^2$):
\begin{align}
\sh &  = \E \big[\Ph(t,x)^2\big]
 =  \frac{1}{4\pi^2}\sum_{n \in \Z^2}  \chi_N^2(n)
\bigg( \frac{e^{-t\jb{n}^2}}{ \jb{n}^2} +  \int_0^t \bigg[e^{- \frac12 (t-t') \jb{n}^2} \bigg]^2 dt' \bigg) 
\notag \\
& =  \frac{1}{4\pi^2}\sum_{n \in \Z^2} \chi_N^2(n) \frac{1}{ \jb{n}^2}  \sim  \log N
\too \infty, 
\label{sig}
\end{align}
  
\noi
as $N \to \infty$.
This essentially shows that   $\{\Psi_N (t)\}_{N\in \N}$ is almost surely unbounded in $W^{0,p}(\T^2)$
for any $1\le p \le \infty$.

In the case of the  wave equation, 
we consider the following linear stochastic damped wave equation with a regularized noise:
\begin{equation*}
\begin{cases}
\dt^2\Pw + \dt\Pw + (1-\Dl)\Pw  = \sqrt 2\P_N \xi,\\
(\Pw,\dt\Pw)|_{t=0}=(\P_N w_0, \P_N w_1),
\end{cases}
%\label{W1}
\end{equation*}

\noi
where $w_0$ and $w_1$ are as in \eqref{IV2}.
Then, the stochastic convolution $\Pw$ in this case is given by 
\begin{equation}
\Pw(t) =  \dt\D(t)\P_Nw_0 +  \D(t)\P_N\big(w_0+w_1) + \sqrt 2 \int_0^t\D(t-t')\P_N dW(t'), 
\label{W2}
\end{equation}

\noi
where the linear damped wave operator $\D(t)$ is given by
\begin{equation}\label{D}
\D(t) = e^{-\frac t2 }\frac{\sin\Big(t\sqrt{\tfrac34-\Dl}\Big)}{\sqrt{\tfrac34-\Dl}}, 
\end{equation}

\noi
viewed as a Fourier multiplier operator:
\begin{align}
\D(t) f  =  e^{-\frac t2}\sum_{n\in \Z^2} 
\frac{\sin\Big(t\sqrt{\tfrac34+ |n|^2 }\Big)}{\sqrt{\tfrac34 + |n|^2 }}
 \ft f (n) e_n.
\label{W3}
\end{align}

\noi
One can easily derive the propagator $\D(t)$ in \eqref{D} by writing the linear damped wave equation
$\dt^2 u + \dt u + (1-\Dl) u = 0$
 on the Fourier side:
 \[ \dt^2 \ft u(t, n)  + \dt \ft u(t, n) + \jb{n}^2 \ft u(t, n) =0\]
 
 \noi
  and solving it directly for each spatial frequency $n \in \Z^2$. 
Then, a standard variation-of-parameter argument yields 
the expression \eqref{W2}.
By a direct computation using \eqref{W2} and~\eqref{W3}, 
we obtain, for any $(t, x) \in \R_+ \times \T^2$, 
\begin{align}
\sw &  = \E \big[\Pw(t,x)^2\big]
 =  \frac{1}{4\pi^2}\sum_{n \in \Z^2} \chi_N^2(n) \frac{1}{ \jb{n}^2}  \sim  \log N
\too \infty, 
\label{sigW}
\end{align}

\noi
as $N \to \infty$.

In the following, we set  
\begin{align*}
\Psi_N = \Ph   \, \text{ or }\,   \Pw
\qquad \text{and}\qquad
\s_N = \sh = \sw.
%\label{W4}
\end{align*}

\noi
Since we do not study the stochastic heat and wave equations at the same time, 
their meaning will be clear from the context.

By a standard argument, we then have the following regularity and convergence  result
for the (truncated) stochastic convolution.
See, for example, 
\cite[Proposition 2.1]{GKO} in the context of the wave equation.

\begin{lemma}\label{LEM:psi}

Given any  $T,\eps>0$ and finite $p \geq 1$, 
 $\{\Psi_N\}_{N\in \N}$ is a Cauchy sequence in $L^p(\O;C([0,T];W^{-\eps,\infty}(\T^2)))$,
 converging to some limit $\Psi$ in $L^p(\O;C([0,T];W^{-\eps,\infty}(\T^2)))$.
Moreover,  $\Psi_N$  converges almost surely to the same  limit $\Psi\in C([0,T];W^{-\eps,\infty}(\T^2))$.
\end{lemma}

Clearly, the limiting stochastic convolution is given by formally taking $N \to \infty$
in \eqref{conv} or \eqref{W2}.
Namely, in the heat case, we have
\begin{align}
\Psi (t) = \Psi^\text{heat}(t) = P(t)  w_0  + \int_0^t P(t-t')   dW(t') , 
\label{Ph1}
\end{align}

\noi
while in the wave case, it is given by 
\begin{equation}
\Psi(t) = \Psi^\text{wave}(t) = \dt\D(t)w_0 + \D(t)\big(w_0+w_1) + \sqrt 2 \int_0^t\D(t-t')dW(t').
\label{W6}
\end{equation}

Next, we study the Gaussian multiplicative chaos 
formally given by 
\begin{align*}
e^{\be \Psi_N} = \sum_{k=0}^\infty \frac{\be^k}{k!} \Psi_N^k (t).
\end{align*}

\noi
Since $\Psi_N^k $, $k \geq 2$, 
does not have any nice limiting behavior as $N \to \infty$, 
we
 now introduce the Wick renormalization:
\begin{align}
:\!\Psi_N^k(t,x)\!:\, \stackrel{\text{def}}{=} H_k\big(\Psi_N(t,x);\s_N \big), 
\label{P2}
\end{align}

\noi
where  $H_k$ denotes the $k$th Hermite polynomial, defined through the generating function:
\begin{equation}\label{Hermite}
e^{tx-\frac{\s^2}2 t} = \sum_{k= 0}^\infty\frac{t^k}{k!}H_k(x;\s).
\end{equation}

\noi
From \eqref{P2} and \eqref{Hermite}, 
the (renormalized) Gaussian multiplicative chaos is then given by
 \begin{align}
\begin{split}
 \U_N(t,x)
 &  = \,:\!e^{\be\Psi_N(t,x)}\!:\,\,  
 \deff
 \sum_{k= 0}^\infty\frac{\be^k}{k!}:\!\Psi_N^k(t, x)\!:\\
 & =  e^{-\frac{\be^2}2 \s_N}e^{\be\Psi_N(t,x)}.
\end{split}
\label{re-non}
 \end{align} 

\noi
We also set $C_N = C_N(\be)$ by 
\begin{align}
C_N = e^{-\frac{\be^2}2 \sigma_N} \too 0,
\label{C1}
\end{align}

\noi
as $N\to \infty$.

The following proposition provides the regularity and convergence properties of the Gaussian multiplicative chaos $\U_N$.

\begin{proposition}\label{PROP:Ups} 
%The following hold:\\
\textup{(i)} Given $0<\be^2<8\pi$, let $1\leq p<\frac{8\pi}{\be^2}$ and define $ \al=\al(p)$ by
\begin{equation}\label{al}
\frac{(p-1)\be^2}{4\pi} <\al(p)<2.
\end{equation}

\noi
Then, given any $T>0$, the sequence of 
stochastic processes  $\U_N$ is uniformly bounded in
\[L^p(\O;L^p([0,T];W^{-\al,p}(\T^2))).\]

\smallskip 
\noi
\textup{(ii)} Given $0<\be^2<4\pi$, let $1<  p < \frac{8\pi}{\be^2}$ and $\al(p)$ as in \eqref{al}. Then, given any $T>0$, $\{\U_N\}_{N\in\N}$ is a Cauchy sequence in $L^p(\O;L^p([0,T];W^{-\al,p}(\T^2)))$ and hence converges to some limit $\U$ in the same class.
In particular, $\U_N$ converges in probability 
to $\U$
in $L^p([0,T];W^{-\al ,p}(\T^2))$.

 \end{proposition}

\medskip

%\noi
In the following, we write the limit $\U$ as 
\begin{align}
\U = \, :\!e^{\be \Psi}\!: \, = \lim_{N\to \infty} \U_N  
= \lim_{N\to \infty}  C_N e^{\be \Psi_N }.
\label{U2}
\end{align}

\noi
We point out that 
by applying Fubini's theorem,  
 a proof of Proposition \ref{PROP:Ups} 
reduces
 to analysis for fixed $(t, x) \in \R_+ \times \T^2$.
Therefore, the proof is identical for
$\Psi_N = \Ph    \text{ and }   \Pw$.

In \cite{Garban}, 
Garban established an analogous result on the Gaussian multiplicative chaos
but in the context of the space-time H\"older regularity;
see \cite[Theorem 3.10]{Garban}.
See also 
\cite[Theorem 6]{ADG} for an analogous approach in the elliptic setting, 
working in the $L^p$-based Besov spaces but only for $1 < p \leq 2$.

In the case of a polynomial nonlinearity
\cite{GKO, GKO2}, the $p$th moment bound follows directly from the second moment
estimate combined with  the Wiener chaos estimate (see, for example, Lemma 2.5 in \cite{GKO2}),
since the stochastic objects in \cite{GKO, GKO2}
all belong to Wiener chaoses of finite order. 
However, 
the Gaussian multiplicative chaos $\U_N$ in \eqref{re-non}
 does \emph{not} belong to any  Wiener chaos of finite order.
Therefore, we need to  estimate all the higher moments by hand. 
The  approach in \cite{Garban} is based on Kahane's convexity inequality \cite{Ka85};
see Lemma~\ref{LEM:Kahane}.
In Section \ref{SEC:Gauss},  we first  compute higher even moments,  
using the Brascamp-Lieb inequality \cite{BL,Lieb,BCCT}.
See Lemma~\ref{LEM:BL} and Corollary~\ref{COR:BL}.
We believe that our approach based on the Brascamp-Lieb inequality is of independent interest. In order to compare this approach with  Kahane's, we also provide a proof of Proposition~\ref{PROP:Ups} based on Kahane's inequality. See Propositions~\ref{PROP:var} and~\ref{PROP:var2}  as well as Appendix~\ref{SEC:B}.

We conclude this subsection by briefly discussing a proof of Proposition \ref{PROP:Gibbs}.

\begin{proof}[Proof of Proposition \ref{PROP:Gibbs}]

As mentioned above, the proof of Proposition \ref{PROP:Ups}
is based on reducing the problem for fixed $(t, x) \in \R_+\times \T^2$.
In particular, it follows from the proof of Proposition \ref{PROP:Ups}
presented in Section \ref{SEC:Gauss}
that $\U_N(0)$ at time $t = 0$ converges to $\U(0)$
in $L^p(\O;W^{-\al,p}(\T^2))$.
Then, by restricting to the (spatial) zeroth Fourier mode, 
we obtain convergence in probability (with respect to the Gaussian free field $\mu_1$ in \eqref{Gauss1})
of the 
density
\begin{align}
R_N =   \exp\bigg(- \ld C_N \int_{\T^2}e^{\be \, \Q_N u} dx \bigg)
  = \exp \big( - 2\pi \ld \ft \U_N(0, 0)\big)
\label{W7}
\end{align}

\noi
to 
\begin{align*}
R =   \exp\bigg(- \ld  \int_{\T^2}: e^{\be  u}:  dx \bigg)
  = \exp \big( - 2\pi \ld \ft \U(0, 0)\big).
%\label{W8}
\end{align*}

\noi
Moreover, by the positivity of $\U_N$
and $\ld$, 
the density $R_N$ in \eqref{W7} is uniformly bounded by 1.
Putting together, 
we conclude the $L^p(\mu_1)$-convergence of the density 
$R_N$ to $R$
 by a standard argument (see \cite[Remark~3.8]{TZ2}). 
More precisely, the $L^p$-convergence of $R_N$ 
 follows from the uniform $L^p$-bound on $R_N$ 
and the softer convergence in probability. \end{proof}

\subsection{Outline of the proof}
\label{SUBSEC:1.3}

In the following, we briefly describe an outline of the proofs
of Theorems~\ref{THM:1}, \ref{THM:1.2}, \ref{THM:2}, \ref{THM:3}, and \ref{THM:4}.

\medskip

\noi
$\bullet$ {\bf Parabolic case:}
Given $v_0 \in L^\infty(\T^2)$, 
we consider the truncated SNLH \eqref{NH3}.
We proceed with the Da Prato-Debussche trick
and write a solution $u_N$ to \eqref{NH3}
as
\begin{align*}
u_N  = v_N  + z + \Psi_N,
\end{align*}

\noi
where  $\Psi_N = \Ph$ is the truncated stochastic convolution in  \eqref{conv} 
and $z$ denotes the linear solution
given by 
\begin{align}
z = P(t) v_0.
\label{linear}
\end{align}

\noi
Then, the residual term $v_N$ satisfies the following equation:
\begin{align}
\label{NH5}
\begin{cases}
\partial_t v_N + \frac 12 (1- \Delta) v_N +  \frac 12 \ld \be
  e^{\be z }  e^{\be v_N} \U_N=0\\
v_N|_{t = 0} = 0,
\end{cases}
\end{align}

\noi
where $\U_N = \, :\!e^{\be \Ph}\!:$ denotes the Gaussian multiplicative noise defined in \eqref{re-non}.

When 
$0 < \be^2 < \frac{8 \pi}{3 + 2 \sqrt 2}\simeq 1.37 \pi$, 
we prove local well-posedness of \eqref{NH5}
by a standard contraction argument.
The key ingredients are
Proposition \ref{PROP:Ups}
on the regularity of the Gaussian multiplicative chaos $\U_N$
and the positivity of the nonlinearity, 
in particular the positivity of $\U_N$
(see Lemma \ref{LEM:posprod}).
In studying 
 continuity in the noise, 
we can no longer exploit any positivity.
For this part of the argument, 
we use a more robust energy method 
and combine it with the uniqueness of a solution to the limiting 
equation (see \eqref{NH5a} below).

Theorem~\ref{THM:1} follows once we prove the following 
local well-posedness result for \eqref{NH5}.

\begin{theorem}
\label{THM:para1}
Let $\ld \ne 0$ and 
$0 < \be^2 < \bh^2 = \frac{8 \pi}{3 + 2 \sqrt 2}\simeq 1.37 \pi$.
Given any $v_0 \in L^\infty(\T^2)$, 
the Cauchy problem \eqref{NH5}
is uniformly locally well-posed
in the following sense;
there exists $T_0 = T_0\big(\|v_0\|_{L^\infty}, \be, \ld \big) >0$
such that 
 given $0 < T\le T_0$
and $N \in \N$,  
there exists a set 
$\O_N(T)\subset \O$
such that 
\begin{itemize}
\item[\textup{(i)}]
for any $\om \in \O_N(T)$, 
there exists a unique solution $v_N$ to \eqref{NH5} 
in the class\textup{:} 
\[C([0,T];W^{s, p} (\T^2))
\subset C([0,T];L^\infty (\T^2))\]

\noi
for some appropriate $0 < s< 1$ and $ p \ge 2$, 
satisfying $sp > 2$.

\smallskip

\item[\textup{(ii)}] there exists a uniform estimate on the probability of 
the complement of $\O_N(T)$\textup{:}
\[P(\O_N(T)^c) \too 0,  \]

\noi
 uniformly in $N \in \N$, as $T \to 0$,

\end{itemize}

Furthermore, 
there exist
 an almost surely positive  stopping time $\tau =  \tau\big(\|v_0\|_{L^\infty}, \be \big)$ 
 and a stochastic process $v\in C([0,T];W^{s, p} (\T^2))$
such that, given any small $T >0 $,  
on the event $\{ \tau \geq T\}$, 
the sequence $\{v_N \}_{N\in \N}$
converges in probability to $v$ in 
$C([0,T];W^{s, p} (\T^2))$.

\end{theorem}

The limit $v$ satisfies the following equation:
\begin{align}
\label{NH5a}
\begin{cases}
\partial_t v + \frac 12 (1- \Delta) v +  \frac 12 \ld \be
 e^{\be z }  e^{\be v}\U=0\\
v|_{t = 0} = 0, 
\end{cases}
\end{align}
 
\noi
where $\U$ is the limit of $\U_N$ constructed in Proposition \ref{PROP:Ups}.
Then, $u = v + z+ \Psi$ formally satisfies the equation \eqref{NH3a}.

Next, we discuss the $\ld > 0$ case.
In this case, the equation \eqref{NH5} enjoys a sign-definite structure.
By writing \eqref{NH5} in the Duhamel formulation, we have
\begin{align*}
v_N(t) = - \frac 12 \ld \be\int_0^t 
P(t - t') \big(   e^{\be z }  e^{\be v_N} \U_N\big)(t') dt'.
\end{align*}

\noi
Since the  kernel for $P(t)= e^{\frac{t}2(\Dl-1)}$ 
and the integrand $ e^{\be z }  e^{\be v_N}\U_N$
are both positive, we see that 
\begin{align}
 \be v_N \leq  0.
\label{posX}
 \end{align}

\noi
This observation shows that the nonlinearity $e^{\be v_N}$ is in fact bounded, 
 allowing us to 
rewrite~\eqref{NH5} as
\begin{align}
\label{NH7}
\begin{cases}
\partial_t v_N + \frac 12 (1- \Delta) v_N +  \frac 12 \ld \be
 e^{\be z }  F(\be v_N) \U_N =0\\
v_N|_{t = 0} = 0,
\end{cases}
\end{align}

\noi
where $F$ is a smooth bounded function such that 
\begin{equation}
F(x) = e^x 
\label{F1}
\end{equation}

\noi
for $x \leq 0$
and $F|_{\R_+} \in C^{\infty}_c(\R_+; \R_+)$.
In particular,   
$F$ is Lipschitz.
By making use of this particular structure
and the positivity of the Gaussian multiplicative chaos $\U_N$,
we  prove a stronger well-posedness result,
from which Theorem \ref{THM:1.2} follows.

\begin{theorem}
\label{THM:para2}
Let $\ld > 0$ and  $0<\be^2 <(\bh^*)^2 = 4\pi $.
Given 
any $v_0 \in L^\infty (\T^2)$,  any $T>0$, 
and any $N \in \N$, there exists a unique solution $v_N$ to \eqref{NH5}
in the energy space\textup{:} 
\begin{align}
\ZZ_T=C([0,T];L^2(\T^2))\cap L^2([0,T];H^{1}(\T^2))
\label{energyZ}
\end{align}

\noi
almost surely
such that 
$v_N$ converges in probability to 
some limit $v$ in the class $\ZZ_T$.
Furthermore, $v$ is the unique solution to 
 the equation
 \eqref{NH5a} in the class $\ZZ_T$.

\end{theorem}

For Theorem \ref{THM:para2}, 
a contraction argument does not suffice
even for constructing solutions
and thus we proceed with an 
 energy method.
 Namely, we first establish
 a uniform (in $N$)  a priori bound for a solution to~\eqref{NH7}.
Then, by  applying a compactness lemma (Lemma~\ref{LEM:AL})
and extracting a convergent subsequence, 
we  prove existence of a solution.
Uniqueness follows from 
 an energy consideration for the difference of two solutions
 in the energy space $\ZZ_T$.
As for  continuity in the noise, in particular convergence of $v_N$ to $v$, 
we lose the positivity of the stochastic term (i.e.~$\U_N - \U$ is not positive).
We thus first establish convergence in some weak norm
and then combine this with strong convergence  (up to a subsequence)
via the compactness argument mentioned above and the uniqueness of 
the limit $v$ as a solution to~\eqref{NH5a} in the energy space $\ZZ_T$.

\medskip

\noi
$\bullet$ {\bf Hyperbolic case:}
Next, we discuss the stochastic damped nonlinear wave equation
when $\ld > 0$.
Let  $N \in \N \cup \{\infty\}$.
Given $(v_0, v_1)\in  \H^s(\T^2)$, 
let $u_N$ be the solution to \eqref{NW3}.
Proceeding with the Da Prato-Debussche trick
$u_N = v_N + z + \Pw$, the residual term $v_N$ satisfies the following equation:
\begin{equation}
\begin{cases}
\dt^2v_N+ \dt v_N +(1-\Dl)v_N +  \ld \be  e^{\be z} e^{\be v_N} \U_N = 0\\
(v_N,\dt v_N)|_{t=0}=(0,0),
\end{cases}
\label{v4}
\end{equation}

\noi
where 
 $\U_N = \, :\!e^{\be \Pw}\!:$ for $N \in \N$, 
  $\U_\infty = \U = \lim_{N \to \infty} \U_N$
constructed in Proposition~\ref{PROP:Ups}, 
and 
$z$ denotes the linear solution  given by 
\begin{align}
z(t)=  \dt\D(t)v_0 + \D(t)\big(v_0+v_1), 
\label{zW}
\end{align}

\noi
satisfying the following linear equation:
\begin{equation*}
\begin{cases}
\dt^2z+ \dt z+(1-\Dl) z= 0\\
(z,\dt z)|_{t=0}=(v_0,v_1).
\end{cases}
\end{equation*}

\noi
Since the smoothing property of the wave operator is weaker than that of the heat equation, 
there is no uniform (in $N$) $L^\infty$-control for $v_N$
(which is crucial in bounding the nonlinearity $e^{\be v_N}$)
and thus  
we need to exploit a sign-definite structure as in SNLH \eqref{NH1} for $\ld > 0$ discussed above.
The main issue is the oscillatory nature of the 
kernel for $\D(t)$ defined in~\eqref{D}.
In particular, unlike the case of the heat equation, 
there is no explicit sign-definite structure for \eqref{v4}.

In the following, we drop the subscript $N$ for simplicity of notations.
Write \eqref{v4} in the Duhamel formulation:
\begin{equation*}
v(t)=- \ld \be \int_0^t \D(t-t')\big(e^{\be z} e^{\be v}\U\big)dt',
\end{equation*}

\noi
where $\D(t)$ is as in \eqref{D}.
The main point is that while the kernel for $\D(t)$ is not sign-definite, 
it is essentially non-negative near the singularity.
This motivates us to introduce a further  decomposition of the unknown:
\begin{equation}\label{XY}
v=X+Y, 
\end{equation}

\noi
where $(X,Y)$ solves the following system of equations:
\begin{align}
\begin{split}
X(t)& =- \ld \be \int_0^te^{-\frac{(t-t')}{2}} S(t-t')\big( e^{\be z}e^{\be X}e^{\be Y}\U\big)(t')dt',\\
Y(t)& = -\ld \be \int_0^t\big(\D(t-t')-e^{-\frac{(t-t')}{2}}S(t-t')\big)
\big( e^{\be z}e^{\be X}e^{\be Y}\U\big)(t')dt'.
\end{split}
\label{XY2}
\end{align}

\noi
Here,  $S(t)$ 
denotes the forward propagator for the standard wave equation: $\dt^2 u-\Dl u=0$ 
with initial data $(u,\dt u)|_{t=0} = (0,u_1)$ given by 
\begin{equation}\label{S1}
S(t) =  \frac{\sin(t|\nabla|)}{|\nabla|}.
\end{equation}

\noi
The key point in that, in view of the positivity of the kernel 
for $S(t)$ (see Lemma~\ref{LEM:waveker} below), 
there is a sign-definite structure for the $X$-equation when $\ld > 0$ 
and we have 
\begin{equation*}
\be X\le 0.
\end{equation*}

\noi
With $F$ as in \eqref{F1}, 
we can then write \eqref{XY2} as
\begin{align}
\begin{split}
X(t)& =- \ld \be \int_0^te^{-\frac{(t-t')}{2}} S(t-t')\big(e^{\be z}F(\be X)e^{\be Y}\U\big)(t')dt',\\
Y(t)& =-\ld \be \int_0^t\big(\D(t-t')-e^{-\frac{(t-t')}{2}}S(t-t')\big)
\big(e^{\be z} F(\be X)e^{\be Y}\U\big)(t')dt'.
\end{split}
\label{XY3}
\end{align}

\noi
Thus, the nonlinear contribution $F(\be X)$ from $X$ is bounded
thanks to the sign-definite structure.
This is crucial since, as we see below, $X$ does not have sufficient regularity
to be in $L^\infty(\T^2)$.
While $X$ and $Y$ both enjoy the Strichartz estimates, 
the difference  of the propagators in the $Y$-equation
provides an extra smoothing, gaining two derivatives
(see Lemma \ref{LEM:smooth} below).
This smoothing of two degrees allows us to place $Y$
in $C([0,T]; H^s(\T^2))$ for some $s>1$ and to make sense of $e^{\be Y}$. 
In Section \ref{SEC:wave}, we prove the following theorem.

\begin{theorem}\label{THM:wave}
Let $\ld > 0$,   $0<\be^2 <\bw^2 = \frac{32-16\sqrt{3}}{5}\pi \simeq 0.86\pi$, 
and $s > 1$.
Suppose that a deterministic positive distribution $\U$ satisfies the regularity property  stated in 
Proposition~\ref{PROP:Ups}.  Namely, 
$\U \in L^p([0,1];W^{-\al ,p}(\T^2))$
%for any finite  $p >1$, 
for  $1\le p < \frac{8\pi}{\be^2}$, 
where $\al = \al(p)$ is as in~\eqref{al}.
Then,
given $(v_0, v_1) \in \H^s(\T^2)$, 
 there exist $T=T\big(\|(v_0, v_1)\|_{\H^s}, \|\U\|_{L^p_T W^{-\al, p}_x}\big)>0$ 
 and a unique solution $(X,Y)$ to~\eqref{XY3} in the class\textup{:} 
 \begin{align*}
  \X_T^{s_1}\times \Y_T^{s_2}
  \subset 
  C([0,T];H^{s_1}(\T^2))\times   C([0,T];H^{s_2}(\T^2))
 \end{align*}
 
 \noi
 for some $0< s_1<1<s_2$
 and some $(\al, p)$ satisfying \eqref{al}.
Moreover, the solution $(X, Y)$ depends continuously 
on 
\[(v_0, v_1, \U) \in \H^s(\T^2)  
\times   L^p([0,T];W^{-\al + \eps ,p}(\T^2))\]

\noi
for sufficiently small $\eps > 0$
\textup{(}such that the pair $(\al + \eps, p)$ satisfies the condition \eqref{al}\textup{)}.
%for some $p > 1$ and $\al >0$.

\end{theorem}

Here,  the spaces $\X^{s_1}_T$ and $\Y^{s_2}_T$
are defined by 
\begin{align}
\X^{s_1}_T& =C([0,T];H^{s_1}(\T^2))\cap C^1([0,T];H^{s_1-1}(\T^2))\cap L^q([0,T];L^r(\T^2)), 
\label{XX1}\\
\Y^{s_2}_T & =  C([0,T];H^{s_2}(\T^2))\cap C^1([0,T];H^{s_2-1}(\T^2)),
\label{YY1}
\end{align}

\noi
for some  suitable $s_1$-admissible pair $(q, r)$. 
See Subsection \ref{SUBSEC:waveker}.
Note that Theorem \ref{THM:3} directly follows 
from Theorem \ref{THM:wave}.
As for 
Theorem \ref{THM:4}, a small modification of the proof of Theorem~\ref{THM:wave}
yields the result.
See Section \ref{SEC:wave} for details.

We point out that this reduction of \eqref{v4} to the system \eqref{XY3}, 
involving the decomposition of the unknown
(in the Da Prato-Debussche argument) 
into a less regular but structured part and a smoother part, 
has some similarity to the paracontrolled approach
to the dynamical $\Phi^4_3$-model.\footnote{This is not to be confused
with the Da Prato-Debussche trick or its higher order variants,
where we decompose an unknown
into a sum of a less regular but {\it explicitly known} (random) distribution
and a smoother remainder.
The point of the decomposition \eqref{XY} is that both $X$ and $Y$ are unknown.}
Once we arrive at
the system \eqref{XY3}, 
we can apply the Strichartz estimates for the $X$-equation (Lemma~\ref{LEM:Stri})
and the extra smoothing for the $Y$-equation (Lemma \ref{LEM:smooth})
along with the positivity of $\U$ (Lemma \ref{LEM:posprod})
to construct a solution $(X, Y)$ by a standard contraction argument.

\medskip

We conclude this introduction by stating some remarks
and comments.

\begin{remark}\label{REM:Garban}
\rm

In \cite{Garban}, Garban studied the closely related massless stochastic
nonlinear heat equation with an exponential nonlinearity
on $(\R/\Z)^2$:
\begin{align}
\dt X - \frac 1{4\pi} \Dl  X   +   e^{\g X} = \wt \xi, 
\label{NHG}
\end{align}

\noi
where $\wt \xi$ is a space-time white noise on $\R_+ \times(\R/\Z)^2$.
By setting 
\[ u(t, x) =  \frac{1}{\sqrt{2\pi}} X \Big(\frac{t}{2\pi},\frac x{2\pi}\Big)
\qquad  \text{and} \qquad 
 \xi(t, x) = \frac{1}{(2\pi)^\frac{3}{2}}
\wt \xi  \Big(\frac{t}{2\pi},\frac x{2\pi}\Big), 
\]

\noi
we see that $\xi$ is a space-time white noise on $\R_+\times \T^2$
and
that $u$ satisfies the massless equation \eqref{NH2a} with coupling constant
\[\be = \sqrt {2\pi} \g.\]

\noi
This provides the conversion of the parameters $\g$ in \cite{Garban} and $\be$ in this paper.

\end{remark}

\begin{remark}\label{REM:massless}\rm
As mentioned before, the massive equation \eqref{NH1} (with $\ld > 0$) arises as the stochastic quantization of the so-called H\o egh-Krohn model \cite{HK, AH74} in Euclidean quantum field theory, while the massless model \eqref{NHG} treated in \cite{Garban} comes from the stochastic quantization of Liouville Conformal Field Theory (LCFT). In \cite{ORTW}, with N.~Tzvetkov, we extended the results of this paper on the stochastic nonlinear heat equation \eqref{NH2a}  on the torus $\T^2$ to the case of a massless stochastic nonlinear heat equation with ``punctures" on any closed Riemannian surface, thus addressing properly the stochastic quantization of LCFT. See Theorem 1.4 in~\cite{ORTW}. We point out that the corresponding problem in the hyperbolic case, i.e.~the massless analogue of Theorem~\ref{THM:wave} for the ``canonical" stochastic quantization of LCFT, was not treated in \cite{ORTW} and remains open. 
See also Remark 4.4 in \cite{ORTW}.
\end{remark}

\begin{remark}[stochastic quantization of the $\cosh(\Phi)_2$-model]\label{REM:cosh}\rm
%\quad 
The parabolic sinh-Gordon equation \eqref{NH1x} formally preserves 
(a renormalized version of) the 
Gibbs measure of the form:
\begin{align*}
\text{``}d \rho_\text{sinh} = 
Z^{-1} e^{- E_\text{sinh}(u)} (u) du\text{''}, 
%\label{ZZ1}
\end{align*}

\noi
associated with the energy:
\begin{align*}
E_\text{sinh}(u) = \frac 12 \int_{\T^2} |\jb{\nb} u|^2 dx + 
 \int_{\T^2} \cosh (\be u) dx.
%\label{ZZ2}
\end{align*}

In view of Proposition \ref{PROP:Ups}, 
we can proceed as in the proof of 
Proposition \ref{PROP:Gibbs} 
and construct the renormalized Gibbs measure $\rho_\text{sinh}$
as a limit of the truncated Gibbs measure:
\begin{align}
 d\rho_{\text{sinh}, N}
= Z_N^{-1} \exp\bigg( - C_N \int_{\T^2} \cosh (\be \Q_N u) \bigg)d\mu_1
\label{ZZ3}
\end{align}

\noi
for $0 < \be^2 < 4\pi$, 
where $\mu_1$ is the massive Gaussian free field defined in \eqref{Gauss1}
and $C_N$ is the renormalization constant 
defined in  
\eqref{C1} %with $\s_N$ in \eqref{sig}
%
%Theorem \ref{THM:1}
but with $\Q_N$ instead of $\P_N$.

As in the case of the truncated SNLH \eqref{NH4}, 
it is easy to see that the truncated Gibbs measure $\rho_{\text{sinh}, N}$
in \eqref{ZZ3} is invariant under the  following truncated
sinh-Gordon equation:\begin{align}
\dt u_N + \tfrac 12 (1- \Dl)  u_N   +  \tfrac{1}{2} \be C_N \Q_N \sinh (\be \Q_N u_N)  =  \xi. 
\label{NH1xx}
\end{align}

\noi
Since the equation \eqref{NH1xx}
does not enjoy any sign-definite structure, we can not apply 
(the proof of) Theorem \ref{THM:1.2}.
On the other hand, our proof of Theorem~\ref{THM:1} 
 is  applicable
 to study~\eqref{NH1xx}, 
yielding local well-posedness of~\eqref{NH1xx}
%(in the sense of Theorem~\ref{THM:1})
for the range 
$0 < \be^2 < \frac{8 \pi}{3 + 2 \sqrt 2}\simeq 1.37 \pi$.
The key point is that, unlike \cite[Theorem 1.11]{Garban}, 
this local well-posedness result
yields convergence of the solution $u_N$ of 
the truncated sinh-Gordon equation \eqref{NH1xx}
to some limit~$u$.
Combining this local well-posedness result
with Bourgain's invariant measure argument \cite{BO94, BO96}, 
we then obtain 
almost sure global well-posedness 
for the parabolic sinh-Gordon equation \eqref{NH1x}
and 
invariance of the renormalized Gibbs measure $\rho_\text{sinh}$
in the sense of 
Theorem \ref{THM:4}.

Note that these results 
for the sinh-Gordon equation hold only in the parabolic setting
since,  
when $\ld < 0$, 
 we do not know how to handle SdNLW \eqref{NW1}
 for any $\be^2 > 0$.

\end{remark}

\begin{remark}\label{REM:SW}\rm
(i)  In Theorem \ref{THM:1}, we treat initial data 
$u_{0, N}$ of the form \eqref{IV1}.
Due to the presence of the random part $\P_N w_0$ of the initial data, 
the variance $\sh$ in \eqref{sig} is time-independent, which results
in the time-independent renormalization constant $C_N$ in Theorem \ref{THM:1}.
It is, however, possible to treat deterministic initial data $u_{0, N} = v_0 \in L^\infty(\T^2)$.
In this case, the associated truncated stochastic convolution $ \wt  \Psi^\text{heat}_N$
is given by 
\begin{align*}
 \wt \Psi^\text{heat}_N (t) =  \int_0^t P(t-t') \P_N  dW(t') 
\end{align*}

\noi
whose variance $\wt  \s^\text{heat}_N$ is now time-dependent and given by 
\begin{align}
\begin{split}
\wt  \s^\text{heat}_N (t) &  = \E \big[ \wt \Psi^\text{heat}_N (t, x)^2\big]
 =  \frac{1}{4\pi^2}\sum_{n \in \Z^2} \chi_N^2(n) \frac{1- e^{-t\jb{n}^2}}{ \jb{n}^2}\\
 &\approx -\frac1{2\pi}\log N^{-1} + \frac1{2\pi}\log(\sqrt{t}\vee N^{-1})= \frac1{2\pi}\log\big(1\vee \sqrt{t}N\big),
\end{split}
\label{sig2}
\end{align}

\noi 
where $A\vee B = \max(A, B)$.
Here, the third step of \eqref{sig2} follows from Lemmas~\ref{LEM:Green} 
and \ref{LEM:GreenQ} below, by viewing $e^{t(\Dl-1)}$ as a regularization operator $\Q_N$ with 
a regularizing parameter $t^{-\frac12}$. By comparing \eqref{sig} and \eqref{sig2}, 
we see that 
$\wt \s^\text{heat}_N (t) <  \sh$, 
which allows us to establish an analogue
of Proposition \ref{PROP:Ups} in this case.
As a result, we obtain an analogue of Theorem~\ref{THM:1}
but with a time-dependent renormalization constant.
A similar comment applies to Theorem \ref{THM:3} in the wave case.

\smallskip
\noi
(ii) 
In \cite{ORSW}, the authors (with P.\,Sosoe)
studied 
the (undamped) stochastic hyperbolic sine-Gordon equation
on $\T^2$:
\begin{align}
\dt^2 u + (1- \Dl)  u   +  \ld \sin(\be u) = \xi.
\label{SSG}
\end{align}

\noi
Due to the undamped structure, 
the variance of the truncated stochastic convolution $\Psi_N(t, x)$
behaves like $\sim t \log N$; compare this with \eqref{sigW} and \eqref{sig2}.
This time dependence allows us to make
the variance as small as we like for any 
 $\be^2 > 0$
 by taking $t > 0$ sufficiently small.
As a result, we proved local well-posedness
of the renormalized version of \eqref{SSG}
for any $\be^2 > 0$, with a (random) time of existence $T\les  \be^{-2}$.

Similarly, 
if we consider the undamped stochastic nonlinear wave equation (SNLW)
with an exponential nonlinearity:
\begin{align}
\dt^2 u +   (1- \Dl)  u   +  \ld \be e^{\be u} = \sqrt{2} \xi, 
\label{NW7}
\end{align}

\noi
then we see that  Proposition~\ref{PROP:Ups} holds
with  the regularity $\al$ 
given by \eqref{al} with $\be^2$ replaced by $\be^2 T$.
Thus, given any $\be^2 > 0$, 
we can make $\al > 0$ arbitrarily small by taking $T\sim \be^{-2}> 0$ small.
See also Proposition 1.1 in \cite{ORSW}.
This allows us to prove local well-posedness
of SNLW \eqref{NW7}
for \textit{any} $\be^2 > 0$. 

Note that in view of \eqref{sig2}, due to the exponential convergence to equilibrium for the linear stochastic heat equation, we have 
$\wt \s^\text{heat}_N(t)\sim \s_N$ as soon as $t\ges N^{-2+ \ta}$
for some (small) $\ta > 0$, and thus the regularization effect as in the wave case can only be captured at time scales $t\ll N^{-2 + \ta}$, which prevents us from building a local solution 
 with deterministic initial data
for arbitrary $\be^2>0$ in the parabolic case.

\end{remark}

\begin{remark}\label{REM:elliptic}\rm
As we mentioned above, 
in the recent work \cite{ADG}, 
Albeverio, De Vecchi and Gubinelli
 investigated the elliptic analogue of \eqref{NH1} and \eqref{NW1}, namely 
the authors  studied the following singular elliptic SPDE:
\begin{align}
(1-\Dl_{x,z})\phi + \al :\!e^{\al \phi}\!:\,\, = \xi
\label{ADG1}
\end{align}

\noi
for $\phi:(x,z)\in\R^2\times\R^2 \mapsto\phi(x,z)\in\R$, 
where 
 $\xi$ is a space-time white noise on $\R^4$.
% and $g$ a non-negative smooth compactly supported function. 
Here, due to scaling considerations, the coupling constant corresponds to $\al = 2\sqrt{\pi}\be$. The authors of \cite{ADG} then proved that 
\eqref{ADG1}  is well-posed in the regime $0<\al^2<\al_{\max}^2 =4(8-4\sqrt{3})\pi\cdot(4\pi)$;
see \cite[Theorem 25 and Proposition 36]{ADG}. In particular,  note that $\bh^*=\sqrt{4\pi}<\frac{\al_{\max}}{2\sqrt{\pi}}<\sqrt{8\pi}$. Their proof also relies 
on the Da Prato-Debussche trick, 
writing $\phi$ as  $\phi = (1-\Dl)^{-1}\xi + \cj\phi$ and solving the corresponding elliptic PDE for the nonlinear component $\cj\phi$. One of the benefits of the elliptic setting is that, due to the dimension being $d=4$, the $L^2$-regime corresponds to $0<\al^2 < 8\pi\cdot (4\pi)$, namely to the full sub-critical regime $0<\be^2<8\pi$ for the reduced coupling constant $\be = \frac{\al}{2\sqrt{\pi}}$. 
This  in particular  yields  an analogue of Proposition~\ref{PROP:Ups} for the (elliptic) Gaussian multiplicative chaos $\,:\!e^{\al (1-\Dl)^{-1}\xi}\!:\,$ in the entire range $0<\al^2<8\pi\cdot (4\pi)$ for which the construction of the $\exp(\Phi)_2$-measure holds, by just working in $L^2$-based Sobolev spaces. 
See \cite[Lemma 22]{ADG}.
Note that the same approach here only gives the convergence of $\U_N$ for $0<\be^2<4\pi$. 
The well-posedness of the elliptic SPDE \eqref{ADG1} then follows from an argument similar as that in Section~\ref{SEC:WP} adapted to the elliptic setting. 
Heuristically speaking, this should provide well-posedness in the whole range $0<\al^2 < 8\pi\cdot(4\pi)$.
However,   there seems to be  an issue similar to that discussed after \eqref{FP0}.
Namely,  $\cj\phi$ does not have sufficient regularity to use an analogue of the condition\,(i) in Lemma~\ref{LEM:posprod} for bounding the product of a distribution and a measure, 
which instead forces the use of an analogue of condition (ii) in Lemma~\ref{LEM:posprod}.
This in turn restricts the range of admissible $\al^2>0$.
\end{remark}

\begin{remark}\label{REM:Hoshino}\rm
(i) In \cite{HKK}, 
Hoshino, Kawabi, and Kusuoka
studied SNLH~\eqref{NH1} with $\ld = 1$
and independently established Theorem \ref{THM:1.2}
and Theorem \ref{THM:2}.
While the analytical part of the argument is analogous,\footnote{The sign-definite structure 
of the equation in the defocusing case also plays
an important role in \cite{HKK}.
See, for example,  the proof of Lemma 3.10 in \cite{HKK}.} 
 the approaches for studying the Gaussian multiplicative chaos $\U_N$
(\cite[Theorem 2.4]{HKK} and  Proposition \ref{PROP:Ups} above)
are quite different.
The proof in \cite{HKK} is based on the Fourier side approach as in \cite{MWX, GKO}, 
establishing only the second moment bound.
On the other hand, 
our argument  is based on the physical side approach as in our previous work \cite{ORSW, ORSW2} on the hyperbolic sine-Gordon model.
By employing the  Brascamp-Lieb inequality
(and Kahane's convexity inequality), 
we also obtain higher moment bounds on the Gaussian multiplicative chaos, 
which is a crucial ingredient 
to prove Theorem~\ref{THM:1} for SNLH~\eqref{NH1}
with general $\ld \in \R\setminus\{0\}$
and Theorem \ref{THM:3} on SdNLW \eqref{NW1}.

After the submission of this paper, the same authors proved  well-posedness and invariance of the Gibbs measure for the parabolic SPDE \eqref{NH1} in the full ``$L^1$" regime $0<\be^2 < 8\pi$;
see~\cite{HKK2}. This relies on arguments similar to those presented in Section~\ref{SEC:WP} but working in $L^p$-based spaces with $1<p<2$ instead of the $L^2$-based Sobolev spaces used in the proof of Theorem~\ref{THM:2}. In particular,  this requires extending the convergence part of Proposition~\ref{PROP:Ups} to the case $1<p<2$.

\smallskip

\noi
(ii) In a recent preprint \cite{Robert}, the second author
studied 
 the fractional nonlinear Schr\"odinger equation with an exponential nonlinearity
   on a $d$-dimensional compact Riemannian manifold:
 \begin{align*}
  i \dt u + (-\Dl)^\frac{\al}{2}  + \ld \be e^{\be|u|^2} u = 0
%\label{NLS}
\end{align*}

\noi
  with the dispersion parameter $\al > d$.
In the defocusing case ($\ld > 0$), under some assumption, the author proved almost sure global well-posedness
and invariance of the associated Gibbs measure.
See \cite{Robert} for precise statements.
In the focusing case ($\ld < 0$), it was shown that the Gibbs measure is not normalizable
for any $\be > 0$.
See also Remark \ref{REM:focusing}. %\cite{ORTW}.
Our understanding of the Schr\"odinger problem, however, is 
far from being satisfactory at this point
and it is of interest to investigate further issues in this direction.

\end{remark}

This paper is organized as follows.
In Section \ref{SEC:toolbox}, we introduce notations and state various tools
from deterministic analysis.
In Section \ref{SEC:Gauss}, 
we study the regularity and convergence properties
of the Gaussian multiplicative chaos (Proposition \ref{PROP:Ups}).
In Section \ref{SEC:fpa}, 
we prove local well-posedness of SNLH \eqref{NH1}
for general $\ld \in \R\setminus\{0\}$
(Theorem~\ref{THM:1}).
In Section~\ref{SEC:WP}, we discuss the $\ld > 0$ case
for SNLH \eqref{NH1}
and present proofs of Theorems~\ref{THM:1.2} and~\ref{THM:2}.
Section~\ref{SEC:wave}
is devoted to the study of SdNLW \eqref{NW1}.
In 
Appendix \ref{SEC:A}, 
we present a simple contraction argument
to prove local well-posedness of  SNLH \eqref{NH5a}
 for any $\ld \in \R\setminus \{0\}$, 
in  the range 
$0 < \be^2 < \frac 43 \pi \simeq 1.33 \pi$
{\it without} using the positivity of the 
  Gaussian multiplicative chaos.
Lastly, in Appendix \ref{SEC:B}, 
we present a proof of
Lemma~\ref{LEM:GMCmom},
which is 
crucial in establishing moment bounds for the Gaussian multiplicative chaos.

\section{Deterministic toolbox}\label{SEC:toolbox}
In this section, 
we introduce some notations and go over preliminaries
from deterministic analysis.
In Subsections~\ref{SUBSEC:Green},~\ref{SUBSEC:heatker},  and~\ref{SUBSEC:waveker}, 
we recall key properties 
of  the kernels of elliptic, heat, and wave equations.
We also state the Schauder estimate (Lemma \ref{LEM:heatker})
and the Strichartz estimates (Lemma \ref{LEM:Stri}).
In Subsection~\ref{SUBSEC:nonlinear}, 
we state other useful lemmas from harmonic and functional analysis.

\subsection{Notations}

We first introduce some notations.
We set
\begin{align}
e_n(x) \stackrel{\textup{def}}{=} \frac1{2\pi}e^{in\cdot x}, 
\qquad  n\in\Z^2, 
\label{exp}
\end{align}
 
\noi 
for the orthonormal Fourier basis in $L^2(\T^2)$. 
Given $s \in \R$, we  define the  Sobolev space  $H^s (\T^2)$ by  the norm:
\[
\|f\|_{H^s(\T^2)} = \|\jb{n}^s\ft{f}(n)\|_{\l^2(\Z^2)},
\]

\noi
where $\ft{f}(n)$ is the Fourier coefficient of $f$ and  $\jb{\,\cdot\,} = (1+|\cdot|^2)^\frac{1}{2}$. 
We also set 
\begin{equation*}
\H^{s}(\T^2)  \stackrel{\textup{def}}{=} H^{s}(\T^2) \times H^{s-1}(\T^2).
\end{equation*}

\noi
Given $s \in \R$ and $p \geq 1$, 
we define
 the  $L^p$-based Sobolev space (Bessel potential space) 
 $W^{s, p}(\T^2)$
 by the norm:
\[\| f\|_{W^{s, p}} = \| \jb{\nb}^s f \|_{L^p} = \big\| \F^{-1}( \jb{n}^s \ft f(n))\big\|_{L^p}.\]

\noi
When $p = 2$, we have $H^s(\T^2) = W^{s, 2}(\T^2)$. 
When we work with space-time function spaces, 
we use short-hand notations such as
 $C_TH^s_x = C([0, T]; H^s(\T^2))$.

For $A, B > 0$, we use $A\lesssim B$ to mean that
there exists $C>0$ such that $A \leq CB$.
By $A\sim B$, we mean that $A\lesssim B$ and $B \lesssim A$.
We also use  a subscript to denote dependence
on an external parameter; for example,
 $A\les_{\al} B$ means $A\le C(\al) B$,
  where the constant $C(\al) > 0$ depends on a parameter $\al$. 
Given two functions $f$ and $g$ on $\T^2$, 
we write  
\begin{align*}
f\approx g
%\label{approx}
\end{align*}

\noi
 if there exist some constants $c_1,c_2\in\R$ such that $f(x)+c_1 \leq g(x) \leq f(x)+c_2$ for any $x\in\T^2\backslash \{0\}
 \cong [-\pi, \pi)^2 \setminus\{0\}$.
Given $A, B \geq 0$, we also set
 $A\vee B = \max(A, B)$
 and
 $A\wedge B = \min(A, B)$.

Given a random variable $X$, we use $\Law(X)$ to denote its distribution.

\subsection{Bessel potential and Green's function}\label{SUBSEC:Green}
In this subsection, we recall several facts about the Bessel potentials and the Green's function for $(1-\Dl)$ on $\T^2$. 
See also Section~2 in \cite{ORSW}.

For $\al>0$, the Bessel potential of order $\al$ on $\T^d$ is  
given by $\jb{\nabla}^{-\al} = (1-\Dl)^{-\frac{\al}2}$ viewed as a Fourier multiplier
operator. Its convolution kernel is given by
 \begin{align}
 J_{\al}(x)  \deff \lim_{N \to \infty}\frac{1}{2\pi} \sum_{n \in\Z^d} \frac{\chi_N(n)}{\jb{n}^{\al}} e_n(x),
 \label{Bessel}
 \end{align}
 where the limit is interpreted in the sense of distributions on $\T^d$. 
 We recall from \cite[Lemma~2.2]{ORSW}  the following local description of these kernels.

 \begin{lemma} \label{LEM:Bessel}
For any $0 < \alpha < d$, 
 the distribution $J_\al$ agrees with an integrable function, 
which is smooth away from the origin.
 Furthermore, there exist a constant $c_{\al, d} >0$ and a smooth function $R$ on $\T^d$ 
 such that 
\begin{align*}%\label{Bessel2}
 J_\al(x) = c_{\al, d} |x|^{\al - d} +R(x)
 \end{align*}

\noi
for all $x\in\T^d\setminus\{0\}
 \cong [-\pi, \pi)^d \setminus\{0\}$.

 \end{lemma}

An important remark is that the coefficient  $c_{\al, d}$ is positive;
see (4,2) in \cite{AS}.
This in particular means that the singular  part of the Bessel potential $J_{\al}$ is positive. 
  We will use this remark in Lemma~\ref{LEM:posprod} 
below to establish a refined product estimate involving positive distributions.

In the following, we focus on $d = 2$.
The borderline case $\al=d = 2$ corresponds 
to the Green's function $G$ for $1 - \Dl$.
On $\T^2$, $G$ is given by 
\begin{align}
G \deff (1-\Delta)^{-1} \dl_0 = \frac{1}{2\pi}\sum_{n\in\Z^2}\frac 1{\jb{n}^{2}}e_n.
\label{G1a}
\end{align}

\noi
It is well known that $G$ is an integrable function, smooth away from the origin, and that it satisfies the asymptotics
\begin{align}
G(x) = -\frac{1}{2\pi}\log|x| + R(x), \qquad x\in\T^2\setminus\{0\},
\label{G1b}
\end{align}
for some smooth function $R$ on $\T^2$. See (2.5) in \cite{ORSW}.

We also recall 
the following description of the truncated Green's function $\P_NG$, 
where $\P_N$ is the smooth frequency projector with the symbol $\chi_N$ in \eqref{chi}.
See Lemma 2.3 and Remark 2.4 in \cite{ORSW}. 

\begin{lemma}
\label{LEM:Green}
Let $N_2 \geq N_1 \geq 1$. 
Then, 
we have
\begin{align*}
\P_{N_1}\P_{N_2} G(x) \approx -\frac{1}{2\pi} \log\big(|x|+N_1^{-1}\big)
%\label{A6b}
\end{align*}

\noi
for any $x\in \T^2 \setminus\{0\}$.
Similarly, we have
\begin{align*}
|\P_{N_j}^2 G(x) - \P_{N_1}\P_{N_2} G(x)| 
\les  \big(1 \vee - \log\big(|x|+N_j^{-1}\big)\big) \wedge \big(N_1^{-1}|x|^{-1} \big)
%\label{A6a}
\end{align*}

\noi
for $j = 1, 2$ and any  $x\in \T^2 \setminus\{0\}$.
\end{lemma}

In establishing invariance of the Gibbs measures
(Theorems \ref{THM:2} and \ref{THM:4}), 
we need to consider the truncated dynamics
\eqref{NH4} and \eqref{NW3a}
with the  truncated nonlinearity.
In order to preserve the sign-definite structure, 
it is crucial that we use  the smoothing operator $\Q_N$ defined in \eqref{Q}
with a non-negative kernel.
In particular, we need to construct the Gaussian multiplicative chaos $\U_N$
with the smoothing operator $\Q_N$
in place of $\P_N$.
For this purpose, we state an analogue of Lemma \ref{LEM:Green}
 for the truncation of the Green's function by $\Q_N$.

\begin{lemma}\label{LEM:GreenQ}
Let $N_2\ge N_1\ge 1$. Then we have
\begin{align}
\Q_{N_1}\Q_{N_2}G(x) \approx -\frac1{2\pi}\log\big(|x|+N_1^{-1}\big)
\label{A7b}
\end{align}

\noi
for any $x\in \T^2 \setminus\{0\}$.
Similarly, we have
\begin{align*}
\big|\Q_{N_j}^2G(x) - \Q_{N_1}\Q_{N_2}G(x)\big|\les \big(1 \vee - \log\big(|x|+N_j^{-1}\big)\big) \wedge \big(N_1^{-1}|x|^{-1} \big)
%\label{A7a}
\end{align*}

\noi
for $j = 1, 2$ and any  $x\in \T^2 \setminus\{0\}$.

\end{lemma}

\begin{proof}
We mainly follow the proof of Lemma 2.3 in \cite{ORSW}. 
We only show \eqref{A7b}  for $N_1=N_2=N$, 
since the other claims follow
from  a straightforward modification. 
Fix $x\in\T^2 \setminus\{0\}\cong [-\pi,\pi)^2\setminus\{0\}$.

\smallskip
\noi
$\bullet$
{\bf Case 1:} 
We first treat the case $|x|\lesssim N^{-1}$. 
Since $\rho\in C^{\infty}_c(\R^2)$, 
we have 
\begin{align}
|\partial_\xi ^{k}\ft \rho_N(\xi)|\les N^{-|k|}\jb{N^{-1}\xi}^{-\ell}
\label{decay1}
\end{align}
 for any $k\in (\Z_{\ge 0})^2$, $\ell\in\N$, 
 and $\xi \in \R^2$.
Then, by   \eqref{G1a}, the mean value theorem, 
and \eqref{decay1} with $|k| = 0$ and $\l = 2$, we have
\begin{align}
\begin{split}
\big|\Q_N^2G(x)-\Q_N^2G(0)\big| 
&= 2\pi \bigg|\sum_{n\in\Z^2}\frac{\ft\rho_N(n)^2}{\jb{n}^{2}}(e_n(x)-e_n(0))\bigg|\\
&\les \sum_{n\in\Z^2}\frac{\ft\rho_N(n)^2}{\jb{n}}|x|
\les \sum_{|n|\le N}\jb{n}^{-1}|x|+\sum_{|n|\ge N}N^2|n|^{-3}|x| \\
& \les N|x|\les 1.
\end{split}
\label{decay2}
\end{align}

\noi
Similarly, 
by  \eqref{G1a}, the mean value theorem with $\ft\rho_N(0) = \frac{1}{2\pi}$, 
and \eqref{decay1} with $\l = 1$, we have
\begin{align}
\begin{split}
\bigg|\Q_N^2G(0)-\frac1{4\pi^2}
\sum_{|n|\le N}\frac{1}{\jb{n}^{2}}\bigg| 
& \les \bigg|\sum_{|n|\le N}\frac{4\pi^2 \ft\rho_N(n)^2-1}{\jb{n}^2}\bigg| 
+  C \sum_{|n|\ge N}\frac{N}{\jb{n}^{2}|n|}\\
&\les \sum_{|n|\le N}\frac{|n|}{N\jb{n}^{2}} + 1\\
&\les 1.
\end{split}
\label{decay3}
\end{align}

\noi
 Hence, from \eqref{decay2} and \eqref{decay3}, 
  we conclude that
\begin{align*}
\Q_N^2G(x)\approx \frac{1}{4\pi^2}\sum_{|n|\le N}\frac{1}{\jb{n}^{2}}
 \approx \frac{1}{2\pi}\log N \approx -\frac1{2\pi}\log\big(|x|+N^{-1}\big),
%\label{A2}
\end{align*}

\noi
where we used Lemma 3.2 in \cite{HRW} at the second step.

\smallskip

\noi
$\bullet$
{\bf Case 2:} 
Next, we consider the case $|x|\gg N^{-1}$. Since $G$ is integrable and $\rho_N$ is non-negative and integrates to 1, we have
\begin{align}
\begin{split}
\big|\Q_N^2G(x)-G(x)\big|&= \Big|\int_{\T^2}\int_{\T^2}\rho_N(x-y)\rho_N(y-z)\big(G(z)-G(x)\big)dzdy\Big|\\
&\les \int_{\T^2}\int_{\T^2}\rho_N(x-y)\rho_N(y-z)\bigg|\log\bigg(\frac{|z|}{|x|}\bigg)\bigg|dzdy + 1,
\end{split}
\label{decay4}
\end{align}

\noi
where, at the second step,  
we used \eqref{G1b} and the fact that $R$ in \eqref{G1b} is smooth.
Since  $\rho_N$ is supported in a ball of radius $O(N^{-1})$ centered at 0, 
 we have $|x-z|\les |x-y|+|y-z| \les N^{-1}$
 in the above integrals, 
 which 
  implies that $|x|\sim |z|$ under the assumption $|x|\gg N^{-1}$. 
  Hence, 
   the $\log$ term in \eqref{decay4} is bounded
   and we obtain
\begin{align}
\big|\Q_N^2G(x)-G(x)\big|\les \int_{\T^2}\int_{\T^2}\rho_N(x-y)\rho_N(y-z)dzdy + 1 \sim 1.
\label{decay5} 
\end{align}

\noi
Therefore, from \eqref{G1b} and \eqref{decay5}, 
we have
\[\Q_N^2G(x) \approx G(x)\approx -\frac{1}{2\pi}\log|x|\approx -\frac{1}{2\pi}\log\big(|x|+N^{-1}\big).\]
This concludes the proof of Lemma \ref{LEM:GreenQ}.
\end{proof}

\subsection{On the heat kernel and the Schauder estimate}\label{SUBSEC:heatker}

In this subsection, we summarize the properties
of the linear heat propagator $P(t)$
defined in \eqref{heat1}.
We denote the  kernel of $P(t)$ by
\begin{align*}
P_t\deff \frac1{2\pi}\sum_{n\in\Z^2}e^{-\frac{t}2\jb{n}^2}e_n.
%\label{Pt2}
\end{align*}

\noi
Then, we have the following lemma
by passing the corresponding result on $\R^2$ to the periodic torus~$\T^2$
via the Poisson summation formula (see \cite[Theorem 3.2.8]{Gra1}).
See also (2.1) in~\cite{ORSW}.

\begin{lemma}
\label{LEM:heatker}
Let $t>0$.\\
\noi
\textup{(i)} $P_t$ is a positive smooth function.

\smallskip

\noi
\textup{(ii)} Let $\al \geq 0$ and $1\le p\le q\le \infty$.
Then,  we have
\begin{align}
\big\| P(t) f \big\|_{L^q (\T^2)} \les t^{-\frac{\al}2-(\frac1p-\frac1q)} \|\jb{\nabla}^{-\al}f\|_{L^p(\T^2)}
\label{decay6}
\end{align}

\noi
 for any $f\in L^p(\T^2)$.

\end{lemma}

\begin{proof}
By the Poisson summation formula, 
and the positivity of  the heat kernel on $\R^2$, we have
\begin{align*}
P_t =  \frac{1}{2\pi} e^{-\frac{t}{2}}\sum_{n\in\Z^2}e^{-\frac{t}2|n|^2}e_n
=   \frac{1}{2\pi} e^{-\frac{t}{2}}\sum_{n\in\Z^2}\F^{-1}_{\R^2} \big(e^{-\frac{t}2|\cdot |^2}\big)(x + 2\pi n)
>0, 
\end{align*}

\noi
where $\F^{-1}_{\R^2}$ denotes the inverse Fourier transform on $\R^2$.
This proves (i).

The Schauder estimate on $\R^2$ follows
from Young's inequality and 
estimating the kernel on $\R^2$ in some Sobolev norm.
As for the Schauder estimate \eqref{decay6} on $\T^2$, 
we apply Young's inequality
and then use the Poisson summation formula
to pass an estimate on (fractional derivatives of) the heat kernel on $\T^2$
to that  in a weighted Lebesgue space on $\R^2$.
This proves (ii).
\end{proof}

\subsection{On the kernel of the wave operator 
and the Strichartz estimates}
\label{SUBSEC:waveker}

Next, we turn our attention to the linear operators
for the (damped) wave equations.
Let $S(t)$ 
be the forward propagator for the standard wave equation
defined in \eqref{S1}.
We denote its kernel by $S_t$, which can be written as the following distribution:
\begin{align*}
S_t\deff \frac1{2\pi}\sum_{n\in\Z^2}\frac{\sin(t|n|)}{|n|}e_n,
%\label{S2}
\end{align*}

\noi
where we set $ \frac{\sin(t|0|)}{|0|} = t$  by convention.

We say that a distribution $T$ is positive  if its evaluation $T(\varphi)$ at any non-negative test function
$\varphi$ is non-negative.
 We  have the following positivity result for $S_t$.
\begin{lemma}\label{LEM:waveker}
For any $t\geq 0$, the distributional kernel $S_t$ 
on the two-dimensional torus $\T^2$ is positive.
\end{lemma}

\begin{proof}
As a distribution, we have 
\[S_t = \frac{1}{2\pi}
\sum_{n\in\Z^2}\frac{\sin(t|n|)}{|n|}e_n = \lim_{N\rightarrow\infty}\sum_{n\in\Z^2}\ft \rho_N(n) \frac{\sin(t|n|)}{|n|}e_n,\]
where $\rho_N$ is as in \eqref{Q}. In particular, 
we can use the Poisson summation formula to write
\begin{align}
\begin{split}
S_t(x)  & =\lim_{N\rightarrow \infty}
\frac{1}{2\pi} \sum_{m\in\Z^2}
\int_{\R^2}\ft \rho_N(\xi)\frac{\sin(t|\xi|)}{|\xi|}e^{i(x+2\pi m)\cdot\xi}d\xi\\
& = \lim_{N\rightarrow \infty}\sum_{m\in\Z^2}
\frac{\sin(t|\nabla|)}{|\nabla|}\rho_N(x+2\pi m).
\end{split}
\label{decay7}
\end{align}

Let  $u_N$  be the solution to the following linear wave equation on $\R^2$:
\begin{align}
\begin{cases}
\dt^2u_N -\Dl u_N = 0,\\
(u_N,\dt u_N)\big|_{t=0}=(0,\rho_N).
\end{cases}
\label{LW}
\end{align}

\noi
It is well known (see, for example,  (27) on p.\,74 in \cite{Evans}) 
that in the two-dimensional case, 
 the solution $u_N$ to \eqref{LW} is given by the following Poisson's formula:
\begin{align*}
u_N(t,x)=\frac1{2\pi } 
\int_{B(x, t)}\frac{\rho_N(y)}{\sqrt{t^2-|x-y|^2}}dy\geq 0
%\label{Poisson2}
\end{align*}

\noi
for any $x\in\R^2$ and $t\ge 0$, 
where $B(x, t) \subset \R^2$ is the ball of radius $t$ centered at $x$ in $\R^2$. 
Hence, from \eqref{decay7}, 
we conclude that 
\begin{align}
S_t(x)  
& = \lim_{N\rightarrow \infty}\sum_{m\in\Z^2}
u_N(x+2\pi m) \geq 0.
\label{decay8}
\end{align}

\noi
We point out that the sum in \eqref{decay8} (for fixed $N \in \N$)
is convergent thanks to the compact support of $\rho_N$
and the finite speed of propagation for the wave equation.
\end{proof}

The next lemma shows that the operators $\D(t)$ 
in \eqref{D} and $e^{-\frac{t}2}S(t)$ in \eqref{S1} are close in the sense that their difference 
provides an extra smoothing property.
 This extra smoothing  plays  a crucial role
 for estimating  $Y$ in \eqref{XY3}.
\begin{lemma}\label{LEM:smooth}
Let $t\geq 0$ and $s\in\R$.

\smallskip

\noi
\textup{(i)} The operator $\D(t)-e^{-\frac{t}{2}}S(t)$ is bounded from $H^s(\T^2)$ to $H^{s+2}(\T^2)$.

\smallskip

\noi
\textup{(ii)} The operator $\dt\big(\D(t)-e^{-\frac{t}{2}}S(t)\big)$ is bounded from $H^s(\T^2)$ to $H^{s+1}(\T^2)$.
\end{lemma}
\begin{proof}
(i)  It suffices to show that the symbol of
$\jb{\nabla}^2\big(e^\frac{t}{2}\D(t)-S(t)\big)$ is bounded. 
Since $\jb{n}\sim \sqrt{\tfrac34+|n|^2}$ for any $n\in\Z^2$, it suffices to bound, for $n\neq 0$, 
\begin{align*}
\big(\tfrac34 + |n|^2\big)
& \Bigg(\frac{\sin\big(t\sqrt{\tfrac34+|n|^2}\big)}{\sqrt{\tfrac34+|n|^2}}-\frac{\sin(t|n|)}{|n|}\Bigg)\\
&=\sqrt{\tfrac34+|n|^2}\bigg(\sin\Big(t\sqrt{\tfrac34+|n|^2}\Big)-\sin(t|n|)\bigg)\\
&\hphantom{X}
 + (\tfrac34+|n|^2)\sin(t|n|)\Bigg(\frac{1}{\sqrt{\tfrac34+|n|^2}}-\frac{1}{|n|}\Bigg)\\
&=: \1+\II.
\end{align*}

By the mean value theorem, we have
\[|\1|\les \jb{n}\Big|\sqrt{\tfrac34+|n|^2}-|n|\Big| 
\les \jb{n}\frac{1}{\sqrt{\tfrac34+|n|^2}+|n|}\les 1.\]

\noi
Similarly, we can bound the second term by
\[|\II| \les \jb{n}^2 \frac{1}{|n|\sqrt{\tfrac34+|n|^2}\big(|n|+\sqrt{\tfrac34+|n|^2}\big)}\les 1.\]

\noi
This proves (i).

\smallskip

\noi
(ii) In this case, we show the boundedness of the symbol
for 
\begin{align*}
\jb{\nb}&  \dt\big(\D(t)-e^{-\frac{t}{2}}S(t)\big)\\
&
= -\frac 12 \jb{\nabla}\big(\D(t) - e^{-\frac{t}{2}}S(t)\big) 
+ e^{-\frac{t}2}\jb{\nb}
\bigg( \cos\Big(t\sqrt{\tfrac34-\Dl}\Big)-\cos(t|\nabla|)\bigg)\\
&  =: \III + \IV.
\end{align*}

\noi
The symbol of $\III$ is clearly  bounded by the argument above.
As for the symbol of $\IV$, it follows from 
the mean value theorem that 
\begin{align*}
\jb{n}\bigg[\cos\Big(t\sqrt{\tfrac34+|n|^2}\Big)-\cos(t|n|)\bigg] \les \jb{n}\Big(\sqrt{\tfrac34+|n|^2}-|n|\Big) \les 1.
\end{align*}

\noi
This completes the proof of Lemma \ref{LEM:smooth}.
\end{proof}

Next, we state the Strichartz estimates
for the linear wave equation.

 \begin{definition}\label{DEF:pair}
Given $0<s<1$, we say that a pair $(q, r)$ of exponents \textup{(}and a pair $(\wt q,\wt r)$, respectively\textup{)}
 is  \emph{$s$-admissible} \textup{(}and \emph{dual $s$-admissible}, respectively\textup{)}, 
  if $1\leq \wt q\leq 2\leq q\leq\infty$ and $1<\wt r\leq 2 \leq r<\infty$ and if they satisfy the 
  following scaling and admissibility conditions:
\[\frac1q + \frac2r = 1-s = \frac1{\wt q}+\frac2{\wt r} -2, 
\qquad  
\frac2q+\frac1r\leq\frac12,
\qquad \text{and}\qquad  \frac2{\wt q}+\frac1{\wt r}\geq \frac52.\]
\end{definition}

Given $\frac 14 < s < \frac 34$, we fix  the following $s$-admissible and dual $s$-admissible pairs:
\begin{align}
(q,r)=\bigg(\frac3s,\frac6{3-4s}\bigg) \qquad\text{and}\qquad 
(\wt q,\wt r)=\bigg(\frac3{2+s},\frac6{7-4s}\bigg).
\label{pairs}
\end{align}

\noi
In Section \ref{SEC:wave}, 
we will only use these pairs.

Let  $0<T\leq 1$, $\frac 14<s<\frac{3}{4}$
and fix the $s$-admissible pair $(q,r)$ and the dual $s$-admissible pair $(\wt q,\wt r)$ given in \eqref{pairs}.
We then define the Strichartz space:
\begin{align}\label{StriX}
\X^s_T=C([0,T];H^s(\T^2))\cap C^1([0,T];H^{s-1}(\T^2))\cap L^q([0,T];L^r(\T^2))
%\label{XX}
\end{align}
and its ``dual'' space:
\begin{align}\label{StriN}
\NN^s_T = L^1([0,T];H^{s-1}(\T^2)) + L^{\wt q}([0,T];L^{\wt r}(\T^2)).
%\label{NN}
\end{align}

We now state the Strichartz estimates.
The Strichartz estimates on $\R^d$ are well-known;
see \cite{GV, LS, KeelTao}. 
Thanks to the finite speed of propagation, 
 the same estimates  also hold on~$\T^d$
 locally in time.

\begin{lemma}\label{LEM:Stri} 
The solution $u$ to the linear wave equation:
\[\begin{cases}
\dt^2u-\Dl u = F\\
(u,\dt u)|_{t=0}=(u_0,u_1)
\end{cases}\] 

\noi
satisfies the following Strichartz estimate:
\begin{align*}
\|u\|_{\X^s_T} \les \|(u_0,u_1)\|_{\H^s} + \|F\|_{\NN^s_T}, 
%\label{Stri}
\end{align*}

\noi
uniformly in $0 < T \leq 1$.

\end{lemma}

We also recall 
from \cite{GKO}
the following interpolation result 
for 
$\X^s_T$ and $\NN^s_T$.
See  (3.22) and~(3.23) in \cite{GKO}
for the proof.

\begin{lemma}\label{LEM:interpol}
The following continuous embeddings hold:

\smallskip

\noi
\textup{(i)} 
Let $0\le\al\le s$ and $2\leq q_1,r_1 \leq \infty$ satisfy the scaling condition\textup{:} 
\[\frac{1}{q_1} =\frac{1-\al/s}q + \frac{\al/s}\infty \qquad\textup{ and }\qquad\frac{1}{r_1}=\frac{1-\al/s}r+\frac{\al/s}{2}.\]

\noi
Then, we have 
\[\|u\|_{L^{q_1}_TW^{\al,r_1}_x}\les \|u\|_{\X^s_T}.\]

\smallskip

\noi
\textup{(ii)}
Let  $0\le \al \le 1-s$ and $1\leq \wt q_1,\wt r_1 \leq 2$ satisfy
the scaling condition\textup{:}  
\[\frac{1}{\wt q_1} =\frac{1-\al/(1-s)}{\wt q}+\frac{\al/(1-s)}1 \qquad\textup{ and }\qquad\frac{1}{\wt r_1}=\frac{1-\al/(1-s)}{\wt r}+\frac{\al/(1-s)}{2}.\]

\noi
Then, we have 
\[\|u\|_{\NN^s_T}\les \|u\|_{L^{\wt q_1}_TW^{-\al,\wt r_1}_x}.\]
\end{lemma}

\subsection{Some useful results from nonlinear analysis}\label{SUBSEC:nonlinear}
We conclude  this section by presenting some further results from harmonic and functional analysis.

We first state the Brascamp-Lieb inequality \cite{BL}. 
This inequality 
plays an important role in the proof of Proposition \ref{PROP:Ups}.
In particular, it allows us to  
establish a good bound on the $p$th moment
of the Gaussian multiplicative chaos $\U_N$  when $p>2$. 
The version we present here is due to \cite{BCCT}.

\begin{definition}\rm 
We say that 
a pair  $(\B,\cc)$ is 
a \emph{Brascamp-Lieb datum}, 
if, for some  $m \in \N \cup \{0\}$ and  $d, d_1, \dots, d_m \in \N$, 
 $\B = (B_1,...,B_m)$ is a collection of 
 linear maps from $\R^d$ to $\R^{d_j }$, $j = 1, \dots, m$, 
 and  $\cc = (q_1, \dots, q_m) \in \R_+^m$.
\end{definition}

We now state the $m$-linear Brascamp-Lieb inequality.

\begin{lemma}[Theorem 1.15 in \cite{BCCT}]
\label{LEM:BL}
Let $(\B,\cc)$ be a Brascamp-Lieb datum.
Suppose that  the following conditions hold\textup{:}
\begin{itemize}
\item
Scaling condition\textup{:}
\begin{align}
\sum_{j=1}^m q_j d_j = d.
\label{scaling}
\end{align}

\item
Dimension condition\textup{:} for all subspace $V \subset\R^d$, there holds
\begin{align}
\dim (V) \le \sum_{j=1}^m q_j \dim (B_j V).
\label{dimension}
\end{align}
\end{itemize}

\noi
Then,  there exists a positive constant $\textup{BL}(\B,\cc)<\infty$ such that 
\begin{align*}
\int_{\R^d} \prod_{j=1}^m f_j (B_j x)^{q_j}  dx 
\le \textup{BL}({\bf B}, \cc)
 \prod_{j=1}^m \bigg( \int_{\R^{d_j} } f_j (y) \, dy \bigg)^{q_j}
%\label{in:BL}
\end{align*}

\noi
for any non-negative functions $f_j\in L^1(\R^{d_j})$, $j=1,...,m$.

\end{lemma}

We point out that the conditions \eqref{scaling} and \eqref{dimension}
guarantee that the Brascamp-Lieb data is non-degenerate, i.e.~the maps
$B_j$, $j = 1, \dots, m$, are surjective
and their common kernel is trivial.  See \cite[Remarks~1.16]{BCCT}.

For our purpose, we only need the following special version
of Lemma \ref{LEM:BL}.

\begin{corollary}
\label{COR:BL}
Let $p \in \N$.
Then, we have
\begin{align}
\begin{split}
 \int_{(\T^{2})^{2p}}   
 \prod_{1\le j < k \le 2p}   
&  |f_{j,k}(\pi_{j, k} (x))|^{\frac1{2p-1}}
dx \\
 & \les   \prod_{1\le j < k \le 2p}   
 \bigg( \int_{(\T^{2})^2}  
  |f_{j,k}(x_j,x_k)|  dx_j dx_k \bigg)^{\frac1{2p-1}}
\end{split}
\label{BL0}
\end{align}

\noi
for any $f_{j,k} \in L^1 (\T^2 \times \T^2)$.
Here, $\pi_{j, k}$
denotes the projection defined by 
$\pi_{j, k}(x) = \pi_{j, k}(x_1, \dots, x_{2p})
= (x_j, x_k) $
for $x = (x_1, \dots, x_{2p})\in (\T^2)^{2p}$.

\end{corollary}

This is precisely the geometric Brascamp-Lieb inequality
stated in \cite[Example 1.6]{BCCT}.
For readers' convenience, we include its reduction to Lemma \ref{LEM:BL}.

\begin{proof}
Write $(\R^2)^{2p} = \prod_{\l=1}^{2p} \R^2_{\l}$ and define projections 
$\pi_{\l}: (\R^2)^{2p} \to \R^2_{\l}$ and $\pi_{j,k}: (\R^2)^{2p} \to \R^2_j \times \R^2_k$ for $j\neq k$
in the usual way.
Now, we set $\B = (\pi_{j, k}: 1 \leq j < k \leq 2p)$
and 
\[\cc = \bigg(\frac1{2p-1},...,\frac1{2p-1}\bigg) \in \R_+^{p(2p-1)}.\]

\noi
It is also easy to check that the scaling condition \eqref{scaling} holds 
since $d_{j, k} = 4$, $1 \leq j < k \leq 2p$
and   $m= p(2p-1)$ and $q_{j, k} = \frac1{2p-1}$, 
while 
the total dimension is $d = 4p$.

As for the dimension condition \eqref{dimension}, 
first note that 
\begin{align*}
\dim(\pi_{j,k} V)  =  \dim(\pi_{j} V) + \dim(\pi_{k} V)
\end{align*}

\noi
for $j \ne k$.
Then, we have
\[ \dim(V) \leq \sum_{j = 1}^{2p} \dim(\pi_j V)
= \frac{1}{2p-1}\sum_{1 \leq j < k \leq 2p} 
\dim(\pi_{j,k} V),\]

\noi
verifying \eqref{dimension}.

The desired estimate \eqref{BL0} follows from
extending $ f_{j ,k}$ on $(\T^2)^2$
as a compactly supported measurable function on $\R^4$ by extending it by 0 outside of $(\T^2)^2 \simeq [-\pi,\pi)^4$
and applying  
 Lemma \ref{LEM:BL}.
\end{proof}

We now recall several product estimates.
See Lemma 3.4 in \cite{GKO} for the proofs.

\begin{lemma}\label{LEM:prod1}
Let $0 \leq s \le 1$.
\\
\noi \textup{(i)}
Suppose that  $1<p_j,q_j,r<\infty$,  $\frac1{p_j}+\frac1{q_j}=\frac1r$, $j=1,2$.
Then,  we have
\begin{align*}
\big\|\jb{\nabla}^s(fg)\big\|_{L^r(\T^d)} \les \big\|\jb{\nabla}^sf\big\|_{L^{p_1}(\T^d)}\|g\|_{L^{q_1}(\T^d)} + \|f\|_{L^{p_2}(\T^d)}\big\|\jb{\nabla}^sg\big\|_{L^{q_2}(\T^d)}.
\end{align*}

\smallskip

\noi
\textup{(ii)}
Suppose that $1<p,q,r < \infty$ satisfy 
$\frac1p+\frac1q\leq\frac1r + \frac{s}d $.
Then, we have
\begin{align}
\big\| \jb{\nb}^{-s} (fg) \big\|_{L^r(\T^d)} \les \big\| \jb{\nb}^{s} f \big\|_{L^q(\T^d) } 
\big\| \jb{\nb}^{-s} g \big\|_{L^p(\T^d)}. 
\label{prod2}
\end{align}

\end{lemma}

Note that
while  Lemma \ref{LEM:prod1} (ii) 
was shown only for 
$\frac1p+\frac1q= \frac1r + \frac{s}d $
in \cite{GKO}, 
the general case
$\frac1p+\frac1q\leq \frac1r + \frac{s}d $
follows from the inclusion $L^{r_1}(\T^d) \subset  L^{r_2}(\T^d)$ for $r_1 \geq  r_2$.

The next lemma shows that 
an improvement over \eqref{prod2}
in Lemma~\ref{LEM:prod1}\,(ii) is possible
 if~$g$ 
 happens to be  a  positive distribution.

\begin{lemma}\label{LEM:posprod}
Let $0\le s \le 1$ and $1 < p < \infty$.
Then, we have
\begin{align}
\|\jb{\nabla}^{-s}(fg)\big\|_{L^p(\T^d)}\les \|f\|_{L^{\infty}(\T^d)}\|\jb{\nabla}^{-s}g\|_{L^p(\T^d)}
\label{posprod1}
\end{align}

\noi
for any $f \in L^\infty(\T^d)$ and any  positive distribution $g\in W^{-s,p}(\T^d)$, 
 satisfying one of the following two conditions:
\textup{(i)} 
 $f\in C(\T^d)$ 
  or 
 \textup{(ii)}  $f\in W^{s,q}(\T^d)$
 for some  $1< q<\infty$ satisfying $\frac1p+\frac1q < 1 +\frac{s}d$.

 \end{lemma}

 This lemma plays an  important role in estimating a product
 involving the non-negative Gaussian multiplicative chaos $\U_N$.
 In studying continuity in the noise, 
 we need to estimate the difference of 
 the  Gaussian multiplicative chaoses.
 In this case, there is no positivity to exploit and hence we instead apply Lemma \ref{LEM:prod1}\,(ii).

\begin{proof}
 We consider $0<s\le 1$ since the $s = 0$ case corresponds to H\"older's inequality. 
  Since $g$ is a positive distribution, it can be identified with a positive Radon measure on $\T^2$;
  see for example~\cite[Theorem 7.2]{Folland}.
If $f\in C(\T^d)$, then the product $fg$ is a well-defined function in $L^1(\T^d)$. 
  With $\rho_N$ as in \eqref{Q2}, we have
   $f_Ng \deff(\rho_N*f)g \rightarrow fg$ in $L^1(\T^d)$, in particular 
   in the distributional sense.
   Hence, from Fatou's lemma, we have
\begin{align}
\|\jb{\nabla}^{-s}(fg)\|_{L^p}\le \liminf_{N\to\infty}\|\jb{\nabla}^{-s}(f_Ng)\|_{L^p}.
\label{pos1}
\end{align}

Since $\rho_M$ is non-negative, we see that 
 $g_M  = \rho_M*g$ is a well-defined smooth, 
 positive distribution which converges to $g$ in $W^{-s,p}(\T^d)$. 
 Then,  it follows from  Lemma~\ref{LEM:prod1}\,(ii)
 that, for each fixed $N \in \N$, 
   $f_Ng_M$ converges to $f_Ng$ in $W^{-s,p}(\T^d)$ as $M\rightarrow \infty$. 
   Hence,  it suffices to prove \eqref{posprod1} for $f_Ng_M$, $N, M\in\N$.
Indeed, if \eqref{posprod1} holds  for $f_Ng_M$, $N, M\in\N$,
then by~\eqref{pos1}, \eqref{posprod1} for $f_Ng_M$, 
the convergence of $\jb{\nb}^{-s} g_M
 = \rho_M *(\jb{\nb}^{-s} g)$ to $\jb{\nb}^{-s} g$ in $L^p(\T^d)$, 
and Young's inequality with $\|\rho_N\|_{L^1} = 1$, 
we obtain
\begin{align*}
\|\jb{\nabla}^{-s}(fg)\|_{L^p}
& \le 
\liminf_{N\to\infty} \lim_{M\to \infty}\|\jb{\nabla}^{-s}(f_Ng_M)\|_{L^p}\\
& \les 
\liminf_{N\to\infty} \lim_{M\to \infty}
 \|f_N\|_{L^{\infty}}\|\jb{\nabla}^{-s}g_M\|_{L^p}\\
& \le
\liminf_{N\to\infty}
 \|f_N\|_{L^{\infty}}\|\jb{\nabla}^{-s}g\|_{L^p}\\
& \le
 \|f\|_{L^{\infty}}\|\jb{\nabla}^{-s}g\|_{L^p}.
\end{align*}

It remains to prove \eqref{posprod1} for $f_N g_M$.
By Lemma~\ref{LEM:Bessel},  we have
\begin{align*}
\|\jb{\nabla}^{-s}(f_Ng_M)\|_{L^p} 
&= \|J_s*(f_Ng_M)\|_{L^p}\\ 
& \les \bigg\|\int_{\T^d}|x-y|^{s-d}|f_N(y)g_M(y)|dy\bigg\|_{L^p}
\intertext{Since $g_m$ is non-negative,}
&\les \|f_N\|_{L^{\infty}}\Big\||\cdot|^{s-d}*g_M\Big\|_{L^p}
\intertext{Using Lemma \ref{LEM:Bessel} again,}
&\sim \|f_N\|_{L^{\infty}}\Big\|(J_s-R)*g_M\Big\|_{L^p}\\
&\le \|f_N\|_{L^{\infty}}\Big(\|\jb{\nabla}^{-s}g_M\|_{L^p} 
+ \big\|\big(\jb{\nabla}^s R\big)*\big(\jb{\nabla}^{-s}g_M\big)\big\|_{L^p}\Big)\\
&\les \|f_N\|_{L^{\infty}}\|\jb{\nabla}^{-s}g_M\|_{L^p},
\end{align*}

\noi
where in the last step we used the fact that $R$ is smooth. This shows \eqref{posprod1}
for $f_Mg_M$ 
and hence for $f \in C(\T^d)$ and a positive distribution $g \in W^{-s, p}(\T^d)$.

In view of Lemma~\ref{LEM:prod1}\,(ii), 
 the condition (ii)
guarantees that the product operation
$(f, g) \in W^{s, q}(\T^d) \times W^{-s, p}(\T^d)\mapsto 
fg \in W^{-s, 1+\eps}(\T^d)$ for some small $\eps > 0$
is a continuous bilinear map.
Namely, it suffices to prove \eqref{posprod1}
for $f_N g_M = (\rho_N *f)(\rho_M*g)$, 
which we already did above.
This completes the proof of Lemma \ref{LEM:posprod}.
\end{proof}

Next, we recall the following fractional  chain rule from \cite{Gatto}.
The fractional chain rule on $\R^d$ was essentially proved in \cite{CW}.\footnote{As pointed out in \cite{Staffilani}, 
the proof in \cite{CW} needs
a small correction, which  yields the fractional chain  rule in a 
less general context.
See \cite{Kato, Staffilani, Taylor}.}
As for the estimates  on $\T^d$, see~\cite{Gatto}.

\begin{lemma}
\label{LEM:FC}
Let $0 < s < 1$.

\smallskip

\noi
\textup{(i)} Suppose that 
$F$ is a Lipschitz function with Lipschitz constant $K>0$.
Then, for any $1<p<\infty$, we have
\begin{align*}
\big\||\nabla|^sF(u)\big\|_{L^p(\T^d)}\les K\big\||\nabla|^su\big\|_{L^p(\T^d)}.
%\label{chain2}
\end{align*}

\noi
\textup{(ii)}
Suppose that  $F\in C^1(\R)$  satisfies
\[\big|F'\big(\tau x + (1-\tau) y\big)\big|\le c(\tau)\big(|F'(x)|+|F'(y)|\big)\]
for every $\tau\in[0,1]$ and $x,y\in\R$, where $c\in L^1([0,1])$.
Then for $1<p,q,r<\infty$ with $\frac1p+\frac1q=\frac1r$, we have
\begin{align*}
\big\||\nabla|^sF(u)\big\|_{L^r(\T^d)}\les \big\|F'(u)\big\|_{L^p(\T^d)}\big\||\nabla|^su\big\|_{L^q(\T^d)}.
%\label{chain1}
\end{align*}
\end{lemma}

\medskip

Lastly, we state a tool from functional analysis.
The following classical Aubin-Lions lemma \cite{AL}
provides a criterion for  compactness.
See also \cite[Corollary 4 on p.\,85]{Simon}.

\begin{lemma}\label{LEM:AL}
Let $\X_{-1},\X_0,\X_1$ be Banach spaces satisfying the continuous embeddings $\X_{1}\subset \X_0\subset \X_{-1}$ such that the embedding $\X_1\subset \X_0$ is compact. 
Suppose that  $B$ is bounded in $ L^p([0,T];\X_1)$ 
such that $\{\dt u: u\in B\}$ is bounded in $L^q([0,T];\X_{-1})$
for some $T>0$ and finite $p,q\geq 1$.
Then,  $B$ is relatively compact in 
$L^p([0,T];\X_0)$. Moreover,  if $B$ is bounded in $ L^\infty([0,T];\X_1)$
and 
$\{\dt u: u\in B\}$ is bounded in $L^q([0,T];\X_{-1})$ for some $ q> 1$, 
then $B$ is relatively compact in $C([0,T];\X_0)$.
\end{lemma}

\section{Gaussian multiplicative chaos}
\label{SEC:Gauss}
In this section, we establish 
the regularity and convergence properties of the Gaussian multiplicative chaos 
$\U_N = \, :\!e^{\be\Psi_N}\!\!:$ 
claimed in   Proposition \ref{PROP:Ups}, 
where $\Psi_N$ denotes the truncated stochastic convolution
for either the heat equation or the wave equation.
These properties are of central importance for  the study of 
the truncated SNLH \eqref{NH3}
and the truncated SdNLW~\eqref{NW3}.
As in the case of the sine-Gordon model studied in  \cite{HS,ORSW}, 
the main difficulty comes from the fact that  the processes $\U_N$ 
 do not belong to any Wiener chaos of finite order. 
There is, however, a major difference 
from  
the analysis on  the imaginary Gaussian multiplicative chaos 
 $\, :\!e^{i \be\Psi_N}\!:$ studied for the sine-Gordon model in \cite{HS, ORSW}.
 As for  the imaginary Gaussian multiplicative chaos, the regularity depends
 only on the values of $\be^2$. 
 On the other hand, 
 the regularity of $\U_N$
 depends not only on the values of $\be^2$
 but also on the integrability index (either for moments or space-time integrability).
 In particular, for higher moments, the regularity gets
 worse.
 This phenomenon
 is  referred to as intermittency in \cite{Garban}. See Remark \ref{rk:interm} below.

\subsection{Preliminaries}
Since  the definition \eqref{re-non} of $\U_N$ involves polynomials of arbitrarily high degrees,
it seems more convenient to study  $\U_N$ on the physical space, 
as in the case of the sine-Gordon equation \cite{ORSW}, 
rather than in the frequency space as in \cite{GKO}. 
For  this purpose, we first recall the main property of the covariance function: 
\begin{align*}
\G_{N_1, N_2}(t,x-y) \deff \E\big[\Psi_{N_1}(t,x) \Psi_{N_2}(t,y)\big]
%\label{def-covar}
\end{align*}

\noi
for 
the truncated stochastic convolution 
$\Psi_{N_j} = \Psi_{N_j}^\text{heat}$  or $\Psi_{N_j}^\text{wave}$, 
where the truncation may be given 
by the smooth frequency projector $\P_N$
or the smoothing operator $\Q_N$ with a positive kernel defined in~\eqref{Q}.
When $N = N_1 = N_2$, we set 
\[ \G_N = \G_{N, N}.\]

\noi
As stated in Subsection \ref{SUBSEC:1.2}, 
the results in this section hold for both $\P_N$ and $\Q_N$.

The next lemma follows
as a corollary to Lemmas \ref{LEM:Green}
and \ref{LEM:GreenQ}.
See 
Lemma 2.7 in  \cite{ORSW}
for the proof.

\begin{lemma}
\label{LEM:covar}
Let $N_2\ge N_1\ge 1$. Then we have
\begin{align*}
\G_{N_1, N_2} (t,x-y) \approx - \frac{1}{2\pi} \log \big(|x-y| + N_1^{-1}\big)
%\label{bd-GN}
\end{align*}

\noi
 for any $t \geq  0$. 
Similarly, we have
\begin{align}
\big|\G_{N_j}(t, x-y) - \G_{N_1, N_2}(t, x- y)\big|\les \big(1 \vee - \log\big(|x-y|+N_j^{-1}\big)\big) \wedge \big(N_1^{-1}|x-y|^{-1} \big)
\label{bd-GN2}
\end{align}

\noi
for $j = 1, 2$
and $t \geq 0$.

\end{lemma}

\subsection{Estimates on the even moments}

In this subsection,  we prove the following proposition
for  the uniform control on the even moments of 
 the random variables $\U_N(t,x)$ for any fixed $(t, x) \in\R_+\times \T^2$ and $N \in \N$.

\begin{proposition}\label{PROP:var}
Let $0 < \be^2<8\pi$. 
Then, the following statements hold.

\smallskip

\noi
\textup{(i)} For any $t\ge 0$, $x\in\T^2$,  and $N\in\N$, 
we have $\E \big[ | \U_N (t,x)| \big] = 1$,

\smallskip

\noi
\textup{(ii)} Let  $p\ge 2$ be even.
Let  $0<\al<2$ and $(p -1) \frac{\be^2}{4\pi} < \min(1, \al)$.
Then,  for any $T>0$, we have
 \begin{align*}
 \sup_{t\in[0,T],x\in\T^2,N\in\N}\E \Big[ \left| \jb{\nabla}^{-\alpha} \U_N (t,x) \right|^{p} \Big] \le C(T).
 %\label{moment}
 \end{align*}

\smallskip

\noi
\textup{(iii)} 
Let  $0<\al<2$ and $\frac{\be^2}{4\pi}<\min(1, \al)$.
Then,  there exists small $\eps>0$  such that  
\begin{align*}
\sup_{t\in[0,T],x\in\T^2}\E \Big[ \big| \jb{\nabla}^{-\al}\big(\U_{N_1}(t,x)-\U_{N_2}(t,x)\big) \big|^{2} \Big]\le 
C(T) N_1^{-\eps}
%\label{momentdiff}
\end{align*}

\noi
for any $T> 0$ and any $N_2\geq N_1\geq 1$.

\end{proposition}

\begin{proof}

For fixed $(t, x) \in \R_+ \times \T^2$, 
$\Psi_N(t, x)$ is a mean-zero Gaussian random variable with variance $\s_N$.
Hence, 
from the positivity of $\U_N$ and \eqref{re-non}, we have
\[ \E\big[ |\U_N(t, x) |\big]
= e^{-\frac{\be^2}{2} \s_N} \E [e^{\be \Psi_N(t, x)}]
= 1.\]

\noi
This proves (i).

Next, we consider (ii).
Let  $p = 2m$,   $m\in \N$. 
Fix $(t, x) \in [0,T]\times \T^2$.
Recalling $\jb{\nb}^{-\al}f = J_\al * f$, where $J_\al$ is as in \eqref{Bessel}, 
we have 
 \begin{align}
 \begin{split}
 \E \Big[  & \big| \jb{\nabla}^{-\alpha} \U_N (t,x) \big|^{2m} \Big]\\
 & = e^{-m\be^2 \sigma_N }\E \Bigg[ \bigg| \int_{\T^2} J_{\al} (x-y) e^{\be \Psi_N (t,y)} dy \bigg|^{2m} \Bigg] \\
  & = e^{-m\be^2 \sigma_N }  \int_{(\T^{2})^{2m}} 
  \E \bigg[ e^{\be \sum_{j=1}^{2m}  \Psi_N (t,y_j) } \bigg]  \bigg(\prod_{j=1}^{2m}  J_{\al} (x-y_j)\bigg)  d\vec y  \\
     & =  e^{-m\be^2 \sigma_N } \int_{(\T^{2})^{2m}}  \exp\bigg(\frac{\be^2}2 
     \E \bigg[\Big| \sum_{j =1}^{2m}  \Psi_N (t,y_j)  \Big|^2\bigg]\bigg)  \bigg(  \prod_{j=1}^{2m}  J_{\al} (x-y_j)\bigg)  d\vec y.
     \end{split}
     \label{E1}
     \end{align}

  \noi
where $d \vec y = dy_1 \cdots dy_{2m}$
and  we used the fact that 
 $ \sum_{j =1}^{2m}  \Psi_N (t,y_j) $ is a Gaussian random variable
 at the last step.
From the definition \eqref{sig} of $\s_N$
and 
 Lemma \ref{LEM:covar}, we have 
  \begin{align} 
  \begin{split}
 \exp\bigg(\frac{\be^2}2 \E \bigg[\Big| \sum_{j =1}^{2m}  \Psi_N (t,y_j)  \Big|^2\bigg]\bigg) 
 & = e^{m \be^2 \sigma_N} \exp\bigg( \be^2 \sum_{1 \leq  j < k \leq 2m} 
 \E \big[ \Psi_N (t,y_j)  \Psi_N (t,y_k)  \big]\bigg)\\
 & \les
e^{m \be^2 \sigma_N}
\prod_{1\leq j < k \leq 2m}
 \big(|y_j - y_k| + N^{-1}\big)^{- \frac{\be^2 }{2\pi}} .
\end{split}
\label{E2}
 \end{align}

 \noi
Hence, from  \eqref{E1} and \eqref{E2}, we obtain
 \begin{align}
 \begin{split}
 \E \Big[ &  \big| \jb{\nabla}^{-\alpha} \U_N (t,x) \big|^{2m} \Big]  \\
 & \les   \int_{(\T^{2})^{2m}}   
 \bigg(\prod_{1\le j < k \le 2m}   \big(|y_j - y_k| + N^{-1}\big)^{ - \frac{\be^2 }{2\pi}}\bigg)
\bigg(  \prod_{j=1}^{2m}  |J_{\al} (x-y_j)|\bigg)  d\vec y \\
 & =    \int_{(\T^{2})^{2m}}   
 \prod_{1\le j  < k  \le 2m}   
 \frac{|J_\al (x-y_j) 
 J_\al (x-y_k)|^{\frac{1}{2m-1}}}{ \big(|y_j - y_k| + N^{-1}\big)^{ \frac{\be^2 }{2\pi}}}
d \vec y.
\end{split}
 \label{E4}
 \end{align}

\noi
By applying the geometric Brascamp-Lieb inequality (Corollary \ref{COR:BL})
and proceeding as in the proof of Proposition 1.1 in \cite{ORSW}
to bound the resulting integral, 
we then obtain
 \begin{align*}
\text{RHS of }\eqref{E4}
 & \les    \prod_{1\le j  < k \le 2m}  
 \Bigg( \int_{(\T^{2})^2}   
\frac{|J_{\al} (x-y_j) J_{\al} (x-y_k)|}{ \big(|y_j - y_k| + N^{-1}\big)^{(2m-1) \frac{\be^2 }{2\pi}}} dy_j dy_k 
\Bigg)^{\frac1{2m-1}} \notag\\
 & =    \Bigg( \int_{(\T^{2})^2}   
\frac{|J_{\al} (x-y) J_{\al} (x-z)|}{ \big(|y- z| + N^{-1}\big)^{(2m-1) \frac{\be^2}{2\pi}}} dy dz \Bigg)^{m}  \notag\\
 & \les 1 ,
% \label{E5}
 \end{align*}  

\noi
uniformly in $t\in[0,T]$, $x\in\T^2$, and $N\in\N$, provided $(2m -1) \be^2 < 4\pi\min(1, \al) $ and $0 < \al < 2$.

Lastly, Part (iii) 
for  the case $p=2$ 
follows from the last part of the proof of Proposition~1.1 in  \cite{ORSW}
(with $t = 2$), 
provided that 
$\be^2 < 4\pi\min(1, \al) $ and $0 < \al < 2$.
The second estimate~\eqref{bd-GN2} in Lemma \ref{LEM:covar} is needed here.
This completes the proof of Proposition~\ref{PROP:var}.
\end{proof}

\begin{remark}\label{rk:interm}
\rm
When $p = 2$, 
the proof of Proposition \ref{PROP:var}
 is identical to that in \cite[Proposition 1.1]{ORSW}.
  For $p > 2$, however, 
the bounds are quite different.
In computing higher moments
for 
the imaginary Gaussian multiplicative chaos $\,:\!e^{i \be\Psi_N}\!:\,$,
it was crucial to exploit certain  cancellation property
\cite{HS, ORSW}.
Namely,  
in the ``multipole picture'' for the imaginary Gaussian multiplicative chaos (and more generally log-correlated Gaussian fields \cite{LRV}), there is a ``charge cancellation'' 
 in estimating higher moments of $\,:\!e^{i\be \Psi_N}\!:\,$ due to its complex nature. 

In the current setting, 
i.e.~without the ``$i$'' in the exponent, 
there is no such cancellation taking place;
the charges accumulate and contribute to worse estimates in the sense that
the higher moment estimates require more smoothing. 
This is the source of the so-called intermittency phenomenon \cite{Garban}, which is quantified by the dependence on $p$ for the choice of $\al$ in Proposition \ref{PROP:var} (ii) above.

\end{remark}

\subsection{Kahane's approach}
Proposition \ref{PROP:var} in the previous subsection allows to get part of the result claimed in Proposition~\ref{PROP:Ups}. Indeed, using Fubini's theorem and arguing as in the proof of Proposition 1.1 in \cite{ORSW}, interpolating between (ii) and (iii) in Proposition \ref{PROP:var} above implies the convergence of $\{\U_N\}_{N\in\N}$ in $L^p(\O;L^p([0,T];W^{-\al,p}(\T^2)))$ in the case of even $p\ge 2$, for all $\al=\al(p)$ as in \eqref{al}.

Note, however, that 
when  $p\in(1,2)$ or $p>2$ is not even, we only get a weaker result than Proposition~\ref{PROP:Ups}. Indeed, when $p>2$ is not even,  if $2m<p<2m+2$ for some $m\in\N$, Proposition~\ref{PROP:var} provides convergence in both $L^{2m}(\O;L^{2m}([0,T];W^{-\al,2m}(\T^2)))$ and $L^{2m+2}(\O;L^{2m+2}([0,T];W^{-\al,2m+2}(\T^2)))$, which by interpolation provides convergence in $L^{p}(\O;L^{p}([0,T];W^{-\al,p}(\T^2)))$ for $\al=\al(p)$ as in \eqref{al}. 
Such an  argument then imposes  the condition 
\begin{align}
0<\be^2<\frac{4\pi}{(2m+2)-1},
\label{BB1}
\end{align} 
which gives a  smaller range than the natural one $0<\be^2<\frac{4\pi}{p-1}$. 
The condition \eqref{BB1}
comes  from the requirement that $\{\U_N\}_{N\in\N}$  be uniformly bounded in $L^{2m+2}(\O;L^{2m+2}([0,T];W^{-\al,2m+2}(\T^2)))$, 
 On the other hand, 
 in the case $p\in (1,2)$, interpolating between (i) and (iii) of Proposition~\ref{PROP:var}  provides the convergence of $\{\U_N\}_{N\in\N}$ in $L^p(\O;L^p([0,T];W^{-\al,p}(\T^2)))$ 
 only for in a more restricted range 
 $\al> \frac{(p-1)\be^2}{2\pi p} (>\al(p))$.

The argument presented in the previous subsection still has the advantage of being applicable to a large class of processes. Namely, whenever the $k$-points correlation functions can be expressed as a product, the use of the Brascamp-Lieb inequality (Corollary~\ref{COR:BL}) allows to decouple them into a product of 2-points correlation functions. As pointed out above, however, 
this only works for even $p\ge 2$, which restricts the range of admissible $\be^2>0$ 
in studying  \eqref{NH1} or \eqref{NW1}.
% see Sections~\ref{SEC:fpa},~\ref{SEC:WP} and~\ref{SEC:wave} below.

\medskip

In this subsection, we instead follow the classical approach of Kahane \cite{Ka85} which relies on the following comparison inequality for the renormalized exponential of Gaussian random variables. 
See, for example, \cite[Theorem 2.1]{RhV} and \cite[Corollary A.2]{RoV}.

\begin{lemma}[Kahane's convexity inequality]\label{LEM:Kahane}
Given $n\in\N$, 
let   $\{X_j\}_{j=1}^n$ and $\{Y_j\}_{j=1}^n$ be two centered Gaussian vectors satisfying 
\[\E\big[X_jX_k\big]\le \E\big[Y_jY_k\big]\]

\noi
 for all $j,k=1,...,n$. Then,  for any sequence  $\{p_j\}_{j=1}^n$
 of non-negative  numbers  and any convex function $F : [0,\infty)\to\R$ with at most polynomial growth at infinity, it holds
\begin{align*}
\E\bigg[F\Big(\sum_{j=1}^n p_je^{X_j-\frac12\E[X_j^2]}\Big)\bigg]
\le \E\bigg[F\Big(\sum_{j=1}^n p_je^{Y_j-\frac12\E[Y_j^2]}\Big)\bigg].
\end{align*}
\end{lemma}

As an application of Lemma~\ref{LEM:Kahane}, one has the following bound on the moments of the random measure\footnote{In the literature, this random measure is also referred
to as a multiplicative chaos.  See  \cite{RoV}.} $\M_N(t, \cdot )$, $t \geq 0$ defined by 
\begin{align}
\M_N(t, A) =  \int_{A}\U_N(t,x)dx
\label{MN}
\end{align}

\noi
for 
$ A\in\BB(\T^2)$, 
where $\BB(\T^2)$ is the Borel $\s$-algebra of $\T^2$.

\begin{lemma}\label{LEM:GMCmom}
For any $0<\be^2<8\pi$ and $1\leq p< \frac{8\pi}{\be^2}$, we have
\begin{align*}
\sup_{t\in\R_+,A\in\BB(\T^2),N\in\N}\E\Big[\M_N(t,A)^p\Big]<\infty.
\end{align*}
\end{lemma}

Lemma~\ref{LEM:GMCmom} is a classical result in the theory of Gaussian multiplicative chaoses. See for example Proposition 3.5 in \cite{RoV}. We present a self-contained proof in Appendix~\ref{SEC:B} below.

With the bounds of Lemmas~\ref{LEM:Kahane} and~\ref{LEM:GMCmom}, we can prove the following uniform estimate on $\{\U_N\}_{N\in\N}$.
\begin{proposition}\label{PROP:var2}
Let $0 < \be^2<8\pi$, $1\leq p<\frac{8\pi}{\be^2}$, and $0<\al<2$ such that $\al>(p -1) \frac{\be^2}{4\pi}$.
Then, we have for any $T>0$
 \begin{align}
 \sup_{t\in[0,T],x\in\T^2,N\in\N}\E \Big[ \left| \jb{\nabla}^{-\alpha} \U_N (t,x) \right|^{p} \Big] \le C(T).
 \label{moment2}
 \end{align}
\end{proposition}

Note that  in Proposition~\ref{PROP:var2}, 
 we do not need to assume that $p$ is even. 
The uniform bound in Proposition~\ref{PROP:Ups}\,(i)  follows  from \eqref{moment2}, while the convergence part of Proposition~\ref{PROP:Ups} follows from interpolating \eqref{moment2} in Proposition~\ref{PROP:var2} and Proposition~\ref{PROP:var}\,(iii) and using the same argument as in the proof of Proposition 1.1 in \cite{ORSW}. 
When  $1<p<2$, 
the use  of  Proposition~\ref{PROP:var}\,(iii) imposes the condition $0<\be^2<4\pi$, 
which yields the restriction on the range of $\be^2$
in Proposition~\ref{PROP:Ups}\,(ii).

\begin{proof}[Proof of Proposition~\ref{PROP:var2}]
We split the proof into two steps.

\medskip

\noi 
$\bullet$ \textbf{Step 1: multifractal spectrum.} We first establish the following bound on the moments of  
 the random measure $\M_N(t)$ over small balls:
\begin{align}\label{moment3}
\sup_{t\in[0,T],x_0\in\T^2,N\in\N}\E\Big[\M_N(t,B(x_0,r))^p\Big] \les r^{(2+\frac{\be^2}{4\pi})p-\frac{\be^2}{4\pi}p^2}
\end{align}
for any $r\in(0,1)$.

By a change of variables, 
 the positivity of $\U_N$, and
 a Riemann sum approximation, 
 we have 
\begin{align}
\begin{split}
\E\Big[\M_N(t,B(x_0,r))^p\Big] 
&= r^{2p}\E\bigg[\Big(\int_{B(0,1)} \U_N(t,x_0+ry)dy\Big)^p\bigg]\\
&\leq r^{2p}\E\bigg[\Big(\int_{\T^2} \U_N(t,x_0+ry)dy\Big)^p\bigg]\\
%&\le r^{2p}\E\bigg[\Big(\int_{\T^2}e^{\be \Psi_N(t,x_0+ry)-\frac{\be^2}2\s_N}dy\Big)^p\bigg]\\
&=r^{2p}\lim_{J\to\infty}\E\bigg[\Big(\sum_{j,k=1}^{J}\frac{4\pi^2}{J^{2}}e^{\be \Psi_N(t,x_0+ry_{j,k})-\frac{\be^2}2\s_N}\Big)^p\bigg],
\end{split}
\label{BB2}
\end{align}

\noi
where $y_{j,k}$, $j, k = 1, \dots, J$, 
is given by  $y_{j,k} = \big(-\pi + \frac{2\pi}{J}(j-1),-\pi + \frac{2\pi}{J}(k-1)\big)\in\T^2\simeq [-\pi,\pi)^2$. 
From  
 Lemma~\ref{LEM:covar},  we can bound the covariance function as
\begin{align}
\begin{split}
&\E\big[\Psi_N(t,x_0+ry_{j_1,k_1})\Psi_N(t,x_0+ry_{j_2,k_2})\big]\\
&\qquad\le -\frac1{2\pi}\log\big(r|y_{j_1,k_1}-y_{j_2,k_2}|+N^{-1}\big)+C\\
&\qquad\le -\frac1{2\pi}\log\big(|y_{j_1,k_1}-y_{j_2,k_2}|+(rN)^{-1}\big)-\frac1{2\pi}\log r +C\\
&\qquad\le -\frac1{2\pi}\log\big(|y_{j_1,k_1}-y_{j_2,k_2}|+N^{-1}\big)-\frac1{2\pi}\log r +C\\
&\qquad\le \E\Big[\big(\Psi_{N}(t,y_{j_1,k_1})+h_r\big)\big(\Psi_{N}(t,y_{j_2,k_2})+h_r\big)\Big]
\end{split}
\label{BB4}
\end{align}

\noi
for any  $0 < r < 1$ and $j_1, j_2, k_1, k_2 = 1,\dots, J$, 
where $h_r$ is a mean-zero  Gaussian random variable with variance $- \frac1{2\pi}\log r + C$, independent from $\Psi_{N}$. 
Then, by applying Kahane's convexity inequality (Lemma~\ref{LEM:Kahane}) with the convex function $x\mapsto x^p$, a Riemann sum approximation, and the independence
of $h_r$ from $\Psi_N$, it follows from \eqref{BB2} that 
\begin{align*}
\E\Big[\M_N(t,B(x_0,r))^p\Big] 
&\le r^{2p}\lim_{J\to\infty}\E\bigg[\Big(\sum_{j,k=1}^{J}\frac{4\pi^2}{J^{2}}e^{\be (\Psi_{N}(t,y_{j,k})+h_r)-\frac{\be^2}2\E[(\Psi_{N}(t,y_{j,k})+h_r)^2]}\Big)^p\bigg]\\
&=r^{2p}\E\bigg[\Big(\int_{\T^2}e^{\be \Psi_{N}(t,y)-\frac{\be^2}2\s_N}
e^{\be h_r-\frac{\be^2}2(- \frac1{2\pi}\log r + C)}dy\Big)^p\bigg]\\
&=r^{2p}\E\bigg[e^{p\be h_r-p\frac{\be^2}2\big(\frac1{2\pi}\log\frac1r + C\big)}\bigg]
\E\bigg[\Big(\int_{\T^2}e^{\be \Psi_{N}(t,y)-\frac{\be^2}2\s_N}dy\Big)^p\bigg]\\
&=r^{2p}e^{(p^2-p)\frac{\be^2}2(-\frac1{2\pi}\log r + C)}
\E\bigg[\Big(\int_{\T^2}e^{\be \Psi_{N}(t,y)-\frac{\be^2}2\s_N}dy\Big)^p\bigg]\\
&\les r^{(2+\frac{\be^2}{4\pi})p-\frac{\be^2}{4\pi}p^2}\E\Big[\M_{N}(t,\T^2)^p\Big].
\end{align*}

\noi
Hence, 
the bound \eqref{moment3} follows  Lemma~\ref{LEM:GMCmom}.

\smallskip
\noi
$\bullet$ \textbf{Step 2:}
 From \eqref{Bessel}, 
 Lemma \ref{LEM:Bessel}, 
 where the remainder $R$ is bounded on $\T^2$, 
 and Minkowski's integral inequality, 
 we have 
\begin{align*}
 \E \Big[ \big| \jb{\nabla}^{-\alpha} \U_N (t,x) \big|^{p} \Big]
 &= \E \bigg[ \Big| \int_{\T^2}J_\al(x-y)\U_N (t,y)dy \Big|^{p} \bigg]\\
 &\les \E \bigg[ \Big( \int_{\T^2}|x-y|^{\al-2}\U_N (t,y)dy \Big)^{p} \bigg]\\
 &\les \Bigg\{\sum_{\ell\ge 0}2^{-(\al-2)\ell}\E \bigg[ \Big( \int_{|x-y|\sim 2^{-\ell}}\U_N (t,y)dy \Big)^{p} \bigg]^{\frac1p}\Bigg\}^p\\
 &\les \sup_{\ell\ge 0}2^{-(\al-2-\eps)p\ell}\E \Big[ \M_N\big(t,B(x,2^{-\ell})\big)^{p} \Big]
 \end{align*}
 
 \noi
for $\eps > 0$, 
 uniformly in $t \in \R_+$, $x \in \T^2$, and $N\in\N$, 
Then, 
using  \eqref{moment3}, we obtain
 \begin{align*}
 \E \Big[ \big| \jb{\nabla}^{-\alpha} \U_N (t,x) \big|^{p} \Big] &\les \sup_{\ell\ge 0}2^{-(\al-2-\eps)p\ell}2^{\frac{\be^2}{4\pi}p^2\ell -(2+\frac{\be^2}{4\pi})p\ell} \les 1
\end{align*}

\noi
by choosing $\eps > 0$ sufficiently small, 
provided that 
 $\al >(p-1)\frac{\be^2}{4\pi}$.
  This  proves \eqref{moment2}.
\end{proof}

\section{Parabolic Liouville equation I: general case}
\label{SEC:fpa}
In this section, 
we present a proof of   Theorem \ref{THM:para1}.
Namely, we prove   local well-posedness 
of the truncated SNLH \eqref{NH5} 
for $v_N = u_N - z - \Psi_N$  in the Da Prato-Debussche formulation
in the range:
\begin{align*}
0 < \be^2 < \bh^2 \deff \frac{8\pi}{3+2\sqrt{2}}
%\label{Q1}
\end{align*}

\noi
{\it without} assuming the positivity of $\ld$.
Here, 
$z$ denotes the deterministic linear solution defined in \eqref{linear}
and $\Psi_N$ denotes the truncated stochastic convolution defined in 
 \eqref{conv}.

Writing  \eqref{NH5} in the  Duhamel formulation, we have
\begin{align}
v_N =  -\frac12\ld\be \int_{0}^t P(t-t')\big(e^{\be z } e^{\be v_N} \U_N\big)(t')  dt'.
\label{Duhamel}
\end{align}

\noi
Given $v_0 \in L^\infty(\T^2)$
and a space-time distribution $\U$, 
we define a map $\Phi$ by 
\begin{align}
\Phi ( v) =  \Phi_{v_0, \U}(v) \deff-\frac{1}{2}\ld\be \int_{0}^t P(t-t') 
\big(e^{\be P(t) v_0 } e^{\be v} \U\big)(t') dt'.
\label{D2}
\end{align}

\noi
Then, \eqref{Duhamel} can be written as 
the following fixed point problem:
\[v_N = \Phi_{v_0, \U_N}(v_N).\]

In the following, we fix 
$0 < \al, s < 1$ and  $p\ge 2$ such that
\begin{align}
 p'\frac{\al +s}2 < 1
 \qquad\text{and}\qquad  
sp >2.
\label{cons1}
\end{align}

\noi
See \eqref{cons2} below for a concrete choice of these parameters.
 Then,  we have the following deterministic well-posedness result
 for the fixed point problem:
 \begin{align}
  v = \Phi_{v_0, \U}(v).
\label{D2a}
 \end{align}

\begin{proposition}
\label{PROP:v1}
Let $\al,s,p$ be as above.
Then,  given any $v_0\in L^{\infty}(\T^2)$ and $R>0$, 
there exists $T=T(\|v_0\|_{L^{\infty}}, R)>0$ 
such that given any positive distribution $\U\in L^p([0,T];W^{-\al,p}(\T^2))$ 
satisfying  
\begin{align}
\|\U\|_{L^p_{T}W^{-\al,p}_x}\le R,
\label{Q2x}
\end{align}

\noi
there exists a unique solution 
 $v\in C([0,T];W^{s,p}(\T^2))$
to \eqref{D2a},
depending continuously on the initial data $v_0$.
\end{proposition}

Note that we do not claim any continuity of the solution 
$v$ in $\U$ for Proposition \ref{PROP:v1}.

\begin{proof}
Fix $R>0$.
We prove that  there exists $T = T(\|v_0\|_{L^\infty}, R)>0$ such that 
$\Phi_{v_0, \U}$ is a contraction on the ball $B\subset C([0,T];W^{s,p}(\T^2))$ 
of radius $O(1)$ centered at the origin.

Let $v \in B$.
Then, by Sobolev's embedding theorem (with $sp>2$),  we have  $v\in C([0,T]; C(\T^2))$. 
For $v_0\in L^{\infty}(\T^2)$, 
we also have  $z\in C((0,T];C(\T^2))$.
In particular, $e^{\be z} e^{\be v}(t)$ is continuous in $x \in \T^2$
for any $t \in (0, T]$.
Then, by  the Schauder estimate (Lemma~\ref{LEM:heatker}~(ii)), 
 Lemma~\ref{LEM:posprod},  and  Young's  inequality with~\eqref{cons1}, we have 
\begin{align}
\begin{split}
\|\Phi (v)\|_{C_{T}W^{s,p}_x} 
& \les \bigg\| \int_0^t (t-t')^{-\frac{(s+\al)}2}
\big\|\jb{\nabla}^{-\al}  \big( e^{\be z } e^{\be v}\U\big)(t')\big\|_{L^p_x}  dt' 
\bigg\|_{L^\infty_t ([0, T])}\\
& \les  \|e^{\be z} e^{\be v} \|_{L^{\infty}_{T,x}}
\bigg\|
\big(\ind_{[0, T]}
|\cdot|^{-\frac{(s+\al) }2}\big)*\| \ind_{[0, T]}\U \|_{W^{-\al,p}_x}  \bigg\|_{L^{\infty}_t}\\
& \les 
T^\ta  e^{C \|v_0\|_{L^\infty} } e^{ C \|v \|_{L^{\infty}_{T}W^{s,p}_x} }   \| \U \|_{L^p_T W^{-\al,p}_x} \\
&  \les T^\ta  R\,  e^{C \|v_0\|_{L^\infty} }   
\les 1
\end{split}
\label{Gv1}
\end{align}

\noi
for $v \in B$ and a positive distribution $\U$ satisfying \eqref{Q2x},
by choosing 
 $T = T(\|v_0\|_{L^\infty}, R)>0$ sufficiently small.

By the fundamental theorem of calculus, we have %mean value theorem, we have 
\begin{align}
 e^{\be v_1} -e^{\be v_2} = \be (v_1-v_2)\int_0^1e^{\be\tau v_1+\be(1-\tau)v_2}d\tau.
\label{Q3}
 \end{align}

\noi
Then, proceeding as in \eqref{Gv1} with \eqref{Q3}, 
we have 
\begin{align}
\begin{split}
\|\Phi (v_1)   -\Phi (v_2) \|_{C_TW^{s,p}_x} 
& \les  T^\ta \Big\| e^{\be z}\big(e^{\be v_1} -e^{\be v_2} \big) \Big\|_{L^{\infty}_{T,x}}
 \| \U \|_{L^p_T W^{-\al , p}_x} \\
& \les T^\ta  Re^{C \|v_0\|_{L^{\infty}}}
e^{C (\|v_1\|_{L^{\infty}_{T, x}}+ \|v_2\|_{L^{\infty}_{T, x}} ) }
  \|v_1 -v_2\|_{L^{\infty}_{T,x}} \\
& \les T^\ta R e^{C \|v_0\|_{L^\infty} }  \|v_1 -v_2\|_{C_TW^{s,p}_x}
\end{split}
\label{Gv2}
\end{align}

\noi
for $v_1, v_2 \in B$ and a positive distribution $\U$ satisfying \eqref{Q2x}.

Hence, from \eqref{Gv1} and \eqref{Gv2}, 
we see that $\Phi$ is a contraction on $B$ by taking 
$T = T(\|v_0\|_{L^\infty}, R)>0$ sufficiently small.
The continuity of the solution $v$ in initial data follows
from a standard argument and hence we omit details.
\end{proof}

\begin{remark}\label{REM:uniq}\rm
In the proof of Proposition \ref{PROP:v1}, 
a contraction argument shows
the uniqueness of the solution $v$ only in the ball $
B\subset C([0,T];W^{s,p}(\T^2))$. 
By a standard continuity argument, 
we can upgrade the uniqueness statement to hold
in the entire $C([0,T];W^{s,p}(\T^2))$.
Since such an  argument is standard, we omit details. 
\end{remark}

Now, let $\U_N$ be the Gaussian multiplicative chaos
in \eqref{re-non}.
In view of Proposition \ref{PROP:Ups}, 
 in order to determine  the largest admissible range for ${\be^2}$, 
 we aim to maximize
\begin{align*}
\be^2  & < \frac{4\pi \al}{p-1} < \frac{p-2}{p(p-1)}8\pi  =: h(p),
\end{align*}

\noi
where we used both of the inequalities in \eqref{cons1}.
A direct computation shows  that $h$ has a unique maximum in $[2,\infty)$ reached at 
$p = p_* = 2+ \sqrt 2$, for which 
we have 
\[
 h(p_*) =   \max_{p\ge 2}h(p) = \frac{8\pi} {3+2\sqrt{2}}
= \bh^2.
\]

\noi
Therefore, for  
$\be^2 <  \bh^2$,
we see that 
the constraints \eqref{cons1} are satisfied by taking
\begin{align}
\begin{split}
& p= 2+\sqrt{2},\qquad s= 2- \sqrt{2}+\eps, \qquad \text{and}
\qquad \\
& \al=(p-1)\frac{\bh^2}{4\pi}-2\eps=2(\sqrt{2}-1)-2\eps
\end{split}
\label{cons2}
\end{align}

\noi
for sufficiently small  $\eps>0$ 
such that  $\al>(p-1)\frac{\be^2}{4\pi}$. 
With this choice of the parameters, 
Proposition \ref{PROP:v1}
with Proposition \ref{PROP:Ups}
establishes local well-posedness
of \eqref{Duhamel}.

In the remaining part of this section, 
we fix the parameters $\al, s$, and $p$ as in \eqref{cons2}
and proceed with a proof of Theorem \ref{THM:para1}.

\begin{proof}[Proof of Theorem \ref{THM:para1}]

Given 
$v_0\in L^{\infty}(\T^2)$ and $\U_N$ in \eqref{re-non}, 
let 
 $v_N=\Phi_{v_0, \U_N}(v_N)$ be the solution to \eqref{Duhamel}
  given by Proposition \ref{PROP:v1}. 
 Proceeding 
  as in the proof of Theorem 1.2 in \cite{ORSW}, it suffices to prove the continuity of the 
  solution map $\Phi = \Phi_{v_0, \U}$ constructed in Proposition~\ref{PROP:v1} with respect to $\U$. 
 
In the proof of Proposition \ref{PROP:v1}, 
 the positivity of the distribution $\U$ played an important role, allowing us to 
apply  Lemma~\ref{LEM:posprod}.
 In studying the difference 
  $\U_N-\U$, 
  we lose such positivity and can no longer apply Lemma \ref{LEM:posprod}.
 This prevents us from showing 
  convergence of $v_N$ in $C([0,T];W^{s,p}(\T^2))$ directly. 
  We instead  use a compactness argument.

Let us  take a sequence of positive distributions
$\U_N$ converging to some limit  $\U$ in $L^p([0,T];W^{-\al,p}(\T^2))\cap L^r([0,T];W^{-s+\eps,r}(\T^2))$, where $r$ is defined by
\begin{align}
r=\frac{4\pi (s-\eps)}{\bh^2 }+1 = 2+\frac{\sqrt{2}}{2}
\label{r}
\end{align}

\noi
with $s$ as in \eqref{cons2}.
Note that the pair $ (s-\eps,r)$ (in place of $(\al, p)$) satisfies \eqref{al} for any $\be^2<\bh^2$.

Let us then denote by $v_N$ and $v$ the corresponding solutions to \eqref{NH5} 
and \eqref{NH5a}, respectively, constructed in Proposition \ref{PROP:v1}. 
We first show an extra regularity for these solutions:
 \begin{align*}
 \dt v_N\in L^p([0,T];W^{s-2,p}(\T^2)).
% \label{Q4}
 \end{align*}
 
 \noi
Indeed, using the equation  \eqref{NH5}
with $p < \infty$ and $s-2<-\al$, we have
\begin{align*}
\|\dt v_N\|_{L^p_TW^{s-2,p}_x}&=\Big\|\tfrac12(\Dl-1)v - \tfrac12\ld\be  e^{\be z}e^{\be v_N}\U_N\Big\|_{L^p_TW^{s-2,p}_x}\\
&\les \|v_N\|_{L^{\infty}_TW^{s,p}_x} + \big\| e^{\be z}e^{\be v}\U_N\big\|_{L^p_TW^{-\al,p}_x}.
\end{align*} 

\noi
Note that both of the terms on the right-hand side are already bounded in the proof of Proposition \ref{PROP:v1}
(by switching the order of  Lemma \ref{LEM:posprod} 
and Young's inequality in \eqref{Gv1}).

Next, observe that by taking $\wt s> s$, 
sufficiently close to $s$, 
we can repeat the proof of Proposition~\ref{PROP:v1} without changing the range of $\be^2<\bh^2$.  
This shows that 
$\{v_N\}_{N \in \N}$  is bounded in $C([0,T];W^{\wt s,p}(\T^2))$. 
Then,  by Rellich's lemma and the Aubin-Lions lemma (Lemma~\ref{LEM:AL}), 
we see that the embedding: 
\[A_T \deff C([0,T];W^{\wt s,p}(\T^2))\cap\big\{\dt v\in L^p([0,T];W^{ s-2,p}(\T^2))\big\}\subset C([0,T];W^{s,p}(\T^2))\] 

\noi
is compact.
Since $\{v_N\}_{N \in \N}$ is bounded in $ A_T$, 
given any subsequence of  $\{v_N\}_{N\in \N}$, 
we can extract a further  subsequence $\{v_{N_k}\}_{k \in \N}$ 
such that  $v_{N_k}$ converges to some limit $\wt v$ in $C([0,T];W^{s,p}(\T^2))$. 
In the following, we show that 
$\wt v = v$.
This implies that the limit is independent of the choice of subsequences
and hence the entire sequence $\{v_N\}_{N \in \N}$
converges to $v$  in $C([0,T];W^{s,p}(\T^2))$.

It remains to prove $\wt v = v$.
In the following, 
we first show that $v_{N_k}=\Phi_{v_0, \U_{N_k}}(v_{N_k})$ converges to 
$\Phi_{v_0, \U}(\wt v)$ in $L^1([0,T];W^{s',p}(\T^2))$ for some $s'\leq - s$.
From  \eqref{D2}, we have
\begin{align}
\begin{split}
\|\Phi_{v_0, \U_{N_k}}(v_{N_k})& -\Phi_{v_0, \U}(\wt v)\|_{L^1_TW^{s',p}_x} \\
&\les \bigg\|\int_0^tP(t-t')\big(e^{\be z}e^{\be \wt v}(\U_{N_k}-\U) \big)(t')
dt'\bigg\|_{L^1_TW^{s',p}_x}
\\
&\hphantom{XX}+\bigg\|\int_0^tP(t-t')\big(e^{\be z}(e^{\be v_{N_k}} - e^{\be \wt v})\U_{N_k}\big)(t')dt'\bigg\|_{L^1_TW^{s',p}_x}\\
&=: \1+\II.
\end{split}
\label{Q4a}
\end{align}

By the Schauder estimate (Lemma~\ref{LEM:heatker}),  Young's  inequality, 
Lemma \ref{LEM:prod1}\,(ii)
with  $\frac1r+\frac1p <\frac1r + \frac{s}2$
(which is guaranteed by $sp > 2$), 
we have
\begin{align}
\begin{split}
\1 &\les \Big\|
|\cdot|^{-(\frac1r-\frac1p)}*\big\|e^{\be z}e^{\be \wt v}(\U_{N_k}-\U)\big\|_{W^{-s+\eps,r}_x}\Big\|_{L^{1}_T}\\
&\les \big\|e^{\be z}e^{\be \wt v}(\U_{N_k}-\U)\big\|_{L^1_TW^{-s+\eps,r}_x}\\
&\les \big\|e^{\be (z+\wt v)}\big\|_{L^{r'}_TW^{s-\eps,p}}\|\U_{N_k}-\U\|_{L^r_TW^{-s+\eps,r}_x}.
\end{split}
\label{Q5}
\end{align}

\noi
By 
Sobolev's  inequality and 
the fractional chain rule (Lemma \ref{LEM:FC}\,(ii)), we have
 \begin{align*}
\big\| |\nb|^{s-\eps} e^{\be (z+\wt v)}(t)\big\|_{L^p_x}
\les \big\||\nb|^s e^{\be (z+\wt v)}(t)\big\|_{L^\frac{p}{1 + \eps p/2}_x}
&\les \big\|e^{\be (z+\wt v)}(t)\big\|_{L^{\frac{2}{\eps}}_{x}}
\big\||\nb|^s (z+\wt v)(t)\big\|_{L^p_x}.
\end{align*}

 \noi
 This yields
\begin{align}
\begin{split}
\big\|e^{\be (z+\wt v)}\big\|_{L^{r'}_TW^{s-\eps,p}}
&\les 
\big\|e^{\be (z+\wt v)}\big\|_{L^{\infty}_{T,x}}
\Big(1 + \|z+\wt v\|_{L^{r'}_TW^{s,p}_x}\Big)\\
& \les 
e^{C\|v_0\|_{L^{\infty}}}e^{C \|\wt v\|_{L^{\infty}_TW^{s,p}_x}}
\Big( 1 + \|v_0\|_{L^{\infty}}+\|\wt v\|_{L^{\infty}_TW^{s,p}_x}\Big).
\end{split}
\label{Q6}
\end{align}

\noi
In the last step,  we used the following bound
which follows from the Schauder estimate (Lemma \ref{LEM:heatker}):
\begin{align*}
\|z\|_{L^{r'}_TW^{s,p}_x} &\les \big\|t^{-\frac{s}{2}}\|v_0\|_{L^p}\big\|_{L^{r'}_T}
\les \|v_0\|_{L^{\infty}}
\end{align*} 

\noi
since $\frac{s}{2}r'<1$ in view of \eqref{cons2} and \eqref{r}.
Therefore, from \eqref{Q5} and \eqref{Q6}, we obtain
\begin{align}
\1&\les 
e^{C\|v_0\|_{L^{\infty}}}e^{C \|\wt v\|_{L^{\infty}_TW^{s,p}_x}}
\Big( 1 + \|v_0\|_{L^{\infty}}+\|\wt v\|_{L^{\infty}_TW^{s,p}_x}\Big)
\|\U_{N_k}-\U\|_{L^r_TW^{-s+\eps,r}_x}.
\label{Gv4}
\end{align}

As for the second term $\II$ on the right-hand side of \eqref{Q4a}, 
we can use the positivity of $\U_{N_k}$ and proceed as in 
 \eqref{Gv2}:
\begin{align}
\II&\les T^\ta  e^{C \big(\|v_0\|_{L^{\infty}}+\|v_{N_k}\|_{L^{\infty}_TW^{s,p}_x}+\|\wt v\|_{L^{\infty}_TW^{s,p}_x}\big)} \|v_{N_k} -\wt v\|_{L^{\infty}_TW^{s,p}_x}\|\U_{N_k}\|_{L^p_TW^{-\al,p}_x} .
\label{Gv5}
\end{align}

Since $v_{N_k}\rightarrow \wt v$ in $C([0,T];W^{s,p}(\T^2))$
and $\U_N \to  \U$ in
 $L^p([0,T];W^{-\al,p}(\T^2))\cap L^r([0,T];W^{-s+\eps,r}(\T^2))$, 
it follows from \eqref{Q4a}, \eqref{Gv4}, and \eqref{Gv5}
that 
 $v_{N_k} = \Phi_{v_0, \U_{N_k}}(v_{N_k})$
 converges to $\Phi_{v_0, \U}(\wt v)$ in $L^1([0,T];W^{s',p}(\T^2))$. 
 By the uniqueness of the distributional limit, 
 we conclude that 
 \begin{align}
 \wt v = \Phi_{v_0, \U}(\wt v).
 \label{Q7}
 \end{align}
 
 \noi
 Since $\wt v$ belongs to  $C([0,T];W^{s,p}(\T^2))$,
 we conclude from the uniqueness
 of the solution to \eqref{Q7}
 that $v = \wt v$, 
 where $v$ denotes the unique fixed point to \eqref{Q7}
 in the class 
  $C([0,T];W^{s,p}(\T^2))$ constructed in Proposition \ref{PROP:v1}.
  See also Remark \ref{REM:uniq}.
\end{proof}

\begin{remark} \label{REM:uniq2} \rm
While the argument above shows
the continuity of the solution map in $\U$, 
its dependence is rather weak.
For the range $0 < \be^2 < \frac{4}{3}\pi$, 
we can strengthen this result
by proving local well-posedness and convergence
{\it without} the positivity of $\U$.
This argument shows that, 
for the range $0 < \be^2 < \frac{4}{3}\pi$,  
the solution map is also Lipschitz with respect to $\U$, as in the hyperbolic case presented in Section~\ref{SEC:wave} below. See Appendix \ref{SEC:A}.
\end{remark}

\section{Parabolic Liouville equation II: using the sign-definite structure}
\label{SEC:WP}

In this section, we study SNLH \eqref{NH1} 
in the defocusing case ($\ld > 0$)
and present a proof of Theorem \ref{THM:1.2} and Theorem \ref{THM:2}.
As we will see below, the particular structure of the equation makes the exponential nonlinearity 
behave as a smooth bounded  function. 
This allows us to  treat the full range $0 < \be^2<4\pi$ in this case.

\subsection{Global well-posedness}%Existence and uniqueness of global solutions}
\label{SUBSEC:para2}

In this subsection, we focus on the equation:
\begin{align}
\begin{cases}
\dt v +\frac12(1-\Dl)v + \frac12\ld\be e^{\be z} e^{\be v} \U  = 0\\
v|_{t=0}=0,
\end{cases}
\label{v1}
\end{align}

\noi
where  $z  = P(t) v_0 $ for some $v_0 \in L^\infty(\T^2)$,
$\U$ is a given {\it deterministic} positive space-time distribution,
and $\ld > 0$.
In this case, as explained in Subsection \ref{SUBSEC:1.3}, 
the equation \eqref{v1}
can be written as 
\begin{align}
\begin{cases}
\dt v +\frac12(1-\Dl)v + \frac12\ld\be e^{\be z} F(\be v)\U  = 0\\
v|_{t=0}=0,
\end{cases} 
\label{v2}
\end{align}

\noi
where $F$ is a smooth bounded and Lipschitz function defined in \eqref{F1}.
Indeed, by writing~\eqref{v2} in the Duhamel formulation:
\begin{align}
v(t)=-\frac12\ld\be\int_0^t P(t-t') \big(e^{\be z} F(\be v) \U \big)(t')dt',
\label{v3}
\end{align}

\noi
it follows from 
the non-negativity of  $\ld$, $\U$, and $F$ 
along with Lemma \ref{LEM:heatker}\,(i)
that  $\be v\leq 0$.
 This means that the Cauchy problems \eqref{v1} and \eqref{v2} are equivalent.

Given $N \in \N$, 
consider the following equation:
 \begin{align}
\label{v3a}
\begin{cases}
\partial_t v_N + \frac 12 (1- \Delta) v_N +  \frac 12 \ld \be
e^{\be z }  F(\be v_N) \U_N=0\\
v_N|_{t = 0} = 0
\end{cases}
\end{align}

\noi
for some given 
smooth space-time non-negative function $\U_N$.
Then, since $\U_N$ is smooth
and $F$ is bounded and Lipschitz, 
we can apply a standard contraction argument
to prove local well-posedness of \eqref{v3a}
in the class $C([0, \tau]; L^2(\T^2))$
for some small $\tau = \tau_N>0$.
Thanks to the boundedness of $F$, 
we can also establish an a priori bound on the $L^2$-norm 
of the solution $v_N$ 
on any time interval $[0, T]$;  see \eqref{R2} below.
This  shows  global existence of $v_N$.

 Our main goal in this subsection is to prove global well-posedness of \eqref{v2}.
\begin{proposition}\label{PROP:flow}
Let 
 ${v_0}\in L^{\infty}(\T^2)$
and  $\U\in L^2([0,T];H^{-1+\eps}(\T^2))$
be a positive distribution
for some $\eps > 0$.
Given $T > 0$, 
suppose that a sequence $\{\U_N\}_{N \in \N}$ 
of smooth non-negative functions
converges to  
$\U$ in $L^2([0,T];H^{-1+\eps}(\T^2))$.
Then, 
the corresponding solution $v_N$ to \eqref{v3a}
converges to 
a limit $v$ 
in the energy space $\ZZ_T$
 defined in~\eqref{energyZ}.
 Furthermore, 
 the limit $v$ is the unique solution to~\eqref{v2}
 in the energy class $\ZZ_T$.

\end{proposition}

In view of Proposition \ref{PROP:Ups} with $p = 2$, 
given $0 < \be^2 < 4\pi $, 
we can choose $\eps > 0$ sufficiently small such that 
$\frac{\be^2}{4\pi}<1-\eps$, 
which guarantees that 
the Gaussian multiplicative chaos $\U_N$ in~\eqref{re-non}
belongs to $ L^2([0,T];H^{-1+\eps}(\T^2))$ for any $T>0$, almost surely.
Moreover, $\U_N$ converges in probability to $\U$ in \eqref{U2}
in the same class.
Then, Theorem \ref{THM:para2}
follows from Proposition~\ref{PROP:flow} above.

\begin{proof}[Proof of Proposition \ref{PROP:flow}]

With a slight abuse of notation,
we set 
\[  \Phi = \Phi_{v_0, \U}
\qquad \text{and} \qquad 
\Phi_N = \Phi_{v_0, \U_N},\]

\noi
where $\Phi_{v_0, \U}$ is defined in \eqref{D2}.
In particular, we have
\begin{align}
\begin{split}
v_N & = \Phi_N(v_N) = 
\Phi_{v_0, \U_N}(v_N) \\
&  =  -\frac12\ld\be\int_0^t P(t-t') \big(e^{\be z} F(\be v_N)\U_N \big)(t')dt'.
\end{split}
\label{Phi}
\end{align}

 Fix $T>0$. 
Given $v_0 \in L^\infty(\T^2)$,  
we see that $z = P(t) v_0$ and $ v_N$ belong to $ C((0, T];  C(\T^2))$
in view of the Schauder estimate (Lemma~\ref{LEM:heatker})
and \eqref{Phi} with smooth $\U_N$.
Hence, we can apply  Lemma~\ref{LEM:posprod}
to estimate the product 
$e^{\be z} F(\be v_N)\U_N$ thanks to the positivity of $\U_N$.

Fix small $\dl > 0$. 
Then,  by the Schauder estimate (Lemma~\ref{LEM:heatker}), Lemma~\ref{LEM:posprod}, 
and Young's inequality, we have
\begin{align}
\begin{split}
\|v_N\|_{L^2_TH^{1+2\dl}_x} 
&\les \bigg\|\int_0^t (t-t')^{-\frac{2+2\dl-\eps}{2}}
\big\|\jb{\nabla}^{-1+\eps}\big(e^{\be {z}}F(\be v_N)\U_N \big)(t')\big\|_{L^2_x}dt'\bigg\|_{L^{2}_T}\\
&\les \big\|e^{\be z}F(\be v_N)\big\|_{L^{\infty}_{T,x}}
\bigg\|\int_0^t (t-t')^{-\frac{2+2\dl-\eps}{2}}\|\U_N(t')\|_{H^{-1+\eps}_x}dt'\bigg\|_{L^2_T}\\
 & \les e^{C\| {v_0}\|_{L^{\infty}}} \|\U_N\|_{L^2_T H^{-1+\eps}_x}, 
\end{split}
\label{R1}
\end{align}

\noi
uniformly in $N \in \N$, 
provided that $2\dl < \eps$.
Here, we crucially used the boundedness of $F$.
Similarly, we have 
\begin{align}
\begin{split}
\|v_N\|_{L^{\infty}_TH^{2\dl}_x} 
&\les e^{C \|{v_0}\|_{L^{\infty}}}
\bigg\|\int_0^t (t-t')^{-\frac{1+2\dl-\eps}{2}}\|\U_N(t')\|_{H^{-1+\eps}_x}dt'\bigg\|_{L^\infty_T}\\
& \les e^{C\|{v_0}\|_{L^{\infty}}}\|\U_N\|_{L^2_TH^{-1+\eps}_x}
\end{split}
\label{R2}
\end{align}

\noi
and
\begin{align}
\begin{split}
\|\dt v_N \|_{L^2_TH^{-1+2\dl}_x} &= \Big\|\tfrac12(\Dl-1)v_N 
-\tfrac{1}{2}\ld \be  e^{\be{z}}F(\be v_N)\U_N \Big\|_{L^2_TH^{-1+2\dl}_x}\\
&\les \|v_N\|_{L^2_TH^{1+2\dl}_x}+\big\| e^{\be{z}}F(\be v_N)\U_N\big\|_{L^2_TH^{-1+\eps}_x}\\
& \les 
e^{C\| {v_0}\|_{L^{\infty}}} \big\| \U_N\big\|_{L^2_TH^{-1+\eps}_x}, 
\end{split}
\label{R3}
\end{align}

\noi
uniformly in $N \in \N$.

Given $s \in \R$, define
$\ZZ^s_T$ and  $\wt\ZZ^{s}_T$ by 
\begin{align*}
\begin{split}
\ZZ_T^s & =C([0,T];H^s (\T^2))\cap L^2([0,T];H^{1+s}(\T^2)),\\
\wt\ZZ^{s}_T
& = 
 \big\{v \in 
 \ZZ^s_T : \dt v\in L^2([0,T];H^{-1+s}(\T^2))\big\}.
\end{split}
%\label{energyZ2}
\end{align*} 

\noi
Then, it follows from  Rellich's lemma and the Aubin-Lions lemma (Lemma~\ref{LEM:AL})
that 
the embedding of $\wt\ZZ^{2\dl}_T\subset \ZZ^\dl_T$ is compact.
Then, 
from \eqref{R1}, \eqref{R2}, and \eqref{R3}
along with the convergence of $\U_N$
to $\U$ in $L^2([0,T];H^{-1+\eps}(\T^2))$, 
we see that $\{v_N \}_{N \in \N}$
is bounded in $\wt\ZZ^{2\dl}_T$
and thus is precompact in $\ZZ_T^\dl$.
Hence, 
there exists a subsequence $\{v_{N_k}\}_{k \in \N}$
converging to some limit $ v$
in $\ZZ_T^\dl$.

Next, we show that the limit $ v$ satisfies the 
Duhamel formulation  \eqref{v3}.
In particular, we prove that 
 $\Phi_{N_k}(v_{N_k})$ converges to $\Phi( v)$ in $L^1([0,T];H^{-1 + \eps}(\T^2))$.
Write 
\begin{align}
\|\Phi_{N_k}(v_{N_k})-\Phi( v)\|_{L^1_TH^{-1 + \eps}_x}
& \les \bigg\|
\int_0^tP(t-t')\big(e^{\be z}F(\be v_{N_k})(\U_{N_k}-\U)\big)(t')dt'\bigg\|_{L^1_{T}H^{-1+ \eps}_x}
\notag\\
&\hphantom{X} +\bigg\|\int_0^tP(t-t')\big(e^{\be z}(F(\be v_{N_k}) - F(\be v))
\U \big)(t') dt'\bigg\|_{L^1_{T}H^{-1 + \eps}_x}\notag\\
&=: \1+\II.
\label{Dv0}
\end{align}

\noi
By the 
 Schauder estimate (Lemma \ref{LEM:heatker}), 
 Young's inequality, 
 and 
 Lemma \ref{LEM:prod1}\,(ii), we have 
 \begin{align}
 \begin{split}
\1 &\les 
\big\|e^{\be z}F(\be v_{N_k})(\U_{N_k}-\U)\big\|_{L^1_T W^{-1+\eps,1}_x}\\
& \les \big\|e^{\be z}F(\be v_{N_k})\big\|_{L^2_TW^{1-\eps,\frac1{1-\eps}}_x}\|\U_{N_k}-\U\|_{L^2_TH^{-1+\eps}_x}
\end{split}
\label{R4}
\end{align}

\noi
for sufficiently small $\eps>0$.

By  the fractional Leibniz rule (Lemma \ref{LEM:prod1}\,(i)), we have 
\begin{align}
\begin{split}
\big\|e^{\be z} & F(\be v_{N_k})  \big\|_{L^2_TW^{1-\eps,\frac1{1-\eps}}_x}\\
&\les \big\|e^{\be z}\|_{L^2_TH^{1-\eps}_x}
\big\|F(\be v_{N_k})\big\|_{L^{\infty}_{T,x}} 
+ \big\|e^{\be z}\big\|_{L^{\infty}_{T,x}}\big\|F(\be v_{N_k})\big\|_{L^2_TH^{1-\eps}_x}.
\end{split}
\label{R5}
\end{align}

\noi
By the fractional chain rule  (Lemma \ref{LEM:FC}\,(ii)), we have
\begin{align}
\begin{split}
\big\|e^{\be z}\|_{L^2_TH^{1-\eps}_x}
 &\sim \big\|e^{\be z}\|_{L^2_{T,x}}+\big\||\nabla|^{1-\eps}e^{\be z}\big\|_{L^2_{T,x}}\\
&\les T^{\frac12} e^{C\|z \|_{L^{\infty}_{T,x}}} 
+ \big\| e^{\be z}\big\|_{L^{\infty}_{T}L^4_x}
\big\||\nabla|^{1-\eps}z\big\|_{L^2_TL^{4}_x}\\
&\le C(T) e^{C\|v_0\|_{L^{\infty}}}\big(1+\|z\|_{L^2_TW^{1-\eps, 4}_x}\big)\\
&\le C(T) e^{C\|v_0\|_{L^{\infty}}}\big(1+\|v_0\|_{L^\infty}\big),
\end{split}
\label{R6}
\end{align}

\noi
where we used the Schauder estimate (Lemma \ref{LEM:heatker}) in the last step.
Similarly, by  the fractional chain rule 
(Lemma \ref{LEM:FC}\,(i))
along  with the boundedness of $F$, we have
\begin{align}
\begin{split}
\big\|F(\be v_{N_k})\big\|_{L^2_TH^{1-\eps}_x}
&\sim \big\|F(\be v_{N_k})\big\|_{L^2_{T,x}}+\big\||\nabla|^{1-\eps}F(\be v_{N_k})\big\|_{L^2_{T,x}}\\
&\les T^{\frac12}+\big\| |\nabla|^{1-\eps} v_{N_k}\big\|_{L^2_{T,x}}\\
&\le C(T) \big(1+\|v_{N_k}\|_{\ZZ_T^\dl}\big).
\end{split}
\label{R7}
\end{align}

Hence, putting \eqref{R4}, \eqref{R5}, \eqref{R6}, and \eqref{R7} together, 
we obtain
\begin{align}
\1 \les e^{C \|v_0\|_{L^{\infty}}}\Big(1 + \|v_0\|_{L^{\infty}}+\|v_{N_k}\|_{\ZZ_T^\dl}\Big)\|\U_{N_k}-\U\|_{L^2_TH^{-1+\eps}_x}.
\label{Dv1}
\end{align}

As for the second term $\II$ in \eqref{Dv0}, 
we use the fundamental theorem of calculus and write %mean value theorem and write 
\begin{align}
F(\be v_{N_k})-F(\be  v) = \be (v_{N_k} -  v) G(v_{N_k}, v), 
\label{F2}
\end{align}

\noi
where
\begin{align}
G(v_1,v_2) = \int_0^1F'\big(\tau \be v_1+(1-\tau)\be v_2\big)d\tau.
\label{F3}
\end{align}

\noi
Since $F$ is Lipschitz,  we see that  $G$ is bounded.
Since $v_{N_k},  v \in \ZZ_T^\dl$, 
we have $v_{N_k}(t),  v(t) \in C(\T^2)$
for almost every $t\in [0, T]$.
Then, by the  Schauder estimate (Lemma~\ref{LEM:heatker}), 
 Lemma~\ref{LEM:posprod}, 
 and H\"older's inequality,  we have
\begin{align}
\begin{split}
\II 
&\les \big\| e^{\be z}(v_{N_k}- v)G(v_{N_k}, v)\U\big\|_{L^1_TH^{-1+\eps}_x}
\\
&\les \big\|e^{\be z}(v_{N_k}- v)G(v_{N_k}, v)\big\|_{L^2_TL^{\infty}_x}
\|\U\|_{L^2_TH^{-1+\eps}_x}
\\
& \les
 e^{C\|z\|_{L^{\infty}_{T,x}}}\|v_{N_k}- v\|_{L^2_TL^{\infty}_x}\|G(v_{N_k}, v)\|_{L^{\infty}_{T,x}}
\|\U\|_{L^2_TH^{-1+\eps}_x}\\
&\les e^{C\|v_0\|_{L^{\infty}}}\|v_{N_k}- v\|_{\ZZ_T^\dl}\|\U\|_{L^2_TH^{-1+\eps}_x}.
\end{split}
\label{Dv2}
\end{align}

From \eqref{Dv0}, \eqref{Dv1}, and \eqref{Dv2}
along with the convergence of $v_{N_k}$ to $ v$ in $\ZZ_T^\dl$
and $\U_{N_k}$ to $\U$
in $L^2([0,T];H^{-1+\eps}(\T^2))$, 
we conclude that 
$\Phi_{N_k}(v_{N_k})$ converges to $\Phi( v)$ in $L^1([0,T];H^{-1 + \eps}(\T^2))$.
Since $v_{N_k}  = \Phi_N(v_{N_k})$, 
this shows that 
\[  v = \lim_{k \to \infty}v_{N_k}
=\lim_{k \to \infty} \Phi_{N_k}(v_{N_k}) = \Phi( v) 
\]

\noi
as distributions and hence as elements in $\ZZ_T^\dl$ since $v \in \ZZ_T^\dl$.
This proves existence of a solution
to \eqref{v3}
in $\ZZ_T^\dl \subset \ZZ_T$.

Lastly, we prove uniqueness of solutions to \eqref{v3} in the energy space $\ZZ_T$.
Let $v_1, v_2 \in \ZZ_T$ be two solutions to \eqref{v3}.
 Then, by setting  $w=v_1-v_2$, the difference  $w$ satisfies
\begin{align}
\dt w +\tfrac12(1-\Dl)w + \tfrac12\ld\be  e^{\be{z}}\big(F(\be v_1)-F(\be v_2)\big) \U = 0.
\label{NHX}
\end{align}

\noi
Since $\be v_j\leq 0$, $j = 1, 2$, it follows from \eqref{F1}
and \eqref{F3} that 
\[G(v_1,v_2) = \int_0^1\exp\big(\tau \be v_1+(1-\tau)\be v_2\big)d\tau\geq 0.\]

\noi
Now, define an energy functional:
\[\En(t) \deff \|w(t)\|_{L^2_x}^2 + \frac12\int_0^t\|w(t')\|_{H^1_x}^2dt'\geq 0.\]

\noi
Since $w \in \ZZ_T$, the energy functional $\En(t)$
is a well-defined differentiable function.
Moreover, with \eqref{NHX} and \eqref{F2}, we have 
\begin{align*}
\frac{d}{dt}\En(t)& = \int_{\T^2}w(t)
\big(2\dt w(t)+(1-\Dl)w(t)\big)dx \\
&= -\ld\be^2\int_{\T^2}w^2(t)  e^{\be{z}}G(v_1,v_2)\U(t)dx\\& \leq 0
\end{align*}

\noi
thanks to the positivity  of $G$ and $\U$ and the assumption that $\ld>0$.
Since $w(0) = 0$, 
we conclude that $\En(t) = 0$ for any $t \geq 0$
and $v_1 \equiv v_2$.
This proves uniqueness in the energy space $\ZZ_T$.

The solution $v \in \ZZ_T^\dl$ constructed in the existence part
depends a priori on a choice of a subsequence $v_{N_k}$.
The uniqueness in $\ZZ_T \supset \ZZ_T^\dl$, however,  shows that 
the limit $v$ is independent of the choice of a subsequence
and hence the entire sequence $\{v_N\}_{N \in \N}$
converges to $v$ in $\ZZ_T^\dl \subset \ZZ_T$.
This completes the proof of Proposition \ref{PROP:flow}.
\end{proof}

\subsection{On  invariance of the Gibbs measure}
\label{SUBSEC:Gibbs1}
In this subsection, 
 we briefly go over the proof of Theorem \ref{THM:2}. 
 Given $N\in\N$,  we consider the truncated SNLH
  \eqref{NH4}
 with initial data given by $u_N|_{t = 0} = w_0$, 
 where $w_0$ is as in \eqref{IV2} distributed
 by the massive Gaussian free field $\mu_1$.
 For this problem,  there 
 is no deterministic linear solution $z$
 and hence write $u_N$
 as   $u_N = v_N + \Psi^\text{heat}$. 
 Then, the residual term $v_N$ satisfies
 \begin{align}
v_N(t)=-\frac12\ld\be\int_0^tP(t-t')\Q_N\big(e^{\be \Q_Nv_N}\U_N\big)(t')dt',
\label{vQ}
\end{align}

\noi
where $\U_N$ is the Gaussian multiplicative chaos 
defined in terms of $\Q_N$.
Since the smoothing operator $\Q_N$ in \eqref{Q}
is equipped with a non-negative kernel, 
the equation \eqref{vQ} enjoys the sign-definite structure:
\begin{align*}
\be\Q_Nv_N(t)= -\frac12\ld\be^2\int_0^tP(t-t')\Q_N^2\big(e^{\be \Q_Nv_N}\U_N\big)(t')dt' \le 0.
\end{align*}

\noi
Namely, we can rewrite \eqref{vQ}
as 
 \begin{align}
v_N(t)=-\frac12\ld\be\int_0^tP(t-t')\Q_N\big(F(\be \Q_Nv_N)\U_N\big)(t')dt', 
\label{vQ1}
\end{align}

\noi
where $F$ is as in \eqref{F1}.

In view of the uniform (in $N$) boundedness of $\Q_N$ on $L^p(\T^2)$, $1\leq p \leq \infty$, 
we  can argue as in Subsection~\ref{SUBSEC:para2}
to prove 
local well-posedness
of \eqref{vQ1}
and establish an a priori bound on $\{v_N\}_{N\in \N}$ in $\wt \ZZ_T^{2\dl} \subset \ZZ_T^\dl$.
Then, by the Aubin-Lions lemma (Lemma \ref{LEM:AL}), 
we see that there exists a subsequence $\{ v_{N_k}\}_{k \in \N}$
converging to some limit $ v$ in $\ZZ_T^\dl$.
Moreover, the uniqueness argument 
for solutions to the limiting equation \eqref{v3} remains true.
 Therefore, in view of the argument in Subsection \ref{SUBSEC:para2}, 
 it suffices
 to show that the limit $ v$
 satisfies the equation~\eqref{v3}.

With a slight abuse of notation, 
 let $\Phi_{N_k}$ denotes the right-hand side of \eqref{vQ1}:
 \begin{align}
  \Phi_{N_k} (v_{N_k})(t)\deff-\frac12\ld\be\int_0^tP(t-t')\Q_{N_k}\big(F(\be \Q_{N_k}v_{N_k})\U_{N_k}\big)(t')dt'.
\label{V0a}
  \end{align} 
 
 \noi
Then, it suffices to 
 show that 
 $\Phi_{N_k}(v_{N_k})$ converges to $\Phi( v)$ in $L^1([0,T];H^{-1 }(\T^2))$, 
 where $\Phi = \Phi_{v_0, \U}$ is as in \eqref{D2}
 (with $v_0 = 0$).
From \eqref{D2} and \eqref{V0a}, we have 
\begin{align}
\|\Phi_{N_k}(v_{N_k}) & -\Phi( v)\|_{L^1_TH^{-1 }_x}\notag \\
& \les \bigg\|
\int_0^tP(t-t')\big(F(\be \Q_{N_k} v_{N_k})(\U_{N_k}-\U)\big)(t')dt'\bigg\|_{L^1_{T}H^{-1}_x}
\notag\\
&\hphantom{X} +
\bigg\|\int_0^tP(t-t')
\big((F(\be \Q_{N_k} v_{N_k}) - F(\be v))
\U \big)(t') dt'\bigg\|_{L^1_{T}H^{-1 }_x}\notag\\
&\hphantom{X} +
\bigg\|\int_0^tP(t-t')(\Q_{N_k}-\Id)\big(F(\be \Q_{N_k}v_{N_k})\U_{N_k}\big)(t')dt'\bigg\|_{L^1_TH^{-1}_x} \notag \\
&=: \1+\II + \III.
\label{V0}
\end{align}

\noi
The terms $\1$ and $\II$
can be handled exactly as in Subsection \ref{SUBSEC:para2}
and, hence, it remains
to treat  the extra term $\III$.

When viewed as a Fourier multiplier operator, 
the symbol  for $\Q_N$ is given by $2\pi \ft \rho_N$; see~\eqref{Q}.
Note that,  for $0<s_1-s<1$,  the symbol
\begin{align}
m_N(\xi) \deff N^{s_1-s}\jb{\xi}^{s-s_1}\big(2\pi\ft\rho_N(\xi)-1\big)
\label{V0b}
\end{align}

\noi
satisfies the bound
\begin{align}
|\dd_\xi^{k}m_N(\xi)|\les \jb{\xi}^{-|k|}
\label{V1}
\end{align}

\noi
for any $k\in (\Z_{\geq 0})^2$. Indeed, when no derivatives hits $2\pi\ft\rho_N -1$, 
we can use the mean value theorem (as $2\pi\ft\rho(0)=1$) to get the bound 
\begin{align*}
\big|N^{s_1-s}\dd_\xi^{k}(\jb{\xi}^{s-s_1})\cdot\big(2\pi\ft\rho_N(\xi)-1\big)\big|
& \les N^{s_1-s}\jb{\xi}^{s-s_1-|k|}\big(1\wedge N^{-1}|\xi|\big)\\
& \le\jb{\xi}^{-|k|}, 
\end{align*}

\noi
whereas when at least one derivative hits $2\pi \ft\rho_N-1$,
we gain a negative power of $N$ from $\ft \rho_N(\xi) = \ft \rho(N^{-1} \xi)$
and we use the fast decay of $\ft\rho$ and its derivatives;
with $|\al| + |\be| = |k|$, we have
\begin{align*}
\big|N^{s_1-s}\dd_\xi^{\al}(\jb{\xi}^{s-s_1})\cdot
\dd_\xi^\be \big(2\pi\ft\rho_N(\xi)-1\big)\big|
& \les N^{s_1-s - |\be|}\jb{\xi}^{s-s_1-|\al|}\cdot \big(N|\xi |^{-1}\big)^{s - s_1 + |\be|}\\
& \les\jb{\xi}^{-|k|}, 
\end{align*}

\noi
verifying \eqref{V1}.

Hence,  by the transference principle 
(\cite[Theorem 4.3.7]{Gra1})
 and the Mihlin-H\"ormander multiplier theorem (\cite[Theorem 6.2.7]{Gra1}), 
 the Fourier multiplier operator $N^{s_1-s}\jb{\nabla}^{s-s_1}\big(\Q_N-\Id\big)$ 
 with the symbol $m_N$ in \eqref{V0b}
 is bounded from $L^p(\T^2)$ to $L^p(\T^2)$ for any $1<p<\infty$ with norm independent of $N$. This implies that the following estimate holds:
\begin{align}
\|(\Q_N-\Id)f\|_{W^{s,p}(\T^2)}\les N^{s-s_1}\|f\|_{W^{s_1,p}(\T^2)}
\label{QN-1}
\end{align}

\noi
for any $0<s_1-s<1$ and $1<p<\infty$. 
Then, applying  \eqref{QN-1} and Lemma \ref{LEM:posprod} again, we can bound $\III$ in \eqref{V0} by
\begin{align}
\begin{split}
\big\|\big(\Q_{N_k}-\Id & \big)\big( F(\be  \Q_{N_k}v_{N_k})\U_{N_k}\big)\big\|_{L^1_TH^{-1}_x} \\
&\les N_k^{-\eps}
\|F(\be \Q_{N_k}v_{N_k})\|_{L^2_T L^\infty_x}
\|\U_{N_k}\|_{L^2_TH^{-1+\eps}_x}\\
&\les {N_k}^{-\eps}
\|\U_{N_k}\|_{L^2_TH^{-1+\eps}_x}.
\end{split}
\label{Dv4}
\end{align}

\noi
Hence, from \eqref{V0}, the convergence of $\1$ and  $\II$ to $0$
as shown in Subsection \ref{SUBSEC:para2}, 
and \eqref{Dv4}, 
we conclude that 
 $\Phi_{N_k}(v_{N_k})$ in \eqref{V0a}
 converges to $\Phi( v)$ in $L^1([0,T];H^{-1 }(\T^2))$.
 Combined with the uniqueness
 of the solution to \eqref{v3} in $\ZZ_T$, 
 this shows that the solution $v_N$ to the truncated SNLH \eqref{vQ1}
converges to the solution $v$ to SNLH \eqref{v3}
(with $z = 0$).

\medskip

Lastly, we establish  invariance of the Gibbs measure $\rh$ 
constructed in Proposition~\ref{PROP:Gibbs} under the dynamics of SNLH \eqref{NH1}. In the following, we write $\Phi_N(t)$ and $\Phi(t)$ for the flow maps of the truncated SNLH \eqref{NH4} and SNLH \eqref{NH1}, respectively, constructed above. 
Note that $\Phi(t)(u_0)$ is  interpreted as $\Phi(t)(u_0)=\Psi + v$, 
where  $\Psi$ is the stochastic convolution defined in \eqref{conv} (with $w_0 = u_0$) and $v$ is the solution to \eqref{v1} (with $z\equiv 0$).
In the remaining part of this section, we take 
  the space-time white noise $\xi = \xi^\om$
in the equation to be on 
a probability space $(\O_1,\PP)$ 
and use  $\om$ to denote the randomness coming from  the space-time white noise.
Moreover, we use $\E_\om$ to denote an expectation with respect to the noise, 
namely, integration with respect to the probability measure $\PP$.
In the following, 
we write $\Phi^\om(t)(u_0)$, when we emphasize the dependence of the solution 
on the noise.
A similar comment applies to  $\Phi_N(t)$. 
Given $N \in \N$, 
we use $\Pc^N_t$ to denote the Markov semigroup associated with 
the truncated dynamics $\Phi_N^\om(t)$:
\[ \Pc^N_t(F)(u_0) = \E_\om \big[F(\Phi_N^\om(t) (u_0))\big]
= \int_{\O_1} F(\Phi_N^\om(t) (u_0)) d\PP(\om).\]

We  first  show invariance of the truncated Gibbs measure $\rhN$ 
in \eqref{Gibbs1a}
under the truncated dynamics \eqref{NH4}. % $\Phi_N$.
\begin{lemma}\label{LEM:inv}
Let  $N \in \N$ and $\eps>0$.
Then, for  any  continuous and bounded function $F : H^{-\eps}(\T^2)\to\R$, we have 
\begin{align*}
\int   \Pc^N_t(F)(u_0)d\rhN(u_0) = \int F(u_0)d\rhN(u_0).
\end{align*}
\end{lemma}
\begin{proof}
Since the truncated Gibbs measure $\rhN$ in \eqref{Gibbs1a} truncated by $\Q_N$ does not have a finite Fourier support, we first approximate it by 
\begin{align}
d \rho_{\text{heat}, N, M} 
& = Z_{N,M}^{-1}\exp\bigg(-\ld C_{N,M}\int_{\T^2}e^{\be \P_M \Q_N u}dx\bigg)d\mu_1,
%\begin{split}
%\rhN & = \lim_{M\to\infty} \rho_{\text{heat}, N, M} \\
%& = \lim_{M\to\infty}Z_{N,M}^{-1}\exp\bigg(-\ld C_{N,M}\int_{\T^2}e^{\be \P_M \Q_N u}dx\bigg)d\mu_1,
%\end{split}
\label{XB1}
\end{align}

\noi
where $\P_M$ is the Fourier multiplier with a compactly supported symbol 
$\chi_N$ in \eqref{chi} and 
\begin{align*}
C_{N,M} = e^{-\frac{\be^2}2\s_{N,M}}=e^{-\frac{\be^2}2\E[ (\P_M \Q_N \Psi^\text{heat}(t, x))^2]}
\longrightarrow C_N, 
%\label{XB1a}
\end{align*}
 as $M\to\infty$. 
 Here, $\Psi^\text{heat}$
is as in  \eqref{Ph1}.

 Let 
  \begin{align*}
 \U_N = e^{-\frac{\be^2}2 \s_N}e^{\be\Q_N \Psi^\text{heat}}
 \qquad \text{and}\qquad
 \U_{N, M} = e^{-\frac{\be^2}2 \s_N}e^{\be\P_N \Q_N \Psi^\text{heat}}.
%\label{XB1b}
 \end{align*} 

\noi
Then,  a slight modification of the proof of Proposition \ref{PROP:Ups}
 shows that 
 $\ft \U_{N, M}(0, 0)$  converges to 
  $\ft \U_{N}(0, 0)$ in $L^p(\O)$
  for $1\le p < \frac{8\pi}{\be^2}$.
 Namely, we have 
\begin{align} - \ld C_{N,M}\int_{\T^2}e^{\be \Q_N\P_M u}dx \too
-  \ld  C_{N}\int_{\T^2}e^{\be \Q_N u}dx
\label{XB1c}
\end{align}

\noi
in $L^p(\mu_1)$ 
  for $1\le p < \frac{8\pi}{\be^2}$
  and also in probability. 
Let $R_N$ be as in \eqref{W7}
and define $R_{N, M}$ by 
\begin{align}
R_{N, M} =   \exp\bigg(- \ld C_{N, M} \int_{\T^2}e^{\be \, \P_N \Q_N u} dx \bigg).
\label{XB1d}
\end{align}

\noi
Then, it follows from \eqref{XB1c}
that $R_{N, M}$ converges to $R_N$ in probability as $M \to \infty$.
Moreover, by the positivity of $\U_N$, $\U_{N, M}$, 
and $\ld$, 
the densities $R_N$ and $R_{N, M}$ are uniformly bounded by 1.
As in the proof of  Proposition \ref{PROP:Gibbs}, 
this implies 
 the $L^p(\mu_1)$-convergence of the density 
$R_{N, M}$ to $R_N$ as $M \to \infty$, 
 which in turn shows  convergence in total variation $\rho_{\text{heat}, N, M}\to\rhN$ as $M\to\infty$.

\medskip

Next, 
%note that for each fixed $N\in\N$, $\rhN$  is supported on smooth functions.
% Hence, without using the sign-definite structure of the equation, we can prove, 
consider the truncated dynamics 
\eqref{NH4} with the Gaussian initial data $\Law (u_N(0)) = \mu_1$.
Then, proceeding 
 as in the proof of  Theorem~\ref{THM:1.2},  
 we see that 
  the flow $\Phi_N$ of \eqref{NH4}  is a limit in probability 
  (with respect to $\PP\otimes \mu_1(d\om, du_0)$) in $C([0,T];H^{-\eps}(\T^2))$, $\eps>0$,
 of the flow $\Phi_{N,M}$ for the following truncated  dynamics: 
 \begin{align}
\begin{cases}\dt u_{N, M} + \tfrac 12 (1- \Dl)  u_{N, M}   +  \tfrac 12 \ld \be C_{N, M} \P_M \Q_N  e^{\be \, \P_M \Q_N 
u_{N, M}} = \xi\\
u_{N,M}|_{t=0}=u_0 \quad \text{with } \Law (u_0) = \mu_1.
\end{cases}
\label{vQ2}
\end{align}

Let us now discuss 
 invariance of $\rho_{\text{heat}, N, M}$ under \eqref{vQ2}.
Let 
 $\Pi_M$ be the sharp Fourier truncation on frequencies $\{|n|\le M\}$.
 Then, 
from the  definition \eqref{chi} of $\P_M$, 
we have  $\Pi_{2M}\P_M = \P_M$ for any $M\in\N$.
In particular, with $\Pi_{2M}^\perp = \Id-\Pi_{2M}$, 
we have $\Pi_{2M}^\perp\P_M = 0$. 
Then, it follows 
from \eqref{XB1} that 
the pushforward measure 
 $(\Pi_{2M}^\perp)_\#\rho_{\text{heat}, N, M} $ is Gaussian:
 \[(\Pi_{2M}^\perp)_\#\rho_{\text{heat}, N, M} = (\Pi_{2M}^\perp)_\#\mu_1.\]
 
\noi
Hence, we have the following decomposition:
\begin{align*}
\rho_{\text{heat}, N, M} = (\Pi_{2M})_\#\rho_{\text{heat}, N, M}\otimes (\Pi_{2M}^\perp)_\#\mu_1.
\end{align*}

By writing
 \begin{align*}
 u_{N, M} = \Pi_{2M} u_{N, M}
+ \Pi_{2M}^\perp  u_{N, M}
= : u^{(1)} + u^{(2)},  
 \end{align*}

\noi
where, for simplicity,  we dropped the subscripts on the right-hand side, 
we see that the high frequency part $u^{(2)}$ satisfies the linear stochastic heat equation:
 \begin{align}
\dt u^{(2)} + \tfrac 12 (1- \Dl)  u^{(2)}    = \Pi_{2M}^\perp\xi.
\label{XB2}
\end{align}

\noi
Since this is a linear equation where spatial frequencies are decoupled,\footnote{In particular, 
by writing \eqref{XB2} on the Fourier side, we see that 
$\ft{u^{(2)}}(n)$ is 
 the (independent) Ornstein-Uhlenbeck process for each frequency
 whose invariant measure is Gaussian.} 
it is easy to check that the Gaussian measure
$ (\Pi_{2M}^\perp)_\#\mu_1$ is invariant under \eqref{XB2}.

The low frequency part $u^{(1)}$ satisfies the following equation:
 \begin{align}
\dt u^{(1)} + \tfrac 12 (1- \Dl)  u^{(1)}   + \NN(u^{(1)}) = \Pi_{2M}\xi, 
\label{XB3}
\end{align}

\noi
where the nonlinearity $\NN = \NN_{N, M}$ is given by 
\begin{align}
\NN(u) = \NN_{N, M}(u) =  \tfrac 12 \ld \be C_{N, M} \P_M \Q_N  e^{\be \, \P_M \Q_N 
u}. 
\label{XB4}
\end{align}

\noi
On the Fourier side, \eqref{XB3} is a finite-dimensional system
of SDEs.
As such, one can easily check by hand that 
$(\Pi_{2M})_\#\rho_{\text{heat}, N, M}$ is invariant under \eqref{XB3}.
In the following, we review this argument.

In the current real-valued setting, we have 
$\ft{u^{(1)}}(-n) = \cj{\ft{u^{(1)}}(n)}$.
Then, 
by writing $\ft{u^{(1)}}(n) = a_n + i b_n$ for $a_n, b_n \in \R$, 
we have
\begin{align}
a_{-n} = a_n\qquad \text{and}\qquad b_{-n} = - b_n.
\label{XB4a} 
\end{align}

\noi
Defining the index sets $\Ld = \Ld(2M) \subset \Z^2$ and $\Ld_0 = \Ld_0(2M) \subset \Z^2$, $M \in \N$:
 \[\Ld = \big\{(\N \times \{0\}) \cup (\Z\times \N) \big\} \cap \{n \in \Z^2: |n|\leq 2M\} \qquad\text{and}\qquad
\Ld_0 = \Ld \cup \{(0, 0)\}, \]

\noi
we can write \eqref{XB3} as 
\begin{align}
\begin{split}
d a_n & = \Big(-\tfrac 12 \jb{n}^2 a_n - \Re \ft {\NN(u^{(1)})}(n) \Big) dt + d (\Re B_n)\\
d b_n & = \Big(-\tfrac 12 \jb{n}^2 b_n - \Im \ft {\NN(u^{(1)})}(n) \Big) dt + d (\Im B_n)\\
\end{split}
\label{XB5}
\end{align}

\noi
for $n \in \Ld$ and 
\begin{align}
d a_0 & = \Big(-\tfrac 12  a_0 - \ft {\NN(u^{(1)})}(0)\Big) dt  + d  B_0.
\label{XB6}
\end{align}

\noi
Here,  $\{B_n\}_{n \in \Ld_0}$ is 
a family of mutually independent
complex-valued Brownian motions  as in \eqref{Wiener}.
Note that $\Var(\Re B_n(t)) = \Var(\Im B_n(t)) = \frac t2$
for $n \in \Ld$, while 
$\Var( B_0(t)) = t$.

Let $F$ be a continuous and bounded  function on $(\bar a, \bar b) = (a_m, b_n)_{m \in \Ld_0, n \in \Ld} 
\in \R^{2|\Ld| + 1}$.
Then, by Ito's lemma, the generator $\L = \L_{N, M}$
of the Markov semigroup associated with~\eqref{XB5} and \eqref{XB6}
is given by 
\begin{align}
\begin{split}
\L F(\bar a, \bar b)
& =  \sum_{n \in \Ld_0} \bigg[ \Big(- \frac 12 \jb{n}^2 a_n 
- \Re  \ft {\NN(u^{(1)})}(n)\Big)
\dd_{a_n}F(\bar a, \bar b)
+ \frac 14 \dd^2_{a_n}F(\bar a, \bar b)\bigg]\\
& \quad
+   \sum_{n \in \Ld} \bigg[ \Big(- \frac 12 \jb{n}^2 b_n 
- \Im  \ft {\NN(u^{(1)})}(n)\Big)
\dd_{b_n}F(\bar a, \bar b)
+ \frac 14 \dd^2_{b_n}F(\bar a, \bar b)\bigg]\\
& \quad 
+  \frac 14 \dd^2_{a_0}F(\bar a, \bar b).
\end{split}
\label{XB7}
\end{align}

\noi
The last term  takes into account  the 
different forcing in  \eqref{XB6}.
In order to prove invariance of 
$\rho_{\text{heat}, N, M}^\text{low}
 \deff (\Pi_{2M})_\#\rho_{\text{heat}, N, M}$
 under the low-frequency dynamics \eqref{XB3}, 
 it suffices to prove  
\[(\L)^*\rho_{\text{heat}, N, M}^\text{low}
= 0.\]

\noi
By viewing
$\rho_{\text{heat}, N, M}^\text{low}$ as a measure on $(\bar a, \bar b)$
with a slight abuse of notation, 
 this is equivalent to proving 
\begin{align}
\int \L F(\bar a, \bar b) \, d\rho_{\text{heat}, N, M}^\text{low} (\bar a, \bar b)
= 
\int \L F(\bar a, \bar b) e^{-\M(u^{(1)})} d\bar ad\bar b
= 0, 
\label{XB7a}
\end{align}

\noi
where  $\M(u^{(1)}) $ is given by 
\begin{align}
\M(u^{(1)}) 
= 
\ld C_{N,M}\int_{\T^2}e^{\be \P_M \Q_N u^{(1)}} dx
+ \sum_{n \in \Ld} \jb{n}^2 (a_n^2 + b_n^2)
+ \frac 12 a_0^2.
\label{XC3}
\end{align}

A direct computation with \eqref{XB4a} shows
\begin{align}
\begin{split}
2\pi \dd_{a_n}\F_x\big[(\P_M \Q_N u^{(1)})^k\big](0)
& = 2k \Re
\F_x\Big[  \P_M \Q_N \big((\P_M \Q_N u^{(1)})^{k-1}\big)\Big](n), \\
2\pi \dd_{b_n}\F_x\big[(\P_M \Q_N u^{(1)})^k\big](0)
& =  2k \Im
\F_x\Big[  \P_M \Q_N \big((\P_M \Q_N u^{(1)})^{k-1}\big)\Big](n)
\end{split}
\label{XC1}
\end{align}

\noi
for $n \in \Ld$
and 
\begin{align*}
2\pi  \dd_{a_0}\F_x\big[(\P_M \Q_N u^{(1)})^k\big](0)
& = k 
\F_x\Big[  \P_M \Q_N \big((\P_M \Q_N u^{(1)})^{k-1}\big)\Big](0).
%\label{XC2}
\end{align*}

\noi
By the Taylor expansion with  \eqref{XC1}  and \eqref{XB4}, we have 
\begin{align}
\begin{split}
\dd_{a_n} \bigg(\ld C_{N,M} & \int_{\T^2}e^{\be \P_M \Q_N u^{(1)}} dx\bigg)
 = \ld C_{N,M}\cdot 2\pi \dd_{a_n} \F_x[ e^{\be \P_M \Q_N u^{(1)}}](0)\\
& = \ld C_{N,M}\cdot  2\pi \sum_{k = 0}^\infty \frac{\be^k \dd_{a_n} \F_x\big[(\P_M \Q_N u^{(1)})^k\big](0)}{k!}\\
& = 4 \Re \ft{\NN(u^{(1)})}(n) 
\end{split}
\label{XC4}
\end{align}

\noi
for $n \in \Ld$.
By a similar computation, we have 
\begin{align}
\begin{split}
\dd_{b_n} \bigg( \ld C_{N,M}\int_{\T^2}e^{\be \P_M \Q_N u^{(1)}} dx\bigg)
& = 4 \Im \ft{\NN(u^{(1)})}(n) 
\end{split}
\label{XC5}
\end{align}

\noi
for $n \in \Ld$, 
and
\begin{align}
\begin{split}
\dd_{a_0} \bigg( \ld C_{N,M}\int_{\T^2}e^{\be \P_M \Q_N u^{(1)}} dx\bigg)
& = 2 \ft{\NN(u^{(1)})}(0) .
\end{split}
\label{XC6}
\end{align}

\noi
Then, using
\eqref{XB7}, \eqref{XC3}, \eqref{XC4}, \eqref{XC5}, and \eqref{XC6}, 
we can rewrite the generator $\L$ as
\begin{align}
\begin{split}
\L F(\bar a, \bar b)
& =  \sum_{n \in \Ld} \bigg[ 
- \frac 14 \dd_{a_n} \M(u^{(1)}) 
\dd_{a_n}F(\bar a, \bar b)
+ \frac 14 \dd^2_{a_n}F(\bar a, \bar b)\\
& \hphantom{XXXX}- \frac 14 \dd_{b_n} \M(u^{(1)}) 
\dd_{b_n}F(\bar a, \bar b)
+ \frac 14 \dd^2_{b_n}F(\bar a, \bar b)
\bigg]\\
& \quad 
- \frac 12 \dd_{a_0} \M(u^{(1)})  
\dd_{a_0}F(\bar a, \bar b)
 + \frac 12 \dd^2_{a_0}F(\bar a, \bar b).
\end{split}
\label{XC7}
\end{align}

\noi
Then, with \eqref{XC3} and \eqref{XC7}, 
integration by parts  yields
\begin{align*}
\int &  \L F(\bar a, \bar b)  e^{-\M(u^{(1)})} d\bar ad\bar b\\
& = \frac 14 \sum_{n \in \Ld} \int
\Big( \dd_{a_n} e^{-\M(u^{(1)})}\cdot 
\dd_{a_n}  F(\bar a, \bar b)
 +  \dd^2_{a_n} F(\bar a, \bar b)\cdot 
 e^{-\M(u^{(1)})}\Big) 
d\bar ad\bar b\\
& \quad +  \frac 14 \sum_{n \in \Ld} \int
\Big( \dd_{b_n} e^{-\M(u^{(1)})} \cdot 
\dd_{b_n}  F(\bar a, \bar b)
 +  \dd^2_{b_n} F(\bar a, \bar b)\cdot 
 e^{-\M(u^{(1)})}\Big) 
d\bar ad\bar b\\
& \quad +  \frac 12 \int
\Big( \dd_{a_0} e^{-\M(u^{(1)})}\cdot 
\dd_{a_0}  F(\bar a, \bar b)
 +  \dd^2_{a_0} F(\bar a, \bar b)\cdot 
 e^{-\M(u^{(1)})}\Big) 
d\bar ad\bar b\\
& = 0.
\end{align*}

\noi
This proves \eqref{XB7a}
and hence invariance of 
$\rho_{\text{heat}, N, M}^\text{low}
=  (\Pi_{2M})_\#\rho_{\text{heat}, N, M}$
 under the low-frequency dynamics \eqref{XB3}.

\medskip

We are now ready to prove invariance of $\rhN$ under $\Phi_N(t)$.
This follows from 

\begin{itemize}
\item[(i)] 
the convergence of $\rho_{\text{heat}, N, M}$ to $ \rhN$ in total variation, 

\smallskip
\item[(ii)]
the convergence
of  $\Phi^\om_{N,M}(t)(u_0)$ to $\Phi_N^\om(t)(u_0)$ in probability with respect to $\PP\otimes \mu_1(\om, u_0)$, and

\smallskip
\item[(iii)]
the invariance of $\rho_{\text{heat}, N, M} $ under $\Phi_{N,M}(t)$.

\end{itemize}

\noi
Indeed, for any $F : H^{-\eps}(\T^2)\to \R$,  continuous and bounded, and any $t\ge 0$, we have
\begin{align}
\bigg|\int &  \E_\om\big[F(\Phi_N^\om(t)(u_0))\big]d\rhN(u_0) 
- \int \E_\om\big[F(\Phi_{N,M}^\om(t)(u_0))\big]d\rho_{\text{heat}, N, M}(u_0)\bigg| \notag \\
&\le Z_{N,M}^{-1}\int \Big|
\E_\om\big[F(\Phi_N^\om(t)(u_0))\big]
- \E_\om\big[F(\Phi_{N, M}^\om(t)(u_0))\big]\Big|
R_{N, M}(u_0) d\mu_1(u_0) \label{XD1}\\
&\quad +
\bigg|\int \E_\om\big[F(\Phi_N^\om(t)(u_0))\big]d\rhN(u_0) 
- \int \E_\om\big[F(\Phi_N^\om(t)(u_0))\big]d\rho_{\text{heat}, N, M}(u_0)\bigg|,\notag
\end{align}

\noi
where $R_{N, M}$ is as in \eqref{XB1d}.
The second term on the right-hand side 
tends to 0 as $M\to \infty$ 
since  $\rho_{\text{heat}, N, M}$ converges to $\rhN$ in total variation. 
As for the first term, 
by the uniform bound $R_{N, M}\leq 1$, 
 we have 
\begin{align*}
& \int \Big|
\E_\om\big[F(\Phi_N^\om(t)(u_0))\big]
- \E_\om\big[F(\Phi_{N, M}^\om(t)(u_0))\big]\Big|
R_{N, M}(u_0) d\mu_1(u_0) \\
&\leq \int \big|
F(\Phi_N^\om(t)(u_0))
- F(\Phi_{N, M}^\om(t)(u_0))\big|
d(\PP\otimes \mu_1)(\om, u_0)\\
&\leq \dl + 2\|F\|_{L^\infty}
\cdot \PP\otimes \mu_1 \Big(
 \big|
F(\Phi_N^\om(t)(u_0))
- F(\Phi_{N, M}^\om(t)(u_0))\big| >\dl\Big)
\end{align*}

\noi
for any $\dl > 0$.
In view of the convergence
of  $\Phi^\om_{N,M}(t)(u_0)$ to $\Phi_N^\om(t)(u_0)$ in probability with respect to $\PP\otimes \mu_1(\om, u_0)$ as $M \to \infty$, we then obtain
\begin{align*}
 \lim_{M\to\infty} \int \Big|
\E_\om\big[F(\Phi_N^\om(t)(u_0))\big]
- \E_\om\big[F(\Phi_{N, M}^\om(t)(u_0))\big]\Big|
R_{N, M}(u_0) d\mu_1(u_0)
\leq \dl 
\end{align*}

\noi
Since the choice of $\dl > 0$ was arbitrary, 
we conclude that 
\begin{align}
 \lim_{M\to\infty} \int \Big|
\E_\om\big[F(\Phi_N^\om(t)(u_0))\big]
- \E_\om\big[F(\Phi_{N, M}^\om(t)(u_0))\big]\Big|
R_{N, M}(u_0) d\mu_1(u_0)
= 0.
\label{XD2}
\end{align}

\noi
Hence, from \eqref{XD1}, \eqref{XD2}, 
and $Z_{N, M}\to Z_N$
together with the invariance of 
$\rho_{\text{heat}, N, M} $ under $\Phi_{N, M}(t)$, we obtain
\begin{align*}
\int \E_\om\big[F(\Phi_N^\om(t) & (u_0))\big]d\rhN(u_0) 
=\lim_{M\to\infty}\int
\E_\om\big[F(\Phi_{N, M}^\om(t)(u_0))\big]d\rho_{\text{heat}, N, M}(u_0)\\
&=\lim_{M\to\infty}\int F(u_0)d\rho_{\text{heat}, N, M}(u_0) = \int F(u_0)d\rhN(u_0).
\end{align*}

\noi
 This concludes the proof of Lemma~\ref{LEM:inv}.
\end{proof}

With Lemma~\ref{LEM:inv}, we can finally prove invariance of the Gibbs measure 
 $\rh$ in Theorem~\ref{THM:2}. Indeed, 
proceeding as in  the proof of Lemma~\ref{LEM:inv} above, 
we can easily deduce
 invariance of the Gibbs measure 
 $\rh$ 
 from 
\begin{itemize}
\item[(i)]
  the convergence of the truncated Gibbs measures $\rhN$ to the Gibbs measure $\rh$ 
  in total variation (Proposition \ref{PROP:Gibbs}),

\smallskip

\item[(ii)]
 the convergence in probability (with respect to $\PP\otimes \mu_1$) of the truncated dynamics \eqref{vQ1} 
 to the full dynamics \eqref{v3} (with $z = 0$)
(Theorem \ref{THM:1.2}), and 

\smallskip

\item[(iii)]
the invariance of the truncated Gibbs measure $\rhN$ in \eqref{Gibbs1a} under the truncated 
 SNLH~\eqref{NH4} (Lemma~\ref{LEM:inv}).

\end{itemize}

\noi
(We also use the  absolute continuity 
of  the truncated Gibbs measure $\rhN$ 
 with respect to  the massive Gaussian free field $\mu_1$, 
 with the uniformly (in $N$) bounded density $R_N \leq 1$.)
This concludes the proof of Theorem~\ref{THM:2}.

\section{Hyperbolic Liouville equation}\label{SEC:wave}
In this  section, 
we study the  stochastic damped nonlinear wave equation
 \eqref{NW1} with the exponential nonlinearity.
We restrict our attention to the defocusing  case ($\ld > 0$).

\subsection{Local well-posedness of SdNLW}\label{SUBSEC:wave1}
In this subsection, we present a proof of Theorem~\ref{THM:wave}
on local well-posedness of the system \eqref{XY3}:
\begin{align}
\begin{split}
X(t)&= \Phi_1(X, Y) \\
& \deff-\ld\be\int_0^te^{-\frac{(t-t')}{2}} S(t-t')\big(e^{\be z}F(\be X)e^{\be Y}\U\big)(t')dt',\\
Y(t)&= \Phi_2(X, Y)\\ & \deff
-\ld\be\int_0^t\big(\D(t-t')-e^{-\frac{(t-t')}{2}}S(t-t')\big)
\big(e^{\be z}F(\be X)e^{\be Y}\U\big)(t')dt', 
\end{split}
\label{XY3a}
\end{align}

\noi
where $F$ is as in \eqref{F1}
and $\U$ is a positive distribution in $L^p([0, 1]; W^{-\al, p}(\T^2))$
with $\al$ and
  $1<  p < \frac{8\pi}{\be^2}$  satisfying \eqref{al}.
Here,
$\D(t)$
and 
$S(t)$ are the linear propagators defined in  \eqref{D} 
and~\eqref{S1}
and $z$ denotes the linear solution 
in \eqref{zW} with initial data 
 $(v_0,v_1)\in\H^{s}(\T^2)$ for some $s>1$.
 
 We prove local well-posedness of \eqref{XY3a}
 by a contraction argument
 for  $(X, Y) \in \X^{s_1}_T\times\Y^{s_2}_T$, 
 where the Strichartz type spaces
$  \X^{s_1}_T$ and $\Y^{s_2}_T$
are defined in \eqref{XX1} and \eqref{YY1}
  for some $\frac 14<s_1<\frac 34$ and $1<s_2<2$ (to be chosen later).
See also \eqref{StriX}.
In the following, we fix 
 the following $s_1$-admissible pair $(q, r)$ and dual $s_1$-admissible pair $(\wt q, \wt r)$ 
 (see  Definition~\ref{DEF:pair} for (dual) admissible pairs):
\begin{align}
(q,r) = \bigg(\frac3{s_1},\frac6{3-4s_1}\bigg) 
\qquad\text{and}\qquad (\wt q,\wt r)=\bigg(\frac3{2+s_1},\frac6{7-4s_1}\bigg).
\label{pair1}
\end{align}

\noi
We also fix  $p\ge 2$, $0<\al\le \min(s_1,1-s_1) < 1$, 
$1\le \wt q \le \wt q_1 \le 2 \le q \le q_1 \le \infty$,  and $1 \le \wt r \le \wt r_1 \le 2 \le r_1 \le r < \infty$,  satisfying the following constraints:

\smallskip

\begin{itemize}
\item[(i)] For  the interpolation lemma (Lemma \ref{LEM:interpol}):
\begin{align}
\begin{split}
& \frac{1}{q_1} =\frac{1-\al/s_1}q + \frac{\al/s_1}\infty,\qquad  
\frac{1}{r_1}=\frac{1-\al/s_1}r+\frac{\al/s_1}{2},\\
& \frac{1}{\wt q_1} =\frac{1-\al/(1-s_1)}{\wt q}+\frac{\al/(1-s_1)}1,
  \qquad\frac{1}{\wt r_1}=\frac{1-\al/(1-s_1)}{\wt r}+\frac{\al/(1-s_1)}{2},
\end{split}
\label{CC8}
\end{align}

\smallskip
\item[(ii)]
For 
Lemmas \ref{LEM:prod1}\,(ii) and \ref{LEM:posprod}:
\begin{align}
 & \frac1{r_1}+\frac1p \leq  \frac1{\wt r_1}+\frac{\al}2, 
\label{CC10}
\end{align}

\smallskip
\item[(iii)] For H\"older's inequality in time
$\| fg \|_{L^{\wt q_1}_T} \les T^\ta \| f\|_{L^{q_1}_T} \| g\|_{L^p_T}$ 
for some $\ta > 0$:
\begin{align}
& \frac1{q_1}+\frac1p < \frac1{\wt q_1} \label{CC9}, 
\end{align}

\smallskip
\item[(iv)]
For Sobolev's inequality
$  W^{-\al, \wt r_1}(\T^2)\subset  H^{s_2 - 2} (\T^2)$:
\begin{align}
\frac{2-s_2 - \al}{2}\geq  \frac1{\wt r_1} - \frac 12 .
\label{CC10a}
\end{align}

\end{itemize}

\noi
The constraints (i) - (iv) allow us to prove local well-posedness
of the system \eqref{XY3a}.

We aim to obtain the best possible range $0 < \be^2 < \bw^2$
under the constraint \eqref{al} from Proposition \ref{PROP:Ups}: % (note that $p\ge 2$):
\begin{align}
& \al\ge(p-1)\frac{\bw^2}{4\pi} .
\label{CC10b}
\end{align}

\noi
First, note that from \eqref{CC8} with \eqref{pair1}, 
$\frac{1}{r_1} - \frac{1}{\wt r_1}$ depends only on $\al$, not on $s_1$.
Then, by 
saturating
 \eqref{CC10} in the constraint (ii) above
and substituting $\frac{1}{r_1} - \frac{1}{\wt r_1} = \frac{4}{3}\al - \frac 23$, we obtain
$\al$ in terms of $p$, 
which reduces \eqref{CC10b}
to 
\[\bw^2 \leq  \frac{2p-3}{5p(p-1)} 8 \pi .\]

\noi
The right-hand side is maximized 
when  $p=\frac{3+\sqrt{3}}{2}\simeq 2.37$, 
giving 
 \[\bw^2 = \frac{32-16\sqrt{3}}{5}\pi\simeq 0.86 \pi.\]

 \noi
 This in turn  implies  $\al = (p-1)\frac{\bw^2}{4\pi}=\frac{2\sqrt{3}-2}{5}$. 
  As for the other parameters, we have freedom  to take any $s_1\in [\al,1-\al]$
   which determines the values of $q,r,q_1,r_1,\wt q,\wt r,\wt q_1,\wt r_1$. 
In the following, 
we set  $s_1=1-\al$ (which gives the best regularity for $X$).
For the sake of concreteness, we 
choose the following parameters:
\begin{align}
\bw^2 &= \frac{32-16\sqrt{3}}{5}\pi,& p &= \frac{3+\sqrt{3}}{2}, & \al &= \frac{2\sqrt{3}-2}{5},\notag \\
  s_1 &= 1-\al, & s_2&= s_1+1,  \notag \\
q &= \frac{15}{7-2\sqrt{3}}, 
&  q_1 &= \frac{15}{9-4\sqrt{3}}, 
& \wt q_1 &= 1, \label{exponent1}\\
  r &= \frac{30}{8\sqrt{3}-13}, &  r_1 &= \frac{30}{16\sqrt{3}-21},
  & \wt r_1 &= 2. \notag 
\end{align}

\noi
We point out that the constraints
\eqref{CC9} and 
\eqref{CC10a} are satisfied
with this choice of parameters.

\begin{proof}[Proof of Theorem~\ref{THM:wave}]

Let  $ 0 < T < 1$
and $B \subset \X^{s_1}_T\times \Y^{s_2}_T$ denotes the ball of radius $O(1)$
centered at the origin.
We set  
\[ K = \| (v_0, v_1)\|_{\H^s}
\qquad \text{and}\qquad  R = \|\U\|_{L^{p}([0, 1]; W^{-\al,p}_x)}\]

\noi
for $(v_0,v_1)\in\H^s(\T^2)$ for some $s> 1$
and 
 a positive distribution $\U\in L^p([0, 1]; W^{-\al, p}(\T^2))$.

\smallskip

\noi
$\bullet$ {\bf Step 1:}
Let 
$(X,Y)\in B\subset \X^{s_1}_T\times \Y^{s_2}_T$.
By the  Strichartz estimate (Lemma~\ref{LEM:Stri}) with the definitions \eqref{StriX} and \eqref{StriN} of the Strichartz  space $\X_T^{s_1}$ and the dual space $\NN_T^{s_1}$, Lemma~\ref{LEM:interpol}, and H\"older's inequality
(with $\wt r_1, \wt q_1\le 2 < p$ in view of \eqref{exponent1}), we have 
\begin{align}
\begin{split}
\|\Phi_1(X,Y)\|_{\X^{s_1}_T} &\les \big\|e^{\be z} F(\be X)e^{\be Y}\U\big\|_{\NN^{s_1}_T}\\
&\les \big\|e^{\be z}F(\be X)e^{\be Y}\U\big\|_{L^{\wt q_1}_TW^{-\al,\wt r_1}_x}\\
&\les T^\ta \big\|e^{\be z}F(\be X)e^{\be Y}\U\big\|_{L^{p}_TW^{-\al,p}_x}
\end{split}
\label{FP0}
\end{align}

\noi
for some $\ta > 0$.

As in the parabolic case, 
we would like to exploit the positivity of $\U$
and apply  Lemma~\ref{LEM:posprod} at this point.
Unlike the parabolic case, however, 
the function $X$ does not have sufficient regularity in order to apply Lemma~\ref{LEM:posprod}~(i).
Namely, we do not know if $X(t)$ is continuous (in $x$)
for almost every $t \in [0, T]$.
We instead rely
 on the hypothesis (ii) in Lemma \ref{LEM:posprod}.
 
In the following discussion, 
we only discuss spatial regularities
holding for almost every $t \in [0, T]$.
For simplicity, we suppress the time dependence.
If we have 
 \begin{align}
 e^{\be z}F(\be X)e^{\be Y}\in W^{\al, r_0}(\T^2) 
\label{CC11}
 \end{align}
 
 \noi
 for  some $r_0 < r_1$ sufficiently close to $r_1$,
 then the 
condition  \eqref{CC10}
guarantees the hypothesis (ii) in Lemma~\ref{LEM:posprod}:
\begin{align}
\frac1{r_0}+\frac1p \le \bigg(\frac1{\wt r_1}+\frac{\al}2\bigg) + \eps
< 1 + \frac{\al}{2}
\label{CC17}
\end{align}

\noi
for some small $\eps > 0$, 
since $\wt r_1 >1 $.
We now verify \eqref{CC11}.
 The fractional Leibniz rule (Lemma~\ref{LEM:prod1}\,(i))
with $\frac{1}{r_0} = \frac{1}{r_1} + \frac{1}{r_2}$ for some large but finite $r_2$
yields 
\begin{align}
\begin{split}
\big\|e^{\be z}F(\be X)e^{\be Y}\big\|_{W^{\al,r_0}_x} 
&\les \|F(\be X)\|_{W^{\al, r_1}_x}\big\|e^{\be(z+Y)}\big\|_{L^{r_2}_x}\\
&\qquad+\|F(\be X)\|_{L^{r_2}_x}\big\|e^{\be(z+Y)}
\big\|_{W^{\al,r_1}_x}.
\end{split}
\label{CC12}
\end{align}

\noi
Recall that  $F$ is Lipschitz.
Hence, 
by  the fractional chain rules (Lemma \ref{LEM:FC}\,(i)), we have
\begin{align}
\begin{split}
\|F(\be X)\|_{W^{\al,r_1}_x} 
&\sim \|F(\be X)\|_{L^{r_1}_x}
 + \big\||\nabla|^{\al}F(\be X)\big\|_{L^{r_1}_x}\\
&\les 1+\|X\|_{W^{\al,r_1}_x} < \infty, 
\end{split}
\label{CC13}
\end{align}

\noi
since 
Lemma \ref{LEM:interpol} (i)  ensures that $X  \in W^{\al,r_1}(\T^2)$.
 Similarly, 
 by  the fractional chain rule (Lemma \ref{LEM:FC}\,(ii)), we have
\begin{align}
\begin{split}
\big\|e^{\be(z+Y)}\big\|_{W^{\al, r_1}_x} 
&\sim \big\|e^{\be(z+Y)}\big\|_{L^{r_1}_x}
+\big\||\nabla|^{\al}e^{\be(z+Y)}\big\|_{L^{r_1}_x}\\
&\les  e^{C\|z + Y\|_{L^{\infty}_x}}
+\big\|e^{\be(z+Y)}\big\|_{L^{r_3}_x}
\big\||\nabla|^{\al}(z+Y)\big\|_{L^{r_1 + \eps}_x}\\
&\les e^{C\|z + Y\|_{H^{1+\eps}_x}}\big(1+\|z+Y\|_{H^1_x}\big)
< \infty
\end{split}
\label{CC14}
\end{align}

\noi
for some large but finite $r_3 $ and small $\eps > 0$,
since $z \in H^s(\T^2) $ and $Y\in H^{s_2}(\T^2)$ with $s,s_2>1$.
In the last step, we used Sobolev's inequality 
$ \frac{1-\al}{2} \geq \frac 12 - \frac 1{r_1 + \eps} $,
which is guaranteed from \eqref{exponent1}:
 \begin{align}
  \frac{\al}2<\frac1{r_1}
\label{XE1}
\end{align}
    and choosing $\eps > 0$ sufficiently small.
Putting \eqref{CC12}, \eqref{CC13}, and \eqref{CC14}, 
we see that \eqref{CC11} is satisfied
for almost every $t \in [0, T]$.

By applying Lemma~\ref{LEM:posprod}
to \eqref{FP0}, 
we have  
\begin{align}
\begin{split}
\|\Phi_1(X,Y)\|_{\X^{s_1}_T}
&\les T^\ta \big\|e^{\be z}F(\be X)e^{\be  Y}\big\|_{L^{\infty}_{T,x}}\|\U\|_{L^{p}_TW^{-\al,p}_x}\\
&\les T^\ta e^{C\|z+Y\|_{L^{\infty}_{T,x}}}\|\U\|_{L^{p}_TW^{-\al,p}_x}\\
&\les T^{\ta}e^{CK}R.
\end{split}
\label{FP1}
\end{align}

\noi

 Next, by applying  Lemma~\ref{LEM:smooth}, 
Sobolev's inequality with \eqref{CC10a}
and $p > \wt r_1$, 
and 
  proceeding as in~\eqref{FP1}, we have
\begin{align}
\begin{split}
\|\Phi_2(X,Y)\|_{\Y^{s_2}_T} &\les\big\|e^{\be z}F(\be X)e^{\be Y}\U \big\|_{L^1_TH^{s_2-2}_x}\\
&\les T^\ta \big\|e^{\be z}F(\be X)e^{\be Y}\U \big\|_{L^p_TW^{-\al,p }_x}\\
&\les T^\ta e^{C K}R.
\end{split}
\label{FP2}
\end{align}

By choosing $T = T(K, R) > 0$ sufficiently small, 
the estimates \eqref{FP1} and \eqref{FP2} 
show boundedness of $\Phi = (\Phi_1, \Phi_2)$
on the ball
$B \subset \X^{s_1}_T\times \Y^{s_2}_T$.

\smallskip

\noi
$\bullet$ {\bf Step 2:}
Next, we establish difference estimates.
Let  $(X_1,Y_1),(X_2,Y_2)
\in B \subset \X^{s_1}_T\times \Y^{s_2}_T$.
Write
\begin{align}
\begin{split}
\|\Phi(X_1, & Y_1)-\Phi(X_2,Y_2)\|_{\X^{s_1}_T\times \Y^{s_2}_T}\\
&
\le \|\Phi_1(X_1,Y_1)-\Phi_1(X_2,Y_1)\|_{\X^{s_1}_T} + \|\Phi_2(X_1,Y_1)-\Phi_2(X_2,Y_1)\|_{\Y^{s_2}_T}\\
&\hphantom{X}
+\|\Phi_1(X_2,Y_1)-\Phi_1(X_2,Y_2)\|_{\X^{s_1}_T}+\|\Phi_2(X_2,Y_1)-\Phi_2(X_2,Y_2)\|_{\Y^{s_2}_T}\\
&=: \1_1+\1_2+\II_1+\II_2.
\end{split}
\label{XE2}
\end{align}

Recall from 
\eqref{F2} and \eqref{F3} that 
\[ F(\be X_1)-F(\be X_2) = \be (X_1-X_2)G(X_1,X_2).\]

\noi
Then, by the Strichartz estimate (Lemma \ref{LEM:Stri}), 
Lemma~\ref{LEM:interpol}\,(ii), 
and  Lemma~\ref{LEM:prod1}\,(ii), we have
\begin{align*}
\1_1&\les \big\|e^{\be z}\big(F(\be X_1)-F(\be X_2)\big)e^{\be Y_1}\U\big\|_{\NN^{s_1}_T}\\
&\les \big\|e^{\be z}(X_1-X_2)G(X_1,X_2)e^{\be Y_1}\U\big\|_{L^{\wt q_1}_TW^{-\al,\wt r_1}_x}\\
& \les T^{\ta}\big\|X_1-X_2\big\|_{L^{q_1}_TW^{\al,r_1}_x}\big\|e^{\be z}G(X_1,X_2)e^{\be Y_1}\U\big\|_{L^p_TW^{-\al,p}_x}, 
\end{align*}
 provided that 
 \begin{align}
 \ta = \frac1{\wt q_1} -\frac1{q_1}-\frac1p >0\qquad\text{ and }\qquad\frac1{r_1}+\frac1p\le \frac1{\wt r_1}+\frac{\al}2,
 \label{CC16}
 \end{align}
 
 \noi
 which are precisely the constraints \eqref{CC9} and \eqref{CC10}.
 Then, applying 
 Lemma \ref{LEM:posprod} as in \eqref{FP1} along with the boundedness of $G$, 
 we obtain 
 \begin{align}
 \1_1\les T^{\ta} e^{C K} R\|X_1-X_2\|_{\X^{s_1}_T}, 
  \label{FP3}
 \end{align}

 \noi
 where we also used Lemma \ref{LEM:interpol}\,(i) to estimate the  norm of $X_1-X_2$.
 As for $\1_2$, 
  Lemma~\ref{LEM:smooth} 
  and Sobolev's inequality with \eqref{CC10a} yield
\begin{align*}
\1_2&\les \big\|e^{\be z}(X_1-X_2)G(X_1,X_2)e^{\be Y_1}\U\big\|_{L^1_TH^{s_2-2}_x}\\
&\les\big\|e^{\be z}(X_1-X_2)G(X_1,X_2)e^{\be Y_1}\U\big\|_{L^1_TW^{-\al,\wt r_1}_x}.
\end{align*}

\noi
Then, proceeding as above, we obtain
 \begin{align}
 \1_2\les T^{\ta} e^{C K} R\|X_1-X_2\|_{\X^{s_1}_T}.
  \label{FP4}
 \end{align}

As for $\II_1$, 
by Lemmas \ref{LEM:Stri} and \ref{LEM:interpol}\,(ii), 
  the fundamental theorem of calculus (as in~\eqref{F2} and \eqref{F3}),
 Lemma \ref{LEM:prod1}\,(ii) with \eqref{CC16}, 
 and then proceeding as in \eqref{FP1} with Lemma~\ref{LEM:posprod},  
  we have 
\begin{align}\label{FP5}
%\begin{split}
\II_1
&\les \big\|e^{\be z}F(\be X_2)\big(e^{\be Y_1}-e^{\be Y_2}\big)\U \big\|_{L^{\wt q_1}_TW^{-\al,\wt r_1}_x}
\notag \\
&\les \bigg\|e^{\be z}F(\be X_2)
(Y_1 - Y_2)\int_0^1\exp\big(\tau \be Y_1+(1-\tau) \be Y_2\big)d\tau \cdot \U \bigg\|_{L^{\wt q_1}_TW^{-\al,\wt r_1}_x}\\
&\les T\ta \big\|Y_1-Y_2\big\|_{L^{q_1}_TW^{\al, r_1}_x}\bigg\|e^{\be z} F(\be X_2)
\int_0^1\exp\big(\tau \be Y_1+(1-\tau) \be Y_2\big)d\tau\cdot \U\bigg\|_{L^{p}_TW^{-\al,p}_x}\notag \\
&\les T^{\ta }e^{C  K}R
 \|Y_1-Y_2\|_{\Y^{s_2}_T}.\notag 
%\end{split}
\end{align}

\noi
In the last step, we use the embedding $H^{s_2}(\T^2) \subset W^{\al, r_1}(\T^2)$,
which is guaranteed by~\eqref{XE1} and $s_2 > 1$.
Similarly, 
by  applying  Lemma~\ref{LEM:smooth}
and Sobolev's inequality with \eqref{CC10a}
  and proceeding as in~\eqref{FP5}, 
we have
\begin{align}\label{FP6}
\begin{split}
\II_2 & \les \big\|e^{\be z}F(\be X_2)\big(e^{\be Y_1}-e^{\be Y_2}\big)\U \big\|_{L^{1}_TW^{-\al,\wt r_1}_x}
\\
&\les T^{\ta }e^{C  K}R
 \|Y_1-Y_2\|_{\Y^{s_2}_T}.
\end{split}
\end{align}

\noi
From Step 1, \eqref{XE2}, \eqref{FP3}, \eqref{FP4}, \eqref{FP5}, and \eqref{FP6}, 
we conclude that $\Phi = (\Phi_1, \Phi_2)$ is a contraction
on the ball $ B\subset \X^{s_1}_T\times \Y^{s_2}_T$,
thus establishing local well-posedness of \eqref{XY3a}.

\smallskip

\noi
$\bullet$ {\bf Step 3:}
Continuous
dependence of the solution $(X, Y)$
on initial data $(v_0, v_1)$ easily follows from the argument in Step 2.
Hence, it remains to prove
 continuous
dependence of the solution $(X, Y)$ on the ``noise'' term $\U$.

Let $(X_j, Y_j)\in B \subset \X^{s_1}_T\times \Y^{s_2}_T$
be solutions to \eqref{XY3a}
with a noise term $\U_j$, $j = 1, 2$.
In estimating the difference, 
we can apply the argument in Step 2
to handle all the terms except
for the following two terms:
\begin{align*}
&\bigg\|\int_0^te^{-\frac{(t-t')}{2}} S(t-t')\big(e^{\be z}F(\be X_1)e^{\be Y_1}(\U_1-\U_2)\big)(t')dt'\bigg\|_{\X^{s_1}_T}\\
&\hphantom{X} +  
\bigg\|\int_0^t\big(\D(t-t')-e^{-\frac{(t-t')}{2}} S(t-t')\big)\big(e^{\be z}F(\be X_1)e^{\be Y_1}(\U_1-\U_2)
\big)(t')dt'\bigg\|_{\Y^{s_2}_T}\\
& =: \III_1+\III_2.
\end{align*}

\noi
The main point is that the difference $\U_1 - \U_2$
does not enjoy positivity and hence we can not apply Lemma \ref{LEM:posprod}.

Let  $r_0 < r_1$ sufficiently close to $r_1$,
satisfying \eqref{CC17}:
\begin{align}
\frac1{r_0}+\frac{1- \eps p}p \le\frac1{\wt r_1}+\frac{\al}2.
\label{CC17x}
\end{align}

\noi
By  the Strichartz estimate (Lemma~\ref{LEM:Stri}),  Lemma \ref{LEM:interpol} (ii), 
and Lemma \ref{LEM:prod1}\,(ii) with \eqref{CC17x}, 
we have
\begin{align*}
\III_1& \les \big\|e^{\be z}F(\be X_1)e^{\be Y_1}(\U_1-\U_2)\big\|_{L^{\wt q_1}_TW^{-\al,\wt r_1}_x}\\
&\les T^{\ta}\big\|e^{\be z} F(\be X_1)e^{\be Y_1}
\big\|_{L^{ q_1}_TW^{\al, r_0}_x}\|\U_1-\U_2\|_{L^{p}_TW^{ -\al,  \frac{p}{1- \eps p}}_x}.
\end{align*}

\noi
Then, 
applying 
\eqref{CC12}, \eqref{CC13}, and \eqref{CC14}
along with  H\"older's inequality in time
and Sobolev's inequality, 
we obtain 
\begin{align*}
\III_1
&\les
 T^{\ta}e^{CK} (1 + K) \|\U_1-\U_2\|_{L^{p}_TW^{ -\al + 2\eps , p}_x}.
%\label{FP7}
\end{align*}

\noi
Thanks to 
Lemma \ref{LEM:smooth}
and the embedding
$ L^{\wt q_1}([0, T]; W^{-\al, \wt r_1}(\T^2))
\subset L^{1}([0, T]; H^{s_2 - 2} (\T^2)) $ (see \eqref{CC10a}), 
the second term $\III_2$
can be handled in an analogous manner.

Let $0 < \be^2 < \bw^2$.
Then, the pair $(\al, p)$ in \eqref{exponent1}
satisfies the condition \eqref{al}.
Then, by taking $\eps > 0$ sufficiently small, 
we see that the pair $(\al-2\eps, p)$
also satisfies the condition~\eqref{al}.
Hence, as $\U_2$ tends to $\U_1$ in 
 $L^p([0, 1]; W^{-\al + 2\eps, p}(\T^2))$, 
 we conclude that $\III_1 + \III_2 \to 0$, 
 establishing the continuity of the solution map
 $(v_0, \U) \mapsto (X, Y)$.
This completes the proof of Theorem~\ref{THM:wave}.
\end{proof}

\subsection{Almost sure global well-posedness and invariance of the Gibbs measure}
In this  subsection, we briefly discuss a proof of Theorem \ref{THM:4}. 
As mentioned in Section \ref{SEC:1}, 
 the well-posedness result of Theorem \ref{THM:wave} proved  in the previous subsection 
 is only local in time
 and hence we need to apply Bourgain's invariant measure argument \cite{BO94, BO96} 
 to extend the dynamics  globally in time almost surely with respect 
 to the Gibbs measure $\rw$ and then show invariance of the Gibbs measure $\rw$.

Given $N\in\N$,  we consider the following truncated SdNLW:
\begin{align}\label{NW6}
\begin{cases}
\dt^2 u_N +\dt u_N +(1-\Dl)u_N + \ld\be C_N\Q_Ne^{\be \Q_Nu_N}=\sqrt{2}\xi\\
(u_N,\dt u_N)|_{t=0}=(\Q_Nw_0,\Q_Nw_1),
\end{cases}
\end{align}

\noi
where $\Q_N$ is as in \eqref{Q} and 
$(w_0,w_1)$ is as in \eqref{IV2}.
Namely, $\Law(w_0, w_1) = 
\mu_1\otimes\mu_0$.\footnote
{In view of the equivalence of  $\mu_1\otimes\mu_0$  and the Gibbs measure $\rw$ in \eqref{Gibbs2},
it suffices to study \eqref{NW6} with the initial data distributed by $\mu_1\otimes \mu_0$.}

By writing  $u_N =  X_N+Y_N + \Psi $, 
where  $\Psi = \Psi^\text{wave}$ is as in \eqref{W6}, we have 
\begin{align*}
\begin{split}
X_N(t)&=-\ld\be\int_0^te^{-\frac{(t-t')}{2}}S(t-t')\Q_N
\big(e^{\be \Q_N X_N}e^{\be \Q_N Y_N}\U_N \big)(t')dt', \\
Y_N(t) &= -\ld\be\int_0^t\big(\D(t-t')-e^{-\frac{(t-t')}{2}}S(t-t')\big)\Q_N
\big(e^{\be \Q_N X_N}e^{\be \Q_N Y_N}\U_N \big)(t')dt'.
\end{split}
%\label{XY4}
\end{align*}

\noi
By the positivity of the smoothing operator $\Q_N$, 
$X_N$ enjoys the sign-definite structure:
\begin{align*}
\be \Q_NX_N = -\ld\be^2\int_0^te^{-\frac{(t-t')}{2}}S(t-t')\Q_N^2
\big(e^{\be \Q_N X_N}e^{\be \Q_N Y_N}\U_N \big)(t')dt' \le 0,
\end{align*}

\noi
thanks to $\ld > 0$ and the positivity of the linear wave propagator $S(t)$
(Lemma \ref{LEM:waveker}).
Hence,  it is enough to consider
\begin{align}
X_N(t)&=-\ld\be\int_0^te^{-\frac{(t-t')}{2}}S(t-t')\Q_N\big(F(\be \Q_N X_N)e^{\be \Q_N Y_N}\U_N \big)(t')dt', 
\label{XY5}
\\
Y_N(t) &= -\ld\be\int_0^t\big(\D(t-t')-e^{-\frac{(t-t')}{2}}S(t-t')\big)\Q_N\big(F(\be \Q_N X_N)e^{\be \Q_N Y_N}\U_N \big)(t')dt', 
\notag
\end{align}

\noi
where  $F$ is as in \eqref{F1}.

In view of the uniform (in $N$) boundedness of $\Q_N$ on $L^p(\T^2)$, $1 \leq p \leq \infty$, 
 we can argue as in Subsection~\ref{SUBSEC:wave1}
to prove local well-posedness of 
the system \eqref{XY5} in a uniform manner for any $N \in \N$.
In order to prove  convergence of the solution $\big((X_N,\dt X_N),(Y_N,\dt Y_N)\big)$ to \eqref{XY5} towards the solution $\big((X,\dt X),(Y,\dt Y)\big)$ of the untruncated dynamics \eqref{XY3a}, 
we can repeat the argument in Step 3
of the previous subsection 
to estimate the difference between 
$\big((X_N,\dt X_N),(Y_N,\dt Y_N)\big)$ and $\big((X,\dt X),(Y,\dt Y)\big)$. 
As in Subsection \ref{SUBSEC:Gibbs1}, 
we need to estimate the terms with 
 $\Q_N - \Id$:
\begin{align*}
&  \bigg\|\int_0^te^{-\frac{(t-t')}{2}}S(t-t')\big(\Q_N-\Id\big)\big(F(\be \Q_N X_N)e^{\be \Q_N Y_N}\U_N \big)(t')dt'\bigg\|_{\X^{s_1}_T}\\
&\hphantom{X} 
+\bigg\|\int_0^t\big(\D(t-t')-e^{-\frac{(t-t')}{2}}S(t-t')\big)\big(\Q_N-\Id\big)\big(F(\be \Q_N X_N)e^{\be \Q_N Y_N}\U_N \big)(t')dt'\bigg\|_{\Y^{s_2}_T}\\
& =:\IV_1 + \IV_2.
\end{align*}

The property  \eqref{QN-1} of $\Q_N$ allows
us to gain a negative power of $N$ at a slight expense of regularity.
By a slight modification of  the argument from the previous subsection
(see \eqref{FP1}), 
we have
\begin{align}
\begin{split}
\IV_1 &\les \Big\|\big(\Q_N-\Id\big)
\big(F(\be \Q_N X_N)e^{\be \Q_N Y_N}\U_N \big)\Big\|_{L^{\wt q_1}_TW^{-\al,\wt r_1}_x}\\
&\les N^{-\eps}\big\|F(\be \Q_N X_N)e^{\be \Q_N Y_N}\U_N \big\|_{L^{\wt q_1}_TW^{-\al+\eps,\wt r_1}_x}\\
 &\les T^{\ta}N^{-\eps}\exp\big(C\|Y_N\|_{L^{\infty}_TH^{s_2}_x}\big)\big\|\U_N\big\|_{L^p_TW^{-\al+\eps,p}_x}.
 \end{split}
\label{FP9}
\end{align}

\noi
Note that by choosing $\eps > 0$ sufficiently small, 
the range  $0 < \be^2<\bw^2$ does not change
even when we replace $- \al$ in \eqref{FP1} by $-\al + \eps$ in \eqref{FP9}.
Similarly, we have
\begin{align}
\IV_2&\les T^{\ta}N^{-\eps}\exp\big(C\|Y_N\|_{L^{\infty}_TH^{s_2}_x}\big)\big\|\U\big\|_{L^p_TW^{-\al+\eps,p}_x}.
\label{FP10}
\end{align}

\noi
The estimates \eqref{FP9} and \eqref{FP10}
combined with the argument in the previous subsection
allows us 
to prove the desired convergence
of 
$\big((X_N,\dt X_N),(Y_N,\dt Y_N)\big)$ to $\big((X,\dt X),(Y,\dt Y)\big)$. 
The rest of the argument follows
from  applying  Bourgain's invariant measure argument \cite{BO94, BO96}.
Since it is standard,  we omit details.
See,  for example, \cite{ORTz, STz, GKOT, OOTol, Bring2, OOTol2} for details.

%%%%%%%%%%%%%%%%%%%%%%%%%
%%%%%%%%%%%%%%%%%%%%%%%%%
%%%%%%%%%%%%%%%%%%%%%%%%%
%%%%%%%%%%%%%%%%%%%%%%%%%
%%%%%%%%%%%%%%%%%%%%%%%%%
%%%%%%%%%%%%%%%%%%%%%%%%%

\appendix

\section{On local well-posedness of SNLH without using the positivity}
\label{SEC:A}
In this appendix, 
we revisit the fixed point problem \eqref{D2a} for SNLH:
\begin{align}
v = \Phi_{v_0, \U}(v),
\label{A1}
\end{align}

\noi
 where the map $\Phi = \Phi_{v_0, \U}$ is defined in \eqref{D2}.
In Sections \ref{SEC:fpa} and \ref{SEC:WP}, we studied this problem
by exploiting the positivity of $\U$ and furthermore the sign-definite structure of the equation when $\ld > 0$.
In the following, we study \eqref{D2a}
for general $\ld \in \R\setminus \{0\}$
and present a simple contraction argument
{\it without} using any positivity of $\U$
for the range $0<\be^2<\frac{4}{3}\pi \simeq  1.33\pi$.
This simple argument provides Lipschitz dependence
of a solution on initial data $v_0$ and noise $\U$.

Let $0 < \al < 1$ and  $p\ge 2$ such that
\begin{align}
\label{cons3}
\displaystyle p'\bigg(\frac{\al}2+\frac1p\bigg) < 1
\qquad \text{and}\qquad 
0<\al\le\frac2p.
\end{align}

\begin{theorem}\label{THM:A}
Let $\al, p$ be as above.
Then,  given any $v_0\in H^{1+\eps}(\T^2)$
and $\U\in L^p([0,1];W^{-\al,p}(\T^2))$ for some small $\eps > 0$, 
there exist $T=T\big(\|v_0\|_{L^{\infty}}, \| \U \|_{L^p([0, 1]; W^{-\al,p}_x)}\big)>0$ 
and  a unique solution 
 $v\in C([0,T];W^{\al + \eps ,\frac 2\al}(\T^2))$
to \eqref{A1},
depending continuously on the initial data $v_0$
and the noise $\U$.

\end{theorem}

In view of Proposition \ref{PROP:Ups}
on the regularity of the Gaussian multiplicative chaos $\U_N$, 
we see that Theorem \ref{THM:A}
provides local well-posedness
of SNLH~\eqref{NH5}
for the range:
\begin{align*}
0 < \be^2 & < \frac{4\pi \al}{p-1} <  8\pi \frac{\min\big(\frac 1p,1-\frac2p\big)}{p-1}, 
\end{align*}

\noi
where we used both of the inequalities in \eqref{cons3}.
 Hence,  optimizing 
 \[\min\Bigg(\max_{p\ge 3}\frac1{p(p-1)},\max_{2\le p\le 3}\frac{p-2}{p(p-1)}\Bigg),\] 
 we find that the maximum is attained at  $p=3$, 
which  gives the range $0 < \be^2 < \frac{4}{3}\pi $.
With $p = 3$, we can take
$\al=\frac23 - \eps $
 for some small $\eps > 0$
 such that \eqref{cons3} is satisfied.
We point out that our argument requires 
the initial data $v_0$ to belong to a smaller space 
$H^{1+\eps}(\T^2)\subset L^\infty(\T^2)$.

\begin{proof}[Proof of Theorem \ref{THM:A}]
Fix small $\eps > 0$
such that 
\begin{align}
p'\bigg(\frac{\al+\eps}2+\frac1p\bigg) < 1.
\label{cons4}
\end{align}

\noi
Given   $v_0\in H^{1+\eps}(\T^2)$
and  $\U\in L^p([0,1];W^{-\al,p}(\T^2))$, 
we consider the map $\Phi = \Phi_{v_0, \U}$ defined in \eqref{D2}
and set 
 $z=P(t)v_0$ as in \eqref{linear}.
Let $B\subset C([0,T];W^{\al+\eps,\frac2{\al}}(\T^2))$ be the ball of radius $O(1)$
centered at the origin
and set 
\[ K = \|v_0\|_{H^{1+\eps}} \qquad \text{and}\qquad 
R = \|\U\|_{L^p([0,1];W^{-\al,p})}.\]

Let $0 < T \leq 1$.
By the  Schauder estimate (Lemma \ref{LEM:heatker})
with  $\frac2\al \ge p$ (as guaranteed in \eqref{cons3}), 
Lemma~\ref{LEM:prod1}\,(ii) with $\frac1p + \frac1{2/\al} = \frac1p + \frac\al2$,
and H\"older's inequality in time with \eqref{cons4}, 
we have 
\begin{align}
\begin{split}
\|\Phi (v)\|_{C_TW^{\al+\eps,\frac2\al}_x} 
& \les \bigg\| \int_0^t (t-t')^{-\frac{2\al + \eps}2 - (\frac{1}p-\frac{\al}2)}
\big\|e^{\be z } e^{\be v}\U(t') \big\|_{W^{-\al,p}_x}  dt' \bigg\|_{L^{\infty}_T}\\
& \les \big\| e^{\be (z+ v)} \big\|_{L_T^\infty W^{\al,\frac2\al}_x}
\int_0^T (t-t')^{-\frac{\al + \eps}2 - \frac{1}p}
\|\U (t')\|_{W^{-\al,p}_x }   dt'\\
 &\les T^\ta \big\| e^{\be (z+ v)} \big\|_{L_T^\infty W^{\al,\frac2\al}_x}   \|\U \|_{L_T^p W^{-\al,p}_x }
\end{split}
\label{A3}
\end{align}

\noi
for some $\ta > 0$.
By the fractional chain rule (Lemma \ref{LEM:FC}\,(ii))
and the Sobolev embeddings:
\begin{align}
\begin{split}
H^{1+\eps}(\T^2) & \subset 
W^{\al+\eps,\frac2\al}(\T^2) \cap L^{\infty}(\T^2),\\
W^{\al+\eps,\frac2\al}(\T^2) & \subset W^{\al,\frac2{\al-\eps}}(\T^2)\cap L^{\infty}(\T^2),
\end{split}
\label{A3a}
\end{align}
 we have
\begin{align}
\begin{split}
\big\| e^{\be (z+ v)} \big\|_{L_T^\infty W^{\al,\frac2\al}_x}
&\sim \big\| e^{\be (z+ v)} \big\|_{L_T^\infty L^{\frac2\al}_x} + \big\| |\nabla|^{\al}e^{\be (z+ v)} \big\|_{L_T^\infty L^{\frac2\al}_x}\\
&\les e^{C\|z+v\|_{L^{\infty}_{T,x}}}+\big\|e^{\be (z+v)}\big\|_{L^{\infty}_TL^{\frac2\eps}_x}\big\||\nabla|^{\al}(z+v)\big\|_{L^{\infty}_TL^{\frac2{\al-\eps}}_x}\\
&\les \exp\Big(C\big(\|v_0 \|_{H^{1+\eps}}+\|v\|_{L^{\infty}_TW^{\al+\eps,\frac2\al}_x}\big)\Big)\\
& \hphantom{X}
\times
\Big(1+\|v_0\|_{H^{1+\eps}} +\| v\|_{L^{\infty}_TW^{\al+\eps,\frac2\al}_x}\Big).
\end{split}
\label{A4}
\end{align}

\noi
Hence, from \eqref{A3} and \eqref{A4}, we have
\begin{align}
\big\|\Phi(v)\big\|_{C_TW^{\al+\eps,\frac2\al}_x} 
&\les T^\ta  e^{CK} (1+K) R
\label{Z1}
\end{align}

\noi
for any $v \in B$.

Proceeding as in \eqref{A3}, we have
\begin{align}
\|\Phi (v_1) -\Phi (v_2) \|_{C_T W^{\al + \eps, \frac2\al}_x} 
\les T^\ta   \big\| e^{\be z} (e^{\be v_1} -e^{\be v_2} ) \big\|_{L^\infty_T W^{\al, \frac2\al}_x}\| \U\|_{L^p_T W^{-\al , p}_x}.
\label{A5}
\end{align}

\noi
By \eqref{Q3}, 
 the fractional Leibniz rule (Lemma \ref{LEM:prod1}\,(i)),  
followed by  the fractional chain rule as in \eqref{A4}, 
we have
\begin{align}
\begin{split}
 \big\| e^{\be z}  & (e^{\be v_1} -e^{\be v_2} ) \big\|_{L^{\infty}_TW^{\al, \frac2\al}_x}\\
& \les \big\|e^{\be z}\big\|_{L^{\infty}_TW^{\al,\frac{2}{\al-\eps/2}}_x}\|v_1-v_2\|_{L^{\infty}_{T,x}}e^{C(\|v_1\|_{L^{\infty}_{T,x}}+\|v_2\|_{L^{\infty}_{T,x}})}\\
 &\hphantom{X}
 +e^{C\|z\|_{L^{\infty}_{T,x}}}\Bigg\{\|v_1-v_2\|_{L^{\infty}_TW^{\al,\frac{2}{\al-\eps}}_x}e^{C(\|v_1\|_{L^{\infty}_{T,x}}+\|v_2\|_{L^{\infty}_{T,x}})} \\
&\hphantom{X}
+ \|v_1-v_2\|_{L^{\infty}_{T,x}}\bigg\|\int_0^1\exp\big(\tau\be v_1+(1-\tau)\be v_2\big)d\tau\bigg\|_{L^{\infty}_TW^{\al,\frac{2}{\al-\eps/2}}_x}\Bigg\}\\
 &
 \les e^{CK}(1+K)\Big( \|v_1-v_2\|_{L^{\infty}_{T,x}}
+  \|v_1-v_2\|_{L^{\infty}_TW^{\al,\frac{2}{\al-\eps}}_x}\Big) 
\end{split}
\label{A6}
\end{align}

\noi
 for any $v_1,v_2\in B$.
\noi
Hence, from \eqref{A5} and \eqref{A6} with \eqref{A3a},
 we have
\begin{align}
\|\Phi (v_1) -\Phi (v_2) \|_{C_T W^{\al + \eps, \frac2\al}_x}
 \les T^\ta e^{CK}(1+K)R
  \|v_1-v_2\|_{L^{\infty}_T
W^{\al + \eps, \frac2\al}_x}
\label{Z2}
\end{align}

\noi
 for any $v_1,v_2\in B$.

From  \eqref{Z1} and \eqref{Z2}, 
a contraction argument yields a solution map:
\[(v_0, \U)\in H^{1+\eps}(\T^2) \times L^p([0,1];W^{-\al,p}(\T^2))
 \longmapsto v\in C([0,T];W^{\al+\eps,\frac2\al}(\T^2))\] 
 
 \noi
 for some $T=T\big(\|v_0\|_{H^{1+\eps}}, \|\U\|_{L^p([0,1];W^{-\al,p}_x)}\big)\in (0,1]$, 
  where $v$ is the unique fixed point of $\Phi_{v_0, \U}$ in 
  the ball $B\subset 
   C([0,T];W^{\al+\eps,\frac2\al}(\T^2))$. 
 As for the Lipschitz dependence of the solution map
 on $\U$, if we take $\U_1,\U_2\in L^p([0,1];W^{-\al,p}(\T^2))$, then in estimating the difference 
 $\Phi_{v_0, \U_1}(v_1)-\Phi_{v_0, \U_2}(v_2)$ for $v_1,v_2\in B \subset C([0,T];W^{\al+\eps,\frac2\al}(\T^2))$,
there is one additional term of the form:
\[ \int_0^t  P(t-t')\big(e^{\be z}  e^{\be v_1}(\U_1-\U_2)\big)(t')dt'.\]

\noi
By proceeding as in \eqref{A3} and \eqref{A4}, we can bound this additional term as 
\begin{align*}
\bigg\|\int_0^t  P(t-t')\big(e^{\be z} & e^{\be v_1}(\U_1-\U_2)\big)(t')dt'\bigg\|_{C_TW^{\al+\eps,\frac2\al}_x}
\\
& \les T^\ta e^{CK} (1 + K)\|\U_1-\U_2\|_{L^p_TW^{-\al,p}_x}.
\end{align*}

\noi
This completes the proof of Theorem \ref{THM:A}.
\end{proof}

\section{Moment bounds for the Gaussian multiplicative chaos}\label{SEC:B}
In this last section, we give a proof of Lemma~\ref{LEM:GMCmom} on the uniform boundedness of the moments of the random measure $\M_N(t)$ in \eqref{MN}. We mainly follow the arguments in \cite{RoV,BM}.

First of all, in view of the  positivity of $\U_N(t)$, 
it suffices to prove Lemma~\ref{LEM:GMCmom} with $A=\T^2$. 
Moreover, the bound for $p=1$ being a consequence of Proposition~\ref{PROP:var}\,(i), 
we may assume  $p>1$. 
We start by fixing some large number $K\gg 1$, independent of $N\in\N$, and we partition $\T^2\simeq [-\pi,\pi)^2$ into cubes $C_{k,\ell}=x_{k,\ell}^K+[-\frac{\pi}K,\frac{\pi}K)^2$, $k,\ell=1,...,K$ of 
side length $2\pi K^{-1}$ centered at $x_{k,\ell}^K =\big(-\pi+\frac{2\pi}{K}(k-1),-\pi+\frac{2\pi}K(\ell-1)\big)\in\T^2$. We then group these into four families of cubes: 
\begin{align*}
\M_N(t,\T^2)&= \sum_{\substack{k,\ell=1\\k\text{ even, } \ell\text{ even}}}^K
\int_{C_{k,\ell}}\U_N(t, x) 
%e^{\be\psi_N(t,x)-\frac{\be^2}2\s_N}
dx+\sum_{\substack{k,\ell=1\\k\text{ even, } \ell\text{ odd}}}^K\int_{C_{k,\ell}}
\U_N(t, x)dx\\
&\quad+\sum_{\substack{k,\ell=1\\k\text{ odd, } \ell\text{ even}}}^K\int_{C_{k,\ell}}
\U_N(t, x)dx+\sum_{\substack{k,\ell=1\\k\text{ odd, } \ell\text{ odd}}}^K\int_{C_{k,\ell}}
\U_N(t, x)
dx\\
&=: \M_N^{(1)}(t)+\M_N^{(2)}(t)+\M_N^{(3)}(t)+\M_N^{(4)}(t).
\end{align*}

\noi
It follows from the (spatial) translation invariance of the law of $\Psi_N(t, \cdot)$
that  
  $\M_N^{(j)}(t)$, $j = 1, \dots, 4$,  have the same law. 
  Hence, by Minkowski's inequality,  we have
\begin{align*}
\E\Big[\M_N(t,\T^2)^p\Big] \le C_p \E\Big[\M_N^{(1)}(t, \T)^p\Big].
\end{align*}

\noi
In order to estimate the last expectation, we 
proceed as in Step 1 of the proof of Proposition~\ref{PROP:var2}.
Namely, by a change of variables
% Kahane's convexity inequality (Lemma~\ref{LEM:Kahane}), 
 and a Riemann sum approximation, 
%  to replace the $\Psi_N(t)$ on each cube $C_{k,\ell}$ in the definition of $\M_N^{(1)}(t)$ by independent copies $\psi_{N,k,\ell}(0)$. Indeed, we first compute
we have
\begin{align*}
\begin{split}
\E\Big[\M_N^{(1)}(t, \T^2)^p\Big]
&=\E\bigg[\Big(\sum_{\substack{k,\ell=1\\k\text{ even, } \ell\text{ even}}}^KK^{-2}
\int_{\T^2}\U_N(t,x_{k,\ell}^K+K^{-1}y)dy\Big)^p\bigg]\\
&=\lim_{J\to\infty}\E\bigg[\Big(\sum_{i,j=1}^{J}\frac{4\pi^2}{J^{2}}\sum_{\substack{k,\ell=1\\k\text{ even, } \ell\text{ even}}}^KK^{-2}e^{\be\Psi_N(t,x_{k,\ell}^K+K^{-1}x_{i,j}^J)-\frac{\be^2}2\s_N}\Big)^p\bigg].
\end{split}
%\label{BB4a}
\end{align*}

\noi
Using Lemma~\ref{LEM:covar}, we can bound the covariance function by
\begin{align}
\begin{split}
\E & \Big[\Psi_N(t,x_{k_1,\ell_1}^K+K^{-1}x_{i_1,j_1}^J)\Psi_N(t,x_{k_2,\ell_2}^K+K^{-1}x_{i_2,j_2}^J)\Big]\\
&= \G_N\big(t,x_{k_1,\ell_1}^K-x_{k_2,\ell_2}^K+K^{-1}(x_{i_1,j_1}^J-x_{i_2,j_2}^J)\big)\\
&\le -\frac1{2\pi}\log\Big(\big|x_{k_1,\ell_1}^K-x_{k_2,\ell_2}^K+K^{-1}(x_{i_1,j_1}^J-x_{i_2,j_2}^J)\big|+N^{-1}\Big)+C
\end{split}
\label{BB5}
\end{align}
for some constant $C>0$ independent of $J$, $K$, and $N$. 
When  $(k_1,\ell_1)=(k_2,\ell_2)$, we thus have the bound
\begin{align}
\begin{split}
\E & \Big[\Psi_N(t,x_{k_1,\ell_1}^K+K^{-1}x_{i_1,j_1}^J)\Psi_N(t,x_{k_2,\ell_2}^K+K^{-1}x_{i_2,j_2}^J)\Big]\\
&\le -\frac1{2\pi}\log\big(|x_{i_1,j_1}^J-x_{i_2,j_2}^J|+(K^{-1}N)^{-1}\big) +\frac1{2\pi}\log K+C\\
&\le -\frac1{2\pi}\log\big(|x_{i_1,j_1}^J-x_{i_2,j_2}^J|+N^{-1}\big) +\frac1{2\pi}\log K+C.
\end{split}
\label{BB6}
\end{align}

\noi
See also \eqref{BB4}.
In the case $(k_1,\ell_1)\neq(k_2,\ell_2)$, 
we first note that  $|x_{k_1,\ell_1}^K-x_{k_2,\ell_2}^K|\ge 2\cdot\frac{2\pi}K$ since $k_1,k_2,\ell_1,\ell_2$ are all even.
Then, with the trivial bound $|x_{i_1,j_1}^J-x_{i_2,j_2}^J|\le \sqrt{2}\cdot 2\pi$, % since $x_{i,j}\in\T^2$, 
we have 
\begin{align}
\big|x_{k_1,\ell_1}^K-x_{k_2,\ell_2}^K+K^{-1}(x_{i_1,j_1}^J-x_{i_2,j_2}^J)\big|\ge (2-\sqrt{2})\frac{2\pi}K.
\label{BB7}
\end{align}

\noi
Thus, from \eqref{BB5} and \eqref{BB7}, we have
\begin{align}
\E \Big[\Psi_N(t,x_{k_1,\ell_1}^K+K^{-1}x_{i_1,j_1}^J)\Psi_N(t,x_{k_2,\ell_2}^K+K^{-1}x_{i_2,j_2}^J)\Big]
\le \frac1{2\pi}\log K + C.
\label{BB8}
\end{align}

\noi
Hence, from \eqref{BB6} and \eqref{BB8}, we obtain 
\begin{align*}
&\E\Big[\Psi_N(t,x_{k_1,\ell_1}^K+K^{-1}x_{i_1,j_1}^J)\Psi_N(t,x_{k_2,\ell_2}^K+K^{-1}x_{i_2,j_2}^J)\Big]\\
&\qquad \le \E\Big[\big( \psi_{N,k_1,\ell_1}(t,x_{i_1,j_1}^J)+h_K\big)
\big( \psi_{N,k_2,\ell_2}(t,x_{i_2,j_2}^J)+h_K\big)\Big], 
\end{align*}

\noi
where  $\psi_{N,k,\ell}$ are some independent\footnote{In particular, we have 
$ \E[ \psi_{N,k_1,\ell_1}(t,x_{i_1,j_1}^J)
 \psi_{N,k_2,\ell_2}(t,x_{i_2,j_2}^J)] = 0$ when 
 $(k_1,\ell_1)\neq(k_2,\ell_2)$.} copies of $\Psi_N$
and $h_K$ is a mean-zero Gaussian random variable with variance $\frac1{2\pi}\log K+C$
 independent from  $\psi_{N,k,\ell}$.

By applying  Kahane's convexity inequality (Lemma~\ref{LEM:Kahane})
and using the independence of $h_K$ from 
$\psi_{N,k,\ell}$ with $\E [h_K^2] = \frac 1{2\pi} \log K + C$,
we have 
\begin{align}\label{moment5}
\begin{split}
\E& \Big[\M_N^{(1)}(t, \T^2)^p\Big]\\
&\le \lim_{J\to\infty}\E\bigg[\Big(\sum_{i,j=1}^{J}4\pi^2 J^{-2}\sum_{\substack{k,\ell=1\\k\text{ even, } \ell\text{ even}}}^KK^{-2}e^{\be(\psi_{N,k,\ell}(t,x_{i,j}^J)+h_K)-\frac{\be^2}2(\s_{N}+\E [h_K^2])}\Big)^p\bigg]\\
&=\E\bigg[\Big(\sum_{\substack{k,\ell=1\\k\text{ even, } \ell\text{ even}}}^KK^{-2}e^{\be h_K-\frac{\be^2}2
\E [h_K^2]}\int_{\T^2}e^{\be\psi_{N,k,\ell}(t,y)-\frac{\be^2}2\s_{N}}dy\Big)^p\bigg]\\
&\le CK^{(p^2-p)\frac{\be^2}{4\pi}}\E\bigg[\Big(\sum_{\substack{k,\ell=1\\k\text{ even, } \ell\text{ even}}}^KK^{-2}\int_{\T^2}e^{\be\psi_{N,k,\ell}(t,y)-\frac{\be^2}2\s_{N}}dy\Big)^p\bigg]
\end{split}
\end{align}

\noi
for some constant $C>0$ independent of $K$ and $N$.

It then remains to bound the expectation in \eqref{moment5}. Let $m\ge 2$ be an integer such that $m-1<p\le m$. 
Then, by the embedding $\l^\frac{p}{m} \subset \l^1$, we have
\begin{align}
\begin{split}
&\E\bigg[\Big(\sum_{\substack{k,\ell=1\\k\text{ even, } \ell\text{ even}}}^KK^{-2}\int_{\T^2}e^{\be\psi_{N,k,\ell}(t,y)-\frac{\be^2}2\s_{N}}dy\Big)^p\bigg]\\ 
&\quad\le \E\bigg[\Big\{\sum_{\substack{k,\ell=1\\k\text{ even, } \ell\text{ even}}}^KK^{-2\frac{p}m}\Big(\int_{\T^2}e^{\be\psi_{N,k,\ell}(t,y)-\frac{\be^2}2\s_{N}}dy\Big)^{\frac{p}m}\Big\}^m\bigg]\\
& \quad =: \E[A_K].
\end{split}
\label{moment5a}
\end{align}

\noi
We divide $A_K$ into two pieces:
\begin{align}\label{moment6}
\begin{split}
A_K & = \sum_{\substack{k,\ell=1\\k\text{ even, } \ell\text{ even}}}^KK^{-2p}\E\bigg[\Big(\int_{\T^2}e^{\be\psi_{N,k,\ell}(t,y)-\frac{\be^2}2\s_{N}}dy\Big)^{p}\bigg]\\
&\quad+\sum_{(\pmb{k}, \pmb{\ell})\in \Lambda_m}K^{-2p}
\E\bigg[\prod_{j=1}^m\Big(\int_{\T^2}e^{\be\psi_{N,k_j,\ell_j}(t,y)-\frac{\be^2}2\s_{N}}dy\Big)^{\frac{p}m}\bigg]
\\
& =: A_K^{(1)} + A_K^{(2)}, 
\end{split}
\end{align}

\noi
where the index set $\Ld_m$ is given by 
\begin{align*}
\Lambda_m= \big\{
(\pmb{k}, \pmb{\l}) =  (k_1,...,k_m,\ell_1,...,\ell_m)\in \{1,...,K\}^{2m}:
k_j,\ell_j\text{ even}, (k_j,\ell_j)\text{ not all equal}\big\}.
\end{align*}

Since  $\psi_{N,k,\ell}$ are identically distributed, we  can bound the diagonal term by
\begin{align}\label{moment7}
\E\big[A_K^{(1)}\big]
&\le K^{2-2p}\E\Big[\M_{N}(t,\T^2)^p\Big].
\end{align}

\noi
As for the second sum
$A_K^{(2)}$
 in \eqref{moment6}, grouping the terms with the same values of $(k,\ell)$ together, each term within the sum can be written in  the form
\begin{align}
\E\bigg[\prod_{j=1}^n\Big(\int_{\T^2}e^{\be\psi_{N,k_j,\ell_j}(t,y)-\frac{\be^2}2\s_{N}}dy\Big)^{a_j\frac{p}m}\bigg]
\label{BB9}
\end{align}

\noi
for some $n\le m$ and some $a_j\in\{0,...,m-1\}$ such that $\sum_{j=1}^na_j = m$ and 
$(k_j,\ell_j)$, $j=1,\dots,n$,  are all distinct. 
Noting that $\psi_{N,k,\ell}$ are independent and identically distributed, 
it follows from 
 H\"older's inequality with  $a_j\frac{p}m\le m-1$ 
 that 
\begin{align}
\begin{split}
\eqref{BB9} &\leq \prod_{j=1}^n\E\bigg[\Big(\int_{\T^2}e^{\be\psi_{N,k_j,\ell_j}(t,y)-\frac{\be^2}2\s_{N}}dy\Big)^{a_j\frac{p}m}\bigg]\\
&\le\prod_{j=1}^n\E\bigg[\Big(\int_{\T^2}e^{\be\psi_{N,k_j,\ell_j}(t,y)-\frac{\be^2}2\s_{N}}dy\Big)^{m-1}\bigg]^{a_j\frac{p}{m(m-1)}}\\
&=\E\bigg[\Big(\int_{\T^2}e^{\be\psi_{N,k_1,\ell_1}(t,y)-\frac{\be^2}2\s_{N}}dy\Big)^{m-1}\bigg]^{\sum_{j=1}^na_j\frac{p}{m(m-1)}}\\
&= \E\big[\M_N(t,\T^2)^{m-1}\Big]^{\frac{p}{m-1}}.
\end{split}
\label{moment8}
\end{align}

Putting \eqref{moment5}, 
 \eqref{moment5a}, 
  \eqref{moment6},  \eqref{moment7}, and 
\eqref{moment8} together, we obtain 
\begin{align*}
\E\Big[\M_N(t,\T^2)^p\Big] &\le CK^{(p^2-p)\frac{\be^2}{4\pi}}
\cdot K^{2-2p}\E\Big[\M_N(t,\T^2)^p\Big] + C_{K,m}
\E\Big[\M_N(t,\T^2)^{m-1}\Big]^\frac{p}{m-1}.
%&= CK^{p^2\frac{\be^2}{4\pi}-(2+\frac{\be^2}{4\pi})p+2}\E\big[\M_N(t,\T^2)^p\big] + C_{K,m}\E\big[\M_N(t,\T^2)^{m-1}\big]^\frac{p}{m-1}.
\end{align*}
Under the assumption\footnote{Recall that we assume $p > 1$ in view of Proposition \ref{PROP:var}\,(i).} that $1<  p< \frac{8\pi}{\be^2}$, the exponent 
$(p^2-p)\frac{\be^2}{4\pi}
+ 2-2p = \big(\frac{\be^2}{4\pi}p -2\big) (p-1)$
of $K$ in the first term on the right-hand side above is negative.
Hence,  by taking $K\gg 1$ (independent of $N$),  we arrive at the bound:
\begin{align}\label{moment9}
\E\Big[\M_N(t,\T^2)^p\Big] \le C_m\E\Big[\M_N(t,\T^2)^{m-1}\Big]^\frac{p}{m-1}, 
\end{align}

\noi
uniformly in $N \in \N$.

We now conclude the proof of Lemma \ref{LEM:GMCmom}
by  induction on $m\ge 2$ with $m-1<p\le m$. 
When  $m=2$, i.e.~$p\in (1,2]$, the conclusion 
of Lemma \ref{LEM:GMCmom}
follows from \eqref{moment9} and  Proposition~\ref{PROP:var}\,(i). 
Now, given an integer $m \geq 3$, assume that Lemma \ref{LEM:GMCmom} holds for all $1<p\le m-1$. 
Fix $1<  p< \frac{8\pi}{\be^2}$ such that  $m-1<p\le m$. 
Then, from \eqref{moment9} and the inductive hypothesis, we have
\begin{align*}
\sup_{t\in\R,N\in\N}\E\Big[\M_N(t,\T^2)^p\Big] 
\le \sup_{t\in\R,N\in\N}C_m\E\Big[\M_N(t,\T^2)^{m-1}\Big]^\frac{p}{m-1} <\infty.
\end{align*}

\noi
Therefore, by induction, 
we conclude the proof of Lemma~\ref{LEM:GMCmom}.

\begin{ackno}

\rm 
The authors would like to thank Massimiliano Gubinelli 
for  the references \cite{Garban, ADG} and also for helpful discussions.
They are also grateful to Jonathan Bennett
for a helpful conversation
on the Brascamp-Lieb inequality. The authors are also thankful to Vincent Vargas, R\'emi Rhodes,  and Christophe Garban for helpful comments on a previous version of this article, which motivated them to write the follow-up paper~\cite{ORTW}. 
Lastly,  the authors wish to thank the anonymous referees for their helpful remarks which improved the quality of the paper.

T.O.~was supported by the European Research Council (grant no.~637995 ``ProbDynDispEq''
and grant no.~864138 ``SingStochDispDyn"). 
T.R.~was supported by the European Research Council (grant no.~637995 ``ProbDynDispEq'').
Y.W.~was supported by 
supported by 
 the EPSRC New Investigator Award 
 (grant no.~EP/V003178/1).

\end{ackno}

\end{document}